\documentclass[a4paper,english]{jnsao}
\usepackage[utf8]{inputenc}
\usepackage[T1]{fontenc}
\usepackage{csquotes}
\usepackage{babel}
\usepackage{tikz}
\usepackage{float}
\usepackage{algpseudocode}
\usepackage{enumitem}
\usepackage{comment}
\usepackage{mathtools}
\usepackage{subcaption}
\usepackage{multirow,graphicx,rotating,booktabs}
\usepackage[nameinlink,capitalize]{cleveref}
\usepackage{ifdraft}

\newcommand{\term}{\emph}

\newcommand{\field}[1]{\mathbb{#1}}
\newcommand{\N}{\mathbb{N}}

\newcommand{\R}{\field{R}}
\newcommand{\extR}{\overline \R}

\newcommand{\norm}[1]{\|#1\|}
\newcommand{\adaptnorm}[1]{\left\|#1\right\|}

\newcommand{\abs}[1]{|#1|}

\newcommand{\inv}[1]{#1^{-1}}
\newcommand{\grad}{\nabla}
\newcommand{\freevar}{\,\boldsymbol\cdot\,}

\newcommand{\Union}\bigcup
\newcommand{\Isect}\bigcap
\newcommand{\union}\cup
\newcommand{\isect}\cap
\newcommand{\bigunion}\bigcup
\newcommand{\bigisect}\bigcap

\newcommand{\defeq}{:=}

\newcommand{\downto}{\searrow}
\newcommand{\upto}{\nearrow}
\newcommand{\subdiff}{\partial}

\DeclareMathOperator*{\argmin}{arg\,min}

\DeclareMathOperator{\Dom}{dom}

\DeclareMathOperator{\range}{ran}

\def \uminusSym{\setbox0=\hbox{$\cup$}\rlap{\hbox 
        to\wd0{\hss\raise0.5ex\hbox{$\scriptscriptstyle{-}$}\hss}}\box0}
    
\newcommand{\iprod}[2]{\langle #1,#2\rangle}
\newcommand{\adaptiprod}[2]{\left\langle #1,#2\right\rangle}
\def\llangle{\langle\kern-3pt\langle}
\def\rrangle{\rangle\kern-3pt\rangle}

\def \weaktostarSym{\setbox0=\hbox{$\rightharpoonup$}\rlap{\hbox 
        to\wd0{\hss\raise1ex\hbox{$\scriptscriptstyle{*\,}$}\hss}}\box0}
    
\def\linear{\mathbb{L}}

\newcommand{\setto}{\rightrightarrows}

\def\extR{\overline \R}

\def\opt#1{\bar #1}
\def\realopt#1{\hat #1}
\def\this#1{#1^k}
\def\nexxt#1{#1^{k+1}}

\def\optu{{\opt{u}}}
\def\optx{{\opt{x}}}
\def\optv{{\opt{v}}}

\def\opty{{\opt{y}}}

\def\realoptu{{\realopt{u}}}

\def\nextx{\nexxt{x}}

\def\nexty{\nexxt{y}}

\def\thisu{\this{u}}
\def\thisx{\this{x}}

\def\thisz{\this{z}}
\def\thisy{\this{y}}
\def\thisv{\this{v}}

\def\tauTest{\phi}

\def\sigmaTest{\psi}

\def\GenGap{\mathcal{G}}

\newcommand{\Precond}{M}

\def\prev#1{#1^{k-1}}
\def\prevu{\prev{u}}

\def\thisz{\this{z}}

\DeclareMathOperator{\prox}{prox}

\def\d{\,d}

\newcommand{\fakenorm}[1]{\llbracket #1 \rrbracket}

\def\infconv{\mathop{\Box}}

\let\phi=\varphi
\let\epsilon=\varepsilon
\def\alt#1{\tilde #1}

\def\PpredictConstr{\mathcal{X}}
\def\DpredictConstr{\mathcal{Y}}
\def\PDpredictConstr{\mathcal{U}}
\def\InvDisplacements{\mathcal{V}}
\def\Id{\mathop{\mathrm{Id}}}

\DeclareMathOperator{\dynregret}{dynamic\_regret}

\def\primalpredict{\breve x}
\def\dualpredict{\breve y}
\def\pdpredict{\breve u}

\def\paffinecontrol{\pi_k}
\def\daffinecontrol{\tilde{\pi}_k}

\def\ifempty#1{\def\temp{#1}\ifx\temp\empty}

\renewrobustcmd{\downto}{{{\mathchoice%
            {\rotatebox[origin=c]{-20}{$\to$}}
            {\rotatebox[origin=c]{-20}{$\to$}}
            {\rotatebox[origin=c]{-20}{\scalebox{0.75}{$\to$}}}
            {\rotatebox[origin=c]{-20}{\scalebox{0.6}{$\to$}}}
}}}

\renewrobustcmd{\upto}{{{\mathchoice%
            {\rotatebox[origin=c]{20}{$\to$}}
            {\rotatebox[origin=c]{20}{$\to$}}
            {\rotatebox[origin=c]{20}{\scalebox{0.75}{$\to$}}}
            {\rotatebox[origin=c]{20}{\scalebox{0.6}{$\to$}}}
}}}

\usepackage{tikz,pgfplots}
\usepackage{filmstrip}

\def\arxivprebuild{false}

\ifdefined\arxivprebuild
    \usetikzlibrary{external}
    \tikzexternaldisable
\else
    \usepackage{tikzexternal}
    \tikzsetexternalprefix{build/}
    \let\tikzexternaldisable\null
    \let\tikzexternalenable\null
\fi

\tikzifexternalizing{
    \newcommand{\afterpage}[1]{#1}
}{
    \usepackage{afterpage}
}

\usepgfplotslibrary{colorbrewer}

\input{better_predictmacros}

\manuscriptcopyright{}
\manuscriptlicense{}
\manuscripteprinttype{arXiv}
\manuscripteprint{2405.02497}
\manuscriptstatus{Manuscript}
\date{2024-05-03; revised 2024-07-05}

\theoremstyle{definition}
\newtheorem{assumption}[theorem]{Assumption}
\crefname{assumption}{Assumption}{Assumptions}
\floatstyle{ruled}
\floatname{algorithm}{Algorithm}
\newfloat{algorithm}{tbp!}{loa}
\crefname{algorithm}{Algorithm}{Algorithms}

\algrenewcommand{\algorithmiccomment}[1]{\hfill$\rightsquigarrow$\ {\itshape #1}}

\author{
    Neil Dizon\thanks{%
        Department of Mathematics and Statistics, University of Helsinki, Finland,
        \email{neil.dizon@helsinki.fi}, \orcid{0000-0001-8664-2255}
    }
    \and
    Jyrki Jauhiainen\thanks{%
        Department of Technical Physics, University of Eastern Finland,
        \email{jyrki.h.jauhiainen@helsinki.fi}, \orcid{0000-0001-6711-6997}
    }
    \and
    Tuomo Valkonen\thanks{%
        ModeMat, Escuela Politécnica Nacional, Quito, Ecuador
        \text{and}
        Department of Mathematics and Statistics, University of Helsinki, Finland,
        \email{tuomo.valkonen@iki.fi}, \orcid{0000-0001-6683-3572}
    }
}
\shortauthor{Dizon, Jauhiainen, Valkonen}
\title{Prediction techniques for dynamic imaging with online primal-dual methods}
\shorttitle{Prediction techniques for dynamic imaging}

\acknowledgements{%
    This research has been supported by the Academy of Finland grants 338614 and 314701.
}

\begin{document}

\maketitle

\begin{abstract}
    Online optimisation facilitates the solution of dynamic inverse problems, such as image stabilisation, fluid flow monitoring, and dynamic medical imaging.
    In this paper, we improve upon previous work on predictive online primal-dual methods on two fronts.
    Firstly, we provide a more concise analysis that symmetrises previously unsymmetric regret bounds, and relaxes previous restrictive conditions on the dual predictor.
    Secondly, based on the latter, we develop several improved dual predictors.
    We numerically demonstrate their efficacy in image stabilisation and dynamic positron emission tomography.
\end{abstract}

\section{Introduction}
\label{sec:intro}

Many real-world applications involve processing information evolving over time. This includes tasks like computational image stabilisation based on rapid successions of noisy images \cite{tico2009digital,tuomov-predict,zhou2016image}, fluid flow monitoring in industrial processes \cite{tuomov-phaserec,holland2010reducing,hunt2014weighing,lipponen2011nonstationary}, as well as the reconstruction of medical images in the presence of physical motion \cite{bousse2016maximum,burger2017variational,iwao2022brain,natterer2001mathematics}. When the monitoring period is long, and the results are needed immediately, while data is still arriving, it is not feasible to solve one large reconstruction problem after all the data has arrived. Instead, online reconstruction techniques are required.

\term{Online optimisation} extends traditional optimisation by allowing the objective function, parameters, or constraints to change over time, with each iteration of the algorithm. In this paper, we consider the formal problem
\begin{equation}
    \label{eq:pd:problem}
    \min_{(x^0,x^1,x^2\ldots) \in \PpredictConstr}~ \sum_{k=0}^\infty  J_k \defeq F_k(x^k)+E_k(x^k)+G_k(K_k x^k),
\end{equation}
where  $F_k, E_k: X_k \to \extR$, $G_k: Y_k \to \extR$ are convex, proper, and lower semi-continuous on Hilbert spaces $X_k$ and $Y_k$ ($k \in \N$),  $E_k$ is additionally smooth, and $K_k \in \linear(X_k; Y_k)$ is linear and bounded. The set $\PpredictConstr \subset \prod_{k=0}^\infty X_k$ encodes temporal coupling between the variables, and the frame index $k$ represents time-evolution. For example, basing computational image stabilisation on consecutive temporally coupled total variation denoising problems, we arrive at
\[
    \min_{(x^0,x^1,x^2\ldots) \in \PpredictConstr}~   \sum_{k=0}^\infty \frac{1}{2}\norm{\thisx - z_k}^2 + \alpha \norm{D_k\thisx}_{2,1},
\]
where  $z_k$ is the measurement data, and $\alpha$ is a regularisation parameter for isotropic total variation based on the (discretised) differential operator $D_k$.
The set $\PpredictConstr$ models, for example, the optical flow between consecutive image frames.
Similarly, dynamic positron emission tomography (PET) reconstruction affected by patient body motion can be modelled as
\[
    \min_{(x^0,x^1,x^2\ldots) \in \PpredictConstr}~   \sum_{k=0}^\infty  \delta_{\ge 0}(\thisx) + \sum_{i=1}^n \left( [A_k\thisx]_i - [z_k]_i \log([A_k\thisx+c_k]_i)\right) + \alpha \norm{D_k\thisx}_{2,1},
\]
where $A_k$ is the forward model based on a partial Radon transform and $c_k$ is a known vector with non-negative entries.

Early dynamic (consensus-type) optimisation studied series of static problems \cite{kar2010gossip,olfati2007consensus}, assuming sufficient computational resources to solve for each $k$ the individual static problems
\begin{equation}
    \label{eq:intro:static-problem}
    \min_{\thisx \in X_k}~ F_k(x^k)+E_k(x^k)+G_k(K_k x^k),
\end{equation}
before new data arrives.
However, such an approach fails to exploit for temporal super-resolution any physical temporal coupling present between the data for the different frames $k$.
Alternatively, it is possible to solve for each increasing $N$ the finite time window problem
\[
    \min_{x^{0:N} \in \PpredictConstr_{0:N}}~ \sum_{k=0}^N F_k(x^k)+E_k(x^k)+G_k(K_k x^k).
\]
where we use the shorthand slicing notations $x^{0:N} \defeq (x^0,\dots,x^N)$ and $\PpredictConstr_{0:N} \defeq \{ x^{0:N} \mid (x^0,x^1,\ldots) \in \PpredictConstr\}$.
However, for large $N$, these problems become numerically prohibitively expensive. Memory constraints may also compel the exclusion of earlier data, thereby limiting the approach to short  windows of recent data. We will thus adopt an online optimisation approach, focusing on cases where one can only afford one (or at most a few) steps of an optimisation algorithm within a time interval of interest. For modern introductions to online optimisation, we refer to \cite{belmega2018online,hazan2016introduction,orabona2020modern}.

Online optimisation algorithms for dynamic problems are categorised into \emph{structured} and \emph{unstructured} methods \cite{simonetto2020time}.  Structured algorithms take advantage of the temporal nature of the problem to predict and approximate an optimal solution at each time step. On the other hand, unstructured algorithms are agnostic of the temporal nature of the problem and rely only on the optimisation problem that are presented at each time, or a history thereof. Non-predictive primal-dual methods under this category include \cite{bernstein2019online,tang2022running,zhang2021online}. Structured online algorithms can be further subcategorised into \emph{prediction-correction} methods and \term{predictors}. A mere predictor carries out one step of an optimisation algorithm with respect to a predicted objective function but does not perform a corrective step when the new problem becomes available, e.g., see \cite{chang2021online,Nonhoff2020945,zhang2023regrets}. In contrast, prediction-correction methods predict how the optimisation problem changes, and then correct for the errors in predictions once the new objective function is revealed. Primal methods under this category include  \cite{hall13dynamical,simonetto2017prediction,simonetto2016class,zhang2019distributed} based on gradient and mirror descents, among others. The available literature on primal-dual methods within this class is limited to \cite{simonetto2018dual,tuomov-predict}. In particular, the Predictive Online Primal-Dual Proximal Splitting (POPD) method of \cite{tuomov-predict} is amenable to the online solution of \eqref{eq:pd:problem} where $E_k = 0$.

The difficulty with proving something about online methods is that \emph{convergence} results are rarely available. Instead, one attempts to bound the \emph{regret} of past updates with respect to all information available up to an instant $N$. In the dynamic case, following \cite{hall13dynamical}, one bounds the \emph{dynamic regret} defined by
\begin{equation*}
    \dynregret(x^{0:N}) = \sup_{\optx^{0:N} \in \PpredictConstr_{0:N}} \sum_{k=0}^N \left( J_k(x^k) -  J_k(\optx^k)\right).
\end{equation*}
This regret may be negative, if the \term{comparison set} $\PpredictConstr_{0:N}$ constrains the \term{comparison sequence} $\optx^{0:N}$ more than the predictions and updates constrain the iterates.
Alternatively, performance evaluation based on \emph{asymptotical tracking errors} is available in the literature but its discussion is not included here. For details, interested reader is referred to, e.g., \cite{simonetto2020time}.
For the POPD of \cite{tuomov-predict}, the regret estimate is even weaker: for a function $\breve J_{0:N}$ dependent on both the comparison set $\PpredictConstr_{0:N}$ and the original objectives $J_0,\ldots,J_N$, we only have a bound on
\begin{equation}
    \label{eq:pd:olddynamicregret}
    \sup_{\optx^{0:N} \in \PpredictConstr_{0:N}}\left( \breve J_{0:N}(x^{0:N}) - \sum_{k=0}^N J_k(\optx^k)\right).
\end{equation}
That is, the regret estimate is non-symmetric between the algorithmic iterates and the comparison sequences.

Due to proof-technical reasons, the dual predictor of the POPD in \cite{tuomov-predict} is also severely constrained to a specific proximal form.
In \cref{sec:pd} of this paper, through improved and much simplified proofs, we (a) remove this restrictions, (b) provide improved, symmetric, regret estimates and, (c) extend the POPD with an additional forward step with respect to $E_k$ (which was zero in \cite{tuomov-predict}). Specifically, our new form of dynamic regret for primal-dual methods symmetrises \eqref{eq:pd:olddynamicregret} by replacing both $\breve{J_k}$ and $J_k$ with a temporal “sub-infimal” convolution $\mathring J_k$ between the dual comparison set and the objective. The bound for this modified dynamic regret depends on the richness of the comparison set and the accuracy of the predictors.
Given the relaxed conditions on the dual predictor, in \cref{sec:predictors}, we analyse the accuracy of a broad class of “pseudo-affine” primal-dual predictors.
We then present various examples of pseudo-affine predictors in the context of optical flow, and how they preserve salient relationships between the primal and dual variables. Finally in \cref{sec:numerics}, we evaluate the proposed method and predictors numerically on image stabilisation and dynamic PET reconstruction.

\subsection*{Notation}

We write $x^{n:m} \defeq (x^n,\ldots,x^m)$ with $n \le m$, and $x^{n:\infty} \defeq (x^n,x^{n+1},\ldots)$. We slice a set $\PpredictConstr \subset \prod_{k=0}^\infty X_k$ as $\PpredictConstr_{n:m} \defeq \{x^{n:m} \mid x^{0:\infty} \in \PpredictConstr\}$ and $\PpredictConstr_n \defeq \PpredictConstr_{n:n}$.
We write $\linear(X; Y)$ for the set of bounded linear operators between (Hilbert) spaces $X$ and $Y$, and $\Id \in \linear(X; X)$ for the identity operator.
For brevity, we write $\iprod{x}{y}_M \defeq \iprod{Mx}{y}$, and $\fakenorm{x}_M \defeq \sqrt{\iprod{x}{x}}_M$ for $M \in \linear(X; X)$.
We write $M \ge 0$ if $M$ is positive semi-definite and $M \simeq N$ if $\iprod{Mx}{x} = \iprod{Nx}{x}$ for all $x$.
When $M$ is positive semi-definite, we use the usual norm notation $\norm{x}_M \defeq \sqrt{\iprod{x}{x}_M}$.

For any $A \subset X$ and $x \in X$ we set
$
    \iprod{A}{x} \defeq \{\iprod{z}{x} \mid z \in A\}.
$
We write $\delta_A$ for the $\{0,\infty\}$-valued indicator function of $A$.
For any $B \subset \R$ (in particular $B=\iprod{A}{x}$), we use the notation $B \ge 0$ to mean that $t \ge 0$ for all $t \in B$.

For $F: X \to (-\infty, \infty]$, the effective domain $\Dom F \defeq \{ x \in X \mid F(x) < \infty\}$.
With $\extR \defeq [-\infty, \infty]$ the set of extended reals, we call $F: X \to \extR$ \term{proper} if $F>-\infty$ and $\Dom F \ne \emptyset$.
Let then $F$ be convex. We write $\subdiff F(x)$ for the subdifferential at $x$ and (for additionally proper and lower semicontinuous $F$)
\[
    \prox_{F}(x) \defeq \argmin_{\alt x \in X}~ F(\alt x) + \frac{1}{2}\norm{\alt x-x}^2
    = \inv{(\Id + \subdiff F)}(x)
\]
for the proximal map. We call $F$ strongly subdifferentiable at $x$ with the factor $\gamma>0$ if
\[
    F(\alt x)-F(x) \ge \iprod{z}{\alt x-x} + \frac{\gamma}{2}\norm{\alt x-x}^2
    \quad \text{for all}\quad z \in \subdiff F(x) \text{ and } \alt x \in X.
\]
In Hilbert spaces, this is equivalent to strong convexity with the same factor. Finally, for $f \in L^q(\Omega; \R^n)$, we write
$
    \norm{f}_{p,q} \defeq \adaptnorm{\xi \mapsto \norm{f(\xi)}_p}_{L^q(\Omega)}.
$

\section{An online primal-dual method}
\label{sec:pd}

\begin{algorithm}
    \caption{New predictive online primal-dual proximal splitting (POPD$_2$)}
    \label{alg:pd:alg}
    \begin{algorithmic}[1]
        \Require For all $k \in \N$, on Hilbert spaces $X_k$ and $Y_k$, convex, proper, lower semi-continuous $F_{k+1}, E_{k+1}: X_{k+1} \to \extR$ and $G_{k+1}^* : Y_{k+1} \to \extR$, primal-dual predictors $P_k: X_k \times Y_k \to X_{k+1} \times Y_{k+1}$, and $K_{k+1} \in \linear(X_{k+1}; Y_{k+1})$.
        Step length parameters $\tau_{k+1},\sigma_{k+1}>0$.
        \State Pick initial iterates $x^0 \in X_0$ and $y^0 \in Y_0$.
        \For{$k \in \N$}
        \State\label{item:alg:pd:pdpredict}$(\nexxt\primalpredict, \nexxt\dualpredict) \defeq P_k(\thisx, \thisy)$.
        \Comment{prediction step}
        \State\label{item:alg:pd:primal}$\nextx \defeq \prox_{\tau_{k+1} F_{k+1}}(\nexxt{\primalpredict} - \tau_{k+1}\nabla E_{k+1}(\nexxt{\primalpredict}) - \tau_{k+1} K_{k+1}^*\nexxt{\dualpredict})$
        \Comment{primal step}
        \State\label{item:alg:pd:dual} $\nexty \defeq \prox_{\sigma_{k+1} G_{k+1}^*}(\nexxt{\dualpredict} + \sigma_{k+1} K_{k+1}(2\nextx-\nexxt{\primalpredict}))$
        \Comment{dual step}
        \EndFor
    \end{algorithmic}
\end{algorithm}

In this section we present and analyse our proposed POPD$_2$ primal-dual method for the online solution of \eqref{eq:pd:problem}. Presented in \cref{alg:pd:alg}, the method incorporates an additional forward step with respect to $E_k$, and simplifies the dual prediction of the POPD of \cite{tuomov-predict}.
Each step $k$ of the algorithm corresponds to a single data frame, with the time-varying data embedded in the functions and operators ($F_k$, $E_k$, $K_k$, $G_k$, and $P_k$).\footnote{It is, of course, possible for each of the functions to depend on multiple real data frames simply by treating them as a single frame in the algorithm, and correspondingly taking the spaces $X_k$ and $Y_k$ larger. They can even grow and shrink with $k$.  These possibilities are exploited \cite[Section 5.2]{tuomov-predict}.}
The primal and dual steps (\cref{item:alg:pd:primal,item:alg:pd:dual}) are analogous to the standard PDPS of Chambolle and Pock \cite{chambolle2010first} with an additional forward step with respect to $E_k$.
These optimisation steps are preceded by a prediction step (\cref{item:alg:pd:pdpredict}), implemented by the operators $P_k: X_k \times Y_k \to X_{k+1} \times Y_{k+1}$ that transfer iterates from one step to the next.

In \cref{sec:pd:assumptions} we outline our assumptions. We then present in \cref{sec:pd:dynamicregret} a symmetric dynamic regret bound, as discussed in the introduction.

\subsection{Assumptions and definitions}
\label{sec:pd:assumptions}

To develop the regret theory, we work with the testing approach to convergence proofs, presented in \cite{tuomov-proxtest,tuomov-firstorder,clasonvalkonen2020nonsmooth}. This depends on encoding convergence rates into distinct testing parameters for the primal and dual variables. With the general notation  $u = (x,y)$, $\thisu=(\thisx,\thisy)$, etc., we work with the following assumptions.

\begin{assumption}
    \label{ass:pd:main}
    For all $k \ge 1$, on Hilbert spaces $X_k$ and $Y_k$, we are given:
    \begin{enumerate}[label=(\roman*),nosep]
        \item Convex, proper, and lower semicontinuous $F_k, E_k: X_k \to \extR$,  $G_k^*: Y_k \to \extR$, as well as $K_k  \in \mathbb{L}(X_k; Y_k)$, such that $\nabla E_k$ exists and is $L_k$-Lipschitz. We write $Q_k \defeq F_k + E_k$, and $\gamma_{F_k}, \gamma_{E_k}, \rho_k \geq 0$ for the factors of (strong) convexity of $F_k, E_k$ and $G_k^*$, respectively. For some $\kappa_k \in (0,1]$ we have
        \begin{equation}
            \label{eq:fe_gamma}
            0 \le \gamma_k \defeq
            \begin{cases}
                    \gamma_{F_k}+\gamma_{E_k} - \kappa_k L_k, & \gamma_{E_k}>0, \\
                    \gamma_{F_k}, & \gamma_{E_k}=0.
            \end{cases}
        \end{equation}

        \item Primal and dual step length parameters $\tau_k,\sigma_k>0$ and testing parameters $\eta_k,\tauTest_k,\sigmaTest_k > 0$ satisfying
        \begin{subequations}%
            \label{eq:pd:stepconds1}%
            \begin{align}%
                \label{eq:pd:symcond}
                \eta_k & = \tauTest_k\tau_k = \sigmaTest_k \sigma_k,
                && \text{(primal-dual coupling)}
                \\
                \label{eq:pd:primaltestcond-positivity}
                1    & \geq {
                    \tau_k  \kappa_k^{-1} L_k+ \tau_k\sigma_k\norm{K_k}^2}
                && \text{(metric positivity)}.
            \end{align}%
        \end{subequations}

        \item  \label{item:pd:main-online-predictors} Primal-dual predictors $P_k: X_k \times Y_k\to X_{k+1}\times Y_{k+1}$ giving the predictions $(\nexxt\primalpredict, \nexxt\dualpredict) \defeq P_k(\thisx, \thisy)$.

        \item A bounded set $\PDpredictConstr \subset \prod_{k=0}^\infty X_k \times Y_k$ of primal-dual comparison sequences with which we define the set of primal and dual comparison sequences as
        \[
            \begin{aligned}
            \PpredictConstr & \defeq \left\lbrace \optx^{0:\infty} \in {\textstyle \prod\nolimits_{k=0}^\infty X_k} \mid (\bar x^{0:\infty},\bar y^{0:\infty}) \in \PDpredictConstr\right\rbrace
            \quad
            \text{and}
            \\
            \DpredictConstr & \defeq \left\lbrace \bar y^{0:\infty} \in {\textstyle \prod\nolimits_{k=0}^\infty Y_k} \mid  (\bar x^{0:\infty},\bar y^{0:\infty}) \in \PDpredictConstr\right\rbrace.
            \end{aligned}
        \]
    \end{enumerate}
\end{assumption}

\begin{example}
    \label{ex:pd:noaccel}
    Similarly to the standard PDPS, as analysed in \cite{tuomov-proxtest,clasonvalkonen2020nonsmooth}, for an unaccelerated method, we can choose the step length and testing parameters as $\tau_k \equiv \tau$ and $\sigma_k \equiv \sigma$ for some $\tau, \sigma > 0$, along with $\eta_k \equiv \tau$, $\tauTest_k \equiv 1$ and $\sigmaTest_k \equiv \frac{\tau}{\sigma}$.
    For an accelerated method, more elaborate choices are needed.
\end{example}

For all $k \ge 1$, we define $\Precond_k , \Gamma_k,  \Omega_{k}  \in \linear(X_k \times Y_k; X_k \times Y_k)$ by
\begin{equation*}
           \Precond_k
                \defeq
                \begin{pmatrix}
                    \inv\tau_k \Id & - K_k^* \\
                    -K_k & \inv\sigma_k \Id
       			\end{pmatrix},\,
            \Gamma_k
             \defeq
            \eta_k
            \begin{pmatrix}
                \gamma_k  \Id & 2K_k^* \\
                -2 K_k & \rho_k \Id
            \end{pmatrix},
            \text{ and }
            \Omega_{k}
            \defeq
            \begin{bmatrix}
                \kappa_k^{-1}L_k & 0\\0&0.
            \end{bmatrix}.
\end{equation*}
We further define for all $k \in \N$ the monotone operator $H_k: X_k \times Y_k \setto X_k \times Y_k$ as
\begin{equation}
    H_k(u) \defeq \begin{pmatrix}
        \subdiff F_k(x) + \nabla E_k(x) + K_k^* y\\
        \subdiff G_k^*(y) - K_k x
    \end{pmatrix}.
\end{equation}
Then $0 \in H_k(\this{\realoptu})$ encodes the first order necessary and sufficient optimality conditions for the static problem \eqref{eq:intro:static-problem} for frame $k$.
Moreover, writing $\thisu \defeq (\thisx,\thisy)$ and $(\nexxt\primalpredict, \nexxt\dualpredict) \defeq P_k(\thisx, \thisy)$, \cref{alg:pd:alg} reads \cite{tuomov-proxtest,he2012convergence} in implicit form
\begin{equation}
    \label{eq:ppext-pdps}
    0 \in \tilde H_k(\thisu) + M_k(\thisu-\this\pdpredict)
    \quad\text{for all}\quad k \ge 1,
\end{equation}
where $\tilde H_k: X_k \times Y_k \setto X_k \times Y_k$ is defined by a modification of $H_k$ as
\begin{equation}
    \label{eq:pd:m}
    \tilde H_k(u) \defeq \begin{pmatrix}
        \subdiff F_k(x) + \nabla E_k(\this{\primalpredict}) + K_k^* y\\
        \subdiff G_k^*(y) - K_k x
    \end{pmatrix}.
\end{equation}
(Alternatively, to avoid introducing $\tilde H_k$, we could replace $M_k$ by a Bregman divergence \cite{tuomov-firstorder}.)

For brevity, with $u^{0:N}=(u^0,\dots,u^N)$, we also write
\begin{align*}
    H_{0:N}(u^{0:N}) & \defeq  H_0(u^0)\times\dots\times H_N(u^N),
    &
    G_{1:N}(y^{1:N}) & \defeq \sum_{k=1}^{N} \eta_k G_{k}(\inv\eta_k \thisy),
    \\
    K_{1:N}x^{1:N} & \defeq (\eta_1 K_{1} x^1, \ldots, \eta_{N} K_N x^N),
    \quad\text{and}
    &
    Q_{1:N}(x^{1:N}) & \defeq \sum_{k=1}^{N} \eta_k [F_{k}  + E_{k}](\thisx).
\end{align*}
In the setting of \cref{ex:pd:noaccel}, $\eta_k \equiv \tau$, so all the functions are simply scaled by the constant primal step length $\tau$.
Observe that $G_{1:N}^*(y^{1:N})=\sum_{k=1}^{N} \eta_k G_{k}^*(\nexty)$ and
\[
    [Q_{1:N}+G_{1:N} \circ K_{1:N}](x^{1:N})
    =\sum_{k=1}^{N} \eta_k[Q_{k} + G_{k} \circ K_{k}](\nextx).
\]

Finally, for each $k$, we define the \emph{Lagrangian duality gap} by
\[
    \begin{aligned}[b]
        \GenGap_{k}^H(\thisu,\this\optu)
        &
        \defeq
    \eta_k \bigl[ F_{k}(\thisx)+E_{k}(\thisx) + \iprod{K_{k}\thisx}{\this\opty} - G_{k}^*(\this\opty) \bigr]
        \\
        \MoveEqLeft[-1]
        -\eta_k\bigl[F_{k}(\this\optx) + E_{k}(\this\optx) + \iprod{K_{k}^*\thisy}{\this\optx} - G_{k}^*(\thisy)
        \bigr].
    \end{aligned}
\]
This is non-negative if $0 \in H_k(\this\optu)$; see, e.g., \cite{clasonvalkonen2020nonsmooth}.

\subsection{A general regret estimate}
\label{sec:pd:dynamicregret}

As we need to develop a dynamic regret theory for \cref{alg:pd:alg}, we first revisit relevant tools to derive meaningful measures of regret.  We first recall the following \term{smoothness three-point inequalities} (on a Hilbert space $X$).

\begin{lemma}[{\cite[Appendix B]{tuomov-proxtest} or \cite[Chapter 7]{clasonvalkonen2020nonsmooth}}]
    \label{lemma:pd:Esmoothness}
    Suppose $E: X \to \extR$ is convex, proper, and lower semicontinuous, and has $L$-Lipschitz gradient. Then
    \begin{gather}
        \label{eq:three-point-smoothness}
        \iprod{\grad E(z)}{x-\optx}
        \ge
        E(x)-E(\optx) -  \frac{L}{2}\norm{x-z}^2
        \quad (\optx, z, x \in X).
    \end{gather}
    If $E$ is, moreover, $\gamma_{E}$-strongly convex, then for any $\beta>0$ and $\optx, z, x \in X$, also
    \begin{equation}
        \label{eq:three-point-smoothness-sc}
            \iprod{\grad E(z)}{x-\optx}
            \ge
            E(x)-E(\optx) + \frac{\gamma_{E}-\beta L^2}{2}\norm{x-\optx}^2
            -\frac{1}{2\beta}\norm{x-z}^2.
    \end{equation}
\end{lemma}

\begin{corollary}
    \label{corollary:pd:condition}
    Let \cref{ass:pd:main} hold. Then, for any $k \in \N$,
    \[
        \iprod{\subdiff F_{k}(\thisx)+\grad E_{k}(\this\primalpredict)}{\thisx-\this\optx}
        \ge
        Q_k(\thisx)-Q_k(\this\optx)
        +\frac{\gamma_{k}}{2}\norm{\thisx-\this\optx}^2
        - \frac{L_k}{2\kappa_k }\norm{\thisx-\this\primalpredict}^2.
    \]
\end{corollary}

\begin{proof}
    If $\gamma_{E_{k}}=0$, \eqref{eq:three-point-smoothness} of \cref{lemma:pd:Esmoothness} with the (strong) convexity of $F_{k}$ yields
    \[
        \iprod{\subdiff F_{k}(\thisx)+\grad E_{k}(\this\primalpredict)}{\thisx-\this\optx}
        \ge Q_{k}(\thisx)-Q_{k}(\this\optx)
        +\frac{\gamma_{F_{k}}}{2}\norm{\thisx-\this\optx}^2
        -\frac{L_{k}}{2}\norm{\thisx-\this\primalpredict}^2.
    \]
    If $\gamma_{E_{k}} > 0$, \eqref{eq:three-point-smoothness-sc} of \cref{lemma:pd:Esmoothness}  with $\beta_k=\kappa_k L_k^{-1}$ and the (strong) convexity of $F_{k}$ yield the estimate
    \[
        \begin{aligned}
        \iprod{\subdiff F_{k}(\thisx)+\grad E_{k}(\this\primalpredict)}{\thisx-\this\optx}
        &
        \ge Q_{k}(\thisx)-Q_{k}(\this\optx) +\frac{\gamma_{F_{k}}+\gamma_{E_{k}}-\kappa_kL_k}{2}\norm{\thisx-\this\optx}^2
        -\frac{L_k}{2\kappa_k}\norm{\thisx-\this\primalpredict}^2.
        \end{aligned}
    \]
    In both cases, the claim follows after an application of \eqref{eq:fe_gamma}.
\end{proof}

Using the preceding lemma, the next result bounds the cumulative sum of Lagrangian duality gaps for the iterates of \cref{alg:pd:alg}. This bound is instrumental in proving our main dynamic regret bound, to follow.
The topic of the next \cref{sec:predictors} is to bound the prediction error $e_N$ present in the result.

\begin{lemma}
    \label{thm:pd:main}
    Let \cref{ass:pd:main} hold for $u^{1:N}$ generated by \cref{alg:pd:alg} for an initial $u^0 \in X_0 \times Y_0$. Then $M_k$, $\eta_k\Precond_k+\Gamma_k$ and $\eta_k (M_k -\Omega_k)$ are positive semi-definite, and
    \begin{multline*}
        \frac{1}{2}\norm{u^N-{\optu^N}}_{\eta_{N} \Precond_{N}+\Gamma_{N}}^2
        +
        \sum_{k=1}^{N}\left( \GenGap^H_{k}(\thisu,\this\optu) + \frac{1}{2}\norm{\thisu-\this \pdpredict}_{\eta_{k}(\Precond_{k} - \Omega_{k})}^2\right)
        \le  \frac{1}{2}\norm{u^0-\optu^0}^2_{\eta_0\Precond_0+\Gamma_0} + e_N(u^{0:N-1},\optu^{0:N}),
    \end{multline*}
    for the \term{prediction error}
    \begin{equation}
        \label{eq:pd:estimate-rhs}
        e_N(u^{0:N-1},\optu^{0:N}) \defeq
         \sum_{k=1}^N \bigg(\frac{1}{2}\norm{\this\pdpredict-\this{\optu}}_{\eta_{k}\Precond_{k}}^2 - \frac{1}{2}\norm{\prevu-\prev{\optu}}_{\eta_{k-1}\Precond_{k-1}+\Gamma_{k-1}}^2\bigg).
    \end{equation}
\end{lemma}

\begin{proof}
    For brevity, and to not abuse norm notation when $\Gamma_k$ is not positive semi-definite, we write $\fakenorm{x}_{\Gamma_k}^2 \defeq \iprod{x}{x}_{\Gamma_k}$. By Young's inequality, we have
    \[
        \eta_k M_k - \eta_k \Omega_k
        =
        {\eta_k}{}
        \begin{pmatrix}
            (\inv\tau_k - \kappa_k^{-1}L_k)  \Id & - K_k^* \\
            - K_k & \inv\sigma_k \Id
        \end{pmatrix}
        \ge
        \phi_k
        \begin{pmatrix}
            \Id - \tau_k \kappa_k^{-1}L_k - \tau_k\sigma_k K_k^*K_k & 0 \\
            0 & 0
        \end{pmatrix}\\
    \]
    and
    \[
        \eta_{k}\Precond_{k}+\Gamma_{k}
        \simeq
        \begin{pmatrix}
            \tauTest_k(1+\gamma_{k}\tau_k) \Id & -\eta_k K_k^*\\
            -\eta_k K_k & \sigmaTest_k(1+\rho_{k}\sigma_k) \Id
        \end{pmatrix}
        \ge
        \begin{pmatrix}
            \tauTest_k(1+\gamma_k\tau_k) \Id - \frac{\eta_k^2}{\sigmaTest_{k}(1+\rho_k\sigma_k)} K_k^*K_k & 0 \\
            0 & 0
        \end{pmatrix}.
    \]
    Thus, \eqref{eq:pd:primaltestcond-positivity} establishes the positive semi-definiteness claims.

    We then expand
    \[
        \begin{aligned}[t]
            \eta_k\iprod{\tilde H_k(\thisu)}{\thisu-\this{\optu}}
            &
            = \eta_k\iprod{\subdiff{F_k(\thisx)}}{\thisx - \this\optx} + \eta_k \iprod{\nabla E_k(\this\primalpredict)}{\thisx - \this\optx}
            \\
            \MoveEqLeft[-1]
            + \eta_k\iprod{\subdiff G_k^*(\thisy)}{\thisy-\this\opty}
            \\
            \MoveEqLeft[-1]
            + \eta_k\iprod{K_k^*\thisy}{\thisx-\this\optx}
            - \eta_k\iprod{K_k\thisx}{\thisy - \this\opty}.
        \end{aligned}
    \]
    Thus the (strong) convexity of $F_k$ and $G^{\ast}_k$ together with \cref{corollary:pd:condition} yields
    \begin{gather}
        \label{eq:pdps:condition}
        \begin{aligned}[b]
            \eta_k\iprod{\tilde H_k(\thisu)}{\thisu-\this{\optu}}
            &
             \ge \eta_k\left(F_k(\thisx) - F_k(\this\optx) + \frac{\gamma_k}{2}\|\thisx - \this\optx\|^2\right)
            \\
            \MoveEqLeft[-1]
            + \eta_k\left(E_k(\thisx) - E_k(\this\optx) - \frac{L_k}{2\kappa_k}\|\this\primalpredict - \thisx\|^2 \right)
            \\
            \MoveEqLeft[-1]
            + \eta_k\left(G_k^* (\thisy) - G_k^*(\this\opty) + \frac{\rho_k}{2}\|\thisy-\this\opty\|^2\right)
            \\
            \MoveEqLeft[-1]
            - \eta_k\iprod{K_k^*\thisy}{\this\optx} + \eta_k \iprod{K_k\thisx}{\this\opty}
            \\
            &
            = \frac{1}{2}\fakenorm{\thisu-\this\optu}_{\Gamma_k}^2 + \GenGap^H_k(\thisu,\this\optu)  - \frac{1}{2}\|\thisu-\this\pdpredict\|_{\eta_k\Omega_k}.
        \end{aligned}
    \end{gather}
    Following the testing methodology \cite{clasonvalkonen2020nonsmooth,tuomov-proxtest}, we apply the linear “testing operator” $\iprod{\freevar}{\thisu-\this{\optu}}_{\eta_k}$ to both sides of \eqref{eq:ppext-pdps}. This followed by \eqref{eq:pdps:condition} yields
    \[
        0 \ge
        \iprod{\thisu-\this\pdpredict}{\thisu-\this{\optu}}_{\eta_k\Precond_k}
        + \frac{1}{2}\fakenorm{\thisu-\this\optu}_{\Gamma_k}^2
        + \GenGap^H_k(\thisu,\this\optu)
        - \frac{1}{2}\|\thisu-\this\pdpredict\|_{\eta_k\Omega_k}
        \quad (k=1,\ldots,N).
    \]
    Pythagoras' identity for the inner product and norm with respect to the operator $\eta_k\Precond_k$ now yields
    \[
        \frac{1}{2}\norm{\this\pdpredict-\this{\optu}}_{\eta_k \Precond_k}^2
        \ge
        \frac{1}{2}\norm{\thisu-\this{\optu}}_{\eta_k \Precond_k+\Gamma_k}^2
        + \GenGap^H_k(\thisu,\this\optu)
        + \frac{1}{2}\norm{\thisu-\this\pdpredict}_{\eta_k(\Precond_k - \Omega_k)}^2
        \quad (k=1,\ldots,N).
    \]
    Summing this over $k\in \lbrace 1,\dots, N \rbrace$, we obtain the claim.
\end{proof}

We are now almost ready to state our main result regarding the dynamic regret of our algorithm.
To proceed, we define the function $\mathring G_{1:N}$ by
\[
    \mathring G_{1:N}(z^{1:N})
    \defeq
    \sup_{\tilde y^{1:N} \in  \DpredictConstr_{1:N}}
    \bigl[
        \iprod{z^{1:N}}{\tilde y^{1:N}} - G_{1:N}^*(\tilde y^{1:N})
    \bigr].
\]
If the dual comparison sets $\DpredictConstr_{1:N}$ were convex, then, recalling the formula $(f_1 + f_2)^* = f_1^* \infconv f_2^*$ for infimal convolutions (denoted $\infconv$) of convex functions $f_1$ and $f_2$, we would have
$
    \mathring G_{1:N} = G_{1:N} \infconv \delta_{\DpredictConstr_{1:N}}^*.
$
In general,
$\mathring G_{1:N} \le G_{1:N} \infconv \delta_{\DpredictConstr_{1:N}}^*$.

It is worth noting that under suitable assumption on the comparison sequence, $\breve G = \mathring G$; see \cite[Example~3.4]{tuomov-predict}. Moreover, if $\DpredictConstr_{1:N}=\prod_{k=1}^N Y_k$, or even just $\DpredictConstr_{1:N} \supset \Dom G_{1:N}^*$, then it is also clear that $\mathring G_{1:N} = G_{1:N}$.
In this case the next theorem provides a dynamic regret bound with respect to the original static objective $Q_{1:N} + G_{1:N} \circ K_{1:N}$ for the first $N$ frames. Otherwise, providing bounds on $Q_{1:N} + \mathring G_{1:N} \circ K_{1:N}$, it modifies the objective by “sub-infimally” convolving $G_{1:N}$ with the temporal evolution constraints presented by the set of dual comparison sequences $\DpredictConstr_{1:N}$. Typically $G_{1:N} \circ K_{1:N}$ would be a sum of independent static regulariser for the temporal frames $1$ to $N$, for example, a sum of independent total variation terms for each frame. Then $\mathring G_{1:N}$ would be a temporally convolved total variation regulariser.

Compared to the “$G$-banana” $\breve G_{1:N}$ of \cite{tuomov-predict}, the “$G$-doughnut” $\mathring G_{1:N}$ has a significantly simpler structure.
Moreover, the following new result is symmetric, whereas the previous results of \cite{tuomov-predict} were unsymmetric, employing $\breve G_{1:N}$ for the iterates $x^{1:N}$, and the original $G_{1:N}$ for the comparison sequences $\optx^{1:N}$. Generally, for the right hand side of the next main regret estimate to be meaningful, it is necessary for the set of comparison sequences to be bounded.

\begin{theorem}
    \label{thm:pd:regret_estimate}
    Let $N \ge 1$, and suppose \cref{ass:pd:main} hold for $u^{1:N}$ generated by \cref{alg:pd:alg} for an initial $u^0 \in X_0 \times Y_0$.
    Then
    \begin{multline}
        \label{eq:pd:regret_estimate}
        [Q_{1:N}(x^{1:N}) + \mathring G_{1:N}(K_{1:N}x^{1:N})]
        - \sup_{\optx^{1:N} \in  \PpredictConstr_{1:N}}
            [Q_{1:N}(\optx^{1:N}) + \mathring G_{1:N}(K_{1:N}\optx^{1:N})]
        \\
        \le
        \sup_{\optu^{0:N}\in  \PDpredictConstr_{0:N}} \bigg( \frac{1}{2}\norm{u^0-\optu^0}^2_{\eta_0\Precond_0+\Gamma_0} + c_N(\optx^{1:N}, y^{1:N}) + e_N(u^{0:N-1},\optu^{0:N}) \bigg),
    \end{multline}
    where the prediction error $e_N(u^{0:N-1},\optu^{0:N})$ is given by \eqref{eq:pd:estimate-rhs}, and the \term{comparison set solution discrepancy}
    \[
        c_N(\optx^{1:N}, y^{1:N}) \defeq  \inf_{\tilde y^{1:N} \in \DpredictConstr_{1:N}} \iprod{K_{1:N}\optx^{1:N}}{y^{1:N}-\tilde y^{1:N}} + G_{1:N}^*(\tilde y^{1:N}) - G_{1:N}^*(y^{1:N}).
    \]
\end{theorem}

\begin{proof}
    Let $\optu^{1:N} \in \PDpredictConstr_{1:N}$.
    For any $\tilde y^{1:N} \in \DpredictConstr_{1:N}$, by the definition of $\mathring G_{1:N}$ as a Fenchel conjugate of $G_{1:N}^* + \delta_{\DpredictConstr_{1:N}}$, we have
    \begin{equation}
        \label{eq:pd:regret_estimate:cn:0}
        \begin{aligned}[t]
        \iprod{K_{1:N}\optx^{1:N}}{y^{1:N}} - G_{1:N}^*(y^{1:N})
        &
        =
        \inf_{\tilde y^{1:N} \in \DpredictConstr_{1:N}}\left(
        \iprod{K_{1:N}\optx^{1:N}}{y^{1:N}-\tilde y^{1:N}} - G_{1:N}^*(y^{1:N})
        + G_{1:N}^*(\tilde y^{1:N})
        \right)
        \\
        \MoveEqLeft[-1]
        + \mathring G_{1:N}(K_{1:N}\optx^{1:N})
        \\
        &
        =
        c_N(\optx^{1:N}, y^{1:N})
        + \mathring G_{1:N}(K_{1:N}\optx^{1:N}).
        \end{aligned}
    \end{equation}

    With this, we rearrange
    \begin{equation*}
        \begin{aligned}[t]
        \sum_{k=1}^{N} \GenGap_{k}^H(\thisu,\this\optu)
        &
        =
        \sum_{k=1}^{N} \eta_k \Bigl(\bigl[F_{k}(\thisx)+E_{k}(\thisx) + \iprod{K_{k}\thisx}{\this\opty} - G_{k}^*(\this\opty)\bigr]
        \\
        \MoveEqLeft[-5]
        - \bigl[F_{k}(\this\optx) + E_{k}(\this\optx) + \iprod{K_{k}^*\thisy}{\this\optx} - G_{k}^*(\thisy))
        \bigr]\Bigr)
        \\
        &
        =
            \bigl[
                Q_{1:N}(x^{1:N}) + \iprod{K_{1:N}x^{1:N}}{\opty^{1:N}} - G_{1:N}^*(\opty^{1:N})
            \bigr]
        \\
        \MoveEqLeft[-5]
            -
            \bigl[
                Q_{1:N}(\optx^{1:N}) + \iprod{K_{1:N}\optx^{1:N}}{y^{1:N}} - G_{1:N}^*(y^{1:N})
            \bigr]
        \\
        &
        =
            \bigl[
                Q_{1:N}(x^{1:N}) + \iprod{K_{1:N}x^{1:N}}{\opty^{1:N}} - G_{1:N}^*(\opty^{1:N})
            \bigr]
        \\
        \MoveEqLeft[-5]
            -
            \bigl[
                Q_{1:N}(\optx^{1:N}) +  c_N(\optx^{1:N}, y^{1:N})
                + \mathring G_{1:N}(K_{1:N}\optx^{1:N})
            \bigr].
        \end{aligned}
    \end{equation*}
    Using this equation in the claim of \cref{thm:pd:main} readily establishes that
    \begin{equation}
        \label{eq:pd:regret_Gdoughnut:1}
        D_N(\optu^{1:N})
        \le
        \frac{1}{2}\norm{u^0-\optu^0}^2_{\eta_0\Precond_0+\Gamma_0}  +c_N(\optx^{1:N}, y^{1:N}) + e_N(u^{0:N-1},\optu^{0:N}).
    \end{equation}
    for
    \[
        D_N(\optu^{1:N}) \defeq
        Q_{1:N}(x^{1:N}) + \iprod{K_{1:N}x^{1:N}}{\opty^{1:N}} - G_{1:N}^*(\opty^{1:N})
        - Q_{1:N}(\optx^{1:N}) - \mathring G_{1:N}(K_{1:N}\optx^{1:N}).
    \]
    We have
    \[
        \begin{aligned}[b]
            \sup_{\optu^{0:N} \in  \PDpredictConstr_{0:N}} D_N(\optu^{1:N})
            &
            =
            \sup_{\optu^{1:N} \in  \PDpredictConstr_{1:N}} D_N(\optu^{1:N})
            \\
            &
            \ge
            \sup_{\optu^{1:N} \in  \PDpredictConstr_{1:N}}  \Bigl(
            Q_{1:N}(x^{1:N}) + \iprod{K_{1:N}x^{1:N}}{\opty^{1:N}} - G_{1:N}^*(\opty^{1:N})
            \Bigr)
            \\
            \MoveEqLeft[-1]
            -
            \sup_{\optu^{1:N} \in  \PDpredictConstr_{1:N}}  \Bigl(
                Q_{1:N}(\optx^{1:N}) + \mathring G_{1:N}(K_{1:N}\optx^{1:N})
            \Bigr)
            \\
            &
            =
            Q_{1:N}(x^{1:N}) + \mathring G_{1:N}(K_{1:N}x^{1:N})
            -
            \sup_{\optx^{1:N} \in  \PpredictConstr_{1:N}}  \Bigl(
                Q_{1:N}(\optx^{1:N}) + \mathring G_{1:N}(K_{1:N}\optx^{1:N})
            \Bigr).
        \end{aligned}
    \]
    Therefore, the claim follows by taking the supremum over $\optu^{0:N} \in \PDpredictConstr_{0:N}$ in \eqref{eq:pd:regret_Gdoughnut:1}.
 \end{proof}

\begin{remark}[Comparison set solution discrepancy]
    \label{rem:pd:solution-set-disrepancy}
    The comparison set solution discrepancy $c_N$ is the price for symmetricity in \cref{thm:pd:regret_estimate}, as compared to \cite{tuomov-predict}. It is non-positive and therefore can be made to disappear from \eqref{eq:pd:regret_estimate} if
    \begin{equation}
        \label{eq:pd:cn-zero-cond}
        G_{1:N}(K_{1:N}\optx^{1:N}) \le \iprod{K_{1:N}\optx^{1:N}}{\tilde y^{1:N}} - G_{1:N}^*(\tilde y^{1:N})
        \quad\text{for some}\quad
        \tilde y^{1:N} \in \DpredictConstr_{1:N},
    \end{equation}
    since the definition of the Fenchel biconjugate establishes $\iprod{K_{1:N}\optx^{1:N}}{y^{1:N}} - G_{1:N}^*(y^{1:N}) \le G_{1:N}(K_{1:N}\optx^{1:N})$.
    The condition \eqref{eq:pd:cn-zero-cond} also establishes
    \[
        \mathring G_{1:N}(K_{1:N}\optx^{1:N})=G_{1:N}(K_{1:N}\optx^{1:N}).
    \]

    Alternatively, if $\DpredictConstr_{1:N}$ is large enough, it may be possible to make $c_N$ small.
    In particular, if the algorithm-generated iterates $y^{1:N}$ are in $\DpredictConstr_{1:N}$, we have $c_N(\optx^{1:N}, y^{1:N})=0$.
    In our applications of interest, this seems, however, unlikely.
    Instead, in the following \cref{sec:predictors}, we will relate \eqref{eq:pd:cn-zero-cond} to total variation preserving predictors and true temporal couplings, taking $\tilde y^{1:N}$ such that $(\optx^{1:N}, \tilde y^{1:N}) \in \PDpredictConstr_{1:N}$.
\end{remark}

\section{Pseudo-affine predictors}
\label{sec:predictors}

Our first purpose in this section is to estimate in \cref{ssec:overallbounds} the prediction errors $e_N$, defined in \eqref{eq:pd:estimate-rhs}, for a class of what we call \textit{pseudo-affine predictors}.
We then provide examples of such predictors in \cref{ssec:TV_predictors,ssec:IP_predictors,ssec:dualscaling}, based on enforcing the preservation of salient relationships between the primal and dual variables.
Roughly speaking, these pointwise total variation and angle-preserving predictors model the fact that new shapes, hence significant changes in local total variation or angles, can only emerge near the image boundaries, i.e., in a set that has a small measure and hence a small effect on the prediction error.

In \cref{alg:pd:alg}, the time evolution of primal and dual variables is described by the predictors $P_k: X_k \times Y_k \to X_{k+1} \times Y_{k+1}$ where we wrote $ (\nexxt\primalpredict,\nexxt\dualpredict)=P_k(\thisx,\thisy)$ for all $k \in \N$. Throughout this section, we assume that
    \begin{equation}
        \label{eq:pseudolinear_form}
        P_k(\thisx,\thisy) = (W_k\thisx + a_{k+1}, T_k\thisy + b_{k+1})
    \end{equation}
 for some $W_k \in \linear(X_k;X_{k+1})$, $T_k \in \linear(Y_k; Y_{k+1})$, $a_k \in X_{k+1}$, $b_k \in Y_{k+1}$. Note that although $P_k$ is assumed affine, it may hide nonlinear dependencies through the dependence on the iteration $k$. This generality allows $W_k$ or $T_k$ to have nonlinear factors that also depend on the primal variable (as elucidated in the examples we consider for our purpose). We  take the same form of primal-dual temporal coupling operators $\opt{P}_k$ for some  $\opt{W}_k \in \linear(X_k;X_{k+1})$, $\opt{T}_k \in \linear(Y_k; Y_{k+1})$,  $\opt{a}_{k+1} \in X_{k+1}$ and $\opt{b}_{k+1} \in Y_{k+1}$ to describe the evolution of variables in the comparison sequence. More succinctly, for some family $\mathcal{P}$ of sequences of primal-dual temporal coupling operators $\{\tilde{P}_k\}_{k=0}^{\infty}$, we define the set of comparison sequences
\begin{gather*}%
    \PDpredictConstr \defeq
    \left\{
    \textstyle \optu^{0:\infty} \in \prod_{k=0}^\infty X_k \times Y_k \mid \optu^0  = (\optx^0,\opty^0) \in \PDpredictConstr _0, \nexxt\optu = \tilde{P}_k(\this\optx,\this\opty), \, (\tilde{P}_k)_{k=0}^{\infty} \in \mathcal{P}
    \right\}.
\end{gather*}%

\subsection{Prediction bounds and penalties}
\label{ssec:overallbounds}

Given the general form \eqref{eq:pseudolinear_form} of primal-dual predictors and temporal coupling operators, we first derive respective bounds for primal and dual predictions. This entails computing  (Lipschitz-like) factors $\Lambda_k, \Theta_k>0$ and prediction penalties $\epsilon_{k+1}, \tilde{\epsilon}_{k+1} \geq 0$ as in the next lemma.

\begin{lemma}
    \label{lemma:primaldualbounds}
    For all $k\in \N$ and for some $M_{\optx}, M_{\opty}>0$, let $\this\optx \in X_k$ with $\norm{\this\optx}_{X_k}^2 \leq M_{\optx}$, $\this\opty \in Y_k$ with $\norm{\this\opty}_{Y_k}^2 \leq M_{\opty}$,
    \begin{align*}
        (\nexxt\primalpredict,\nexxt\dualpredict)&=P_k(\thisx,\thisy) = \bigl(W_k \thisx + a_{k+1}, T_k \thisy + b_{k+1} \bigr),
        \quad\text{ and  }
        \\
        (\nexxt\optx,\nexxt\opty)&=\opt{P}_k(\this\optx,\this\opty) = \bigl(\opt{W}_k \this\optx + \opt{a}_{k+1}, \opt{T}_k \this\opty + \opt{b}_{k+1} \bigr)
    \end{align*}
    for a fixed $\{\opt{P}_k\}_{k=0}^{\infty} \in \mathcal{P}$, $W_k, \opt{W}_k \in \linear(X_k;X_{k+1})$,  $T_k, \opt{T}_k \in \linear(Y_k; Y_{k+1})$,  $a_{k+1},\opt{a}_{k+1} \in X_{k+1}$, and $b_{k+1}, \opt{b}_{k+1} \in Y_{k+1}$.  Then for any $\paffinecontrol >0$ and $\Lambda_k > \norm{W_k}^2$,
    \begin{gather*}
        \frac{1}{2}\norm{\nexxt\primalpredict - \nexxt\optx}_{X_{k+1}}^2 \leq \frac{\Lambda_k}{2}\norm{\thisx - \this\optx}_{X_{k}}^2 + \epsilon_{k+1}
    \shortintertext{where}
        \epsilon_{k+1}
        \defeq
        \frac{\Lambda_k M_{\optx}(1 + \paffinecontrol)}{\Lambda_k - \norm{W_k}^2} \norm{W_k-\opt{W}_k}^2
        + \frac{\Lambda_k(1+\paffinecontrol^{-1})}{\Lambda_k - \norm{W_k}^2}\norm{a_{k+1} - \opt{a}_{k+1}}^2.
    \end{gather*}
    Similarly, for any $\daffinecontrol>0$ and $\Theta_k > \norm{T_k}^2$,
    \begin{gather*}
        \frac{1}{2}\norm{\nexxt\dualpredict - \nexxt\opty}_{Y_{k+1}}^2 \leq \frac{\Theta_k}{2}\norm{\thisy - \this\opty}_{Y_{k}}^2 + \tilde{\epsilon}_{k+1}
    \shortintertext{where}
        \tilde{\epsilon}_{k+1}
        \defeq
        \frac{\Theta_k M_{\opty}(1+\daffinecontrol)}{\Theta_k - \norm{T_k}^2}\norm{T_k-\opt{T}_k}^2
        + \frac{\Theta_k(1+\daffinecontrol^{-1})}{\Theta_k - \norm{T_k}^2}\norm{b_{k+1}-\opt{b}_{k+1}}^2.
    \end{gather*}
\end{lemma}

\begin{proof}
    For any $t>0$, we apply Young's inequality to obtain
    \begin{align*}
        \norm{\nexxt\primalpredict - \nexxt\optx}^2
        &
        = \norm{W_k\thisx + a_{k+1} - \opt{W}_k\this\optx - \opt{a}_{k+1} }^2\\
        &
        \leq (1+t) \norm{W_k\thisx - W_k\this\optx}^2 + (1+t^{-1})\norm{W_k\this\optx - \opt{W}_k\this\optx + a_{k+1} - \opt{a}_{k+1}}^2\\
        &
        \leq (1+t)  \norm{W_k}^2 \norm{\thisx - \this\optx}^2
        \\
        \MoveEqLeft[-1]
        + (1+t^{-1})\big[(1+\paffinecontrol)\norm{(W_k -\opt{W}_k)(\this\optx)}^2
            + (1+\paffinecontrol^{-1}) \norm{a_{k+1} - \opt{a}_{k+1}}^2
        \big]
        \\
        &
        \leq  (1+t)\Lambda_k\norm{\thisx - \this\optx}^2
        \\
        \MoveEqLeft[-1]
        + (1+t^{-1})\big[(1+\paffinecontrol)\norm{W_k -\opt{W}_k}^2M_{\optx}
        + (1+\paffinecontrol^{-1}) \norm{a_{k+1} - \opt{a}_{k+1}}^2
        \big].
    \end{align*}
    Choosing $t = 1 - \norm{W_k}^2/\Lambda_k$ yields the desired Lipschitz-like constants and penalties for the primal prediction. The dual prediction bounds and penalties are obtained by a similar computation.
\end{proof}

We can now state our main result regarding the prediction errors $e_N$. The following theorem also imposes additional step length condition so that the prediction errors remain bounded.

\begin{theorem}
    \label{thm:overallbounds}
    Let \cref{ass:pd:main} together with the assumptions of \cref{lemma:primaldualbounds} hold.
    Suppose further that $\norm{\frac{\eta_k}{\eta_{k+1}}K_k - T_k^{\ast}K_{k+1}W_k}^2 \leq C_k$ for some $C_k \ge 0$.
    If the testing and step length parameters satisfy
    \begin{equation}%
        \label{eq:pd:primaltestcond}
         \tauTest_k(1+\gamma_k\tau_k) > \tauTest_{k+1} \Lambda_k,
    \end{equation}
    then
    \[
        e_N(u^{0:N-1},\optu^{0:N}) \le \sum_{k=0}^{N-1} \epsilon_{k+1}^{\dagger}(\this\optu),
    \]
    where, for any $\pi_k ,\tilde{\pi}_k, \beta > 0$ and $\kappa \in (0,1)$,
    \begin{equation}
        \label{eq:pd:overallpredictionerror}
        \begin{split}
            \epsilon_{k+1}^\dagger(\this\optu)
            &
            =
            \bigg(\frac{\sigmaTest_{k+1} \Theta_k - \kappa\sigmaTest_k(1+\rho_k\sigma_k)}{2} + \frac{\eta_{k+1}^2(C_k\beta +\norm{K_{k+1}}^2\norm{W_k}^2)}{2\beta(\tauTest_k(1+\gamma_k\tau_k) - \tauTest_{k+1} \Lambda_k)}\bigg) \norm{\thisy - \this\opty}^2
            \\
            \MoveEqLeft[-1]
            +
            \bigg(\frac{\eta_{k+1}^2 \norm{K_{k+1}}^2  \norm{T_k}^2}{2(1-\kappa) \psi_k (1+\rho_k\sigma_k)} + \frac{\eta_{k+1} \norm{K_{k+1}}^2}{2} +  \frac{\tauTest_{k+1}\Lambda_k}{\Lambda_k - \norm{W_k}^2} \bigg)
            \\
            \MoveEqLeft[-2]
            \cdot
            \Bigl[  M_{\optx}(1+\paffinecontrol)\norm{\opt{W}_k - W_k}^2 +
            (1+\paffinecontrol^{-1})\norm{\opt{a}_{k+1} - a_{k+1}}^2 \Bigr]
            \\
            \MoveEqLeft[-1]
            +  \bigg( \frac{\eta_{k+1}^2  (C_k\beta + \norm{K_{k+1}}^2\norm{W_k}^2) }{2(\tauTest_k(1+\gamma_k\tau_k) - \tauTest_{k+1} \Lambda_k)} + \frac{\eta_{k+1}}{2} + \frac{\sigmaTest_{k+1}\Theta_k}{\Theta_k - \norm{T_k}^2}\bigg)
            \\
            \MoveEqLeft[-2]
            \cdot
            \Bigl[M_{\opty}(1 + \daffinecontrol)\norm{\opt{T}_k - T_k}^2 + (1+\daffinecontrol^{-1})\norm{\opt{b}_{k+1} - b_{k+1}}^2\Bigr].
        \end{split}
    \end{equation}
\end{theorem}

\begin{remark}
    \label{rem:overallbounds}
    The first line of \eqref{eq:pd:overallpredictionerror}
    depends on the boundedness of the dual iterates; the remaining lines depend on the difference between the predictions and true temporal couplings, which disappear whenever the predictions are as good as the true temporal couplings. The factors $C_k$ measure the compatibility of the predictions with the operators $K_k$.
    In view of \cref{thm:overallbounds}, the regret estimate of \cref{thm:pd:regret_estimate} now shows that
    \begin{multline*}
        [Q_{1:N}(x^{1:N}) + \mathring G_{1:N}(K_{1:N}x^{1:N})]
        - \sup_{\optx^{1:N} \in  \PpredictConstr_{1:N}}
            [Q_{1:N}(\optx^{1:N}) + \mathring G_{1:N}(K_{1:N}\optx^{1:N})]
        \\
        \le
        \sup_{\optu^{0:N}\in  \PDpredictConstr_{0:N}} \bigg( \frac{1}{2}\norm{u^0-\optu^0}^2_{\eta_0\Precond_0+\Gamma_0} + c_N(\optx^{1:N}, y^{1:N}) + \sum_{k=0}^{N-1} \epsilon_{k+1}^{\dagger}(\this\optu) \bigg).
    \end{multline*}
    Minding the interpretation of $c_N(\optx^{1:N}, y^{1:N})$ in \cref{rem:pd:solution-set-disrepancy}, the regret therefore depends on how well (a) the comparison set solves the static problems, and (b) the predictors track the true temporal couplings.

    In the unaccelerated setting of \cref{ex:pd:noaccel}, with linear temporal evolution and predictors ($a_{k+1}=\bar a_{k+1}=0$, and $b_{k+1}=\bar b_{k+1}=0$), we can simplify
    \begin{equation*}
        \begin{split}
            \epsilon_{k+1}^\dagger(\this\optu)
            &
            =
            \bigg(\frac{\tau}{\sigma}\frac{\Theta_k - \kappa(1+\rho_k\sigma)}{2} + \frac{\tau^2(C_k\beta +\norm{K_{k+1}}^2\norm{W_k}^2)}{2\beta(1+\gamma_k\tau -  \Lambda_k)}\bigg) \norm{\thisy - \this\opty}^2
            \\
            \MoveEqLeft[-1]
            +
            \bigg(\frac{\tau\sigma \norm{K_{k+1}}^2  \norm{T_k}^2}{2(1-\kappa) (1+\rho_k\sigma)} + \frac{\tau \norm{K_{k+1}}^2}{2} +  \frac{\Lambda_k}{\Lambda_k - \norm{W_k}^2} \bigg)
            \cdot
            M_{\optx}\norm{\opt{W}_k - W_k}^2
            \\
            \MoveEqLeft[-1]
            +  \bigg( \frac{\tau^2  (C_k\beta + \norm{K_{k+1}}^2\norm{W_k}^2) }{2(1+\gamma_k\tau -  \Lambda_k) } + \frac{\tau}{2} + \frac{\tau}{\sigma}\frac{\Theta_k}{\Theta_k - \norm{T_k}^2}\bigg)
            \cdot
            M_{\opty}\norm{\opt{T}_k - T_k}^2.
        \end{split}
    \end{equation*}
    Moreover, if we have sufficient strong convexity compared to the bounds $\Theta_k$ and $\Lambda_k$, i.e., such that $1+\rho_k\sigma \ge \inv\kappa\Theta_k$ and $1+\gamma_k\tau > \Lambda_k$, we estimate
        \begin{equation*}
        \begin{split}
            \epsilon_{k+1}^\dagger(\this\optu)
            &
            \le
            \bigg(\frac{\tau\sigma \norm{K_{k+1}}^2  \norm{T_k}^2}{2(1-\kappa) (1+\rho_k\sigma)} + \frac{\tau \norm{K_{k+1}}^2}{2} +  \frac{\Lambda_k}{\Lambda_k - \norm{W_k}^2} \bigg)
            \cdot
            M_{\optx}\norm{\opt{W}_k - W_k}^2
            \\
            \MoveEqLeft[-1]
            +  \bigg( \frac{\tau}{2} + \frac{\tau}{\sigma}\frac{\Theta_k}{\Theta_k - \norm{T_k}^2}\bigg)
            \cdot
            M_{\opty}\norm{\opt{T}_k - T_k}^2.
        \end{split}
    \end{equation*}
    Additionally, using $\tau\sigma\norm{K_{k+1}}^2 \le 1$ from \eqref{eq:pd:primaltestcond-positivity} in \cref{ass:pd:main}, we further estimate
    \begin{equation*}
        \begin{split}
            \epsilon_{k+1}^\dagger(\this\optu)
            &
            \le
            \bigg(\frac{\norm{T_k}^2}{2(1-\kappa) \Theta_k} + \frac{1}{2\sigma} +  \frac{\Lambda_k}{\Lambda_k - \norm{W_k}^2} \bigg)
            \cdot
            M_{\optx}\norm{\opt{W}_k - W_k}^2
            \\
            \MoveEqLeft[-1]
            +  \frac{\tau}{\sigma}\bigg(\frac{1}{2\sigma} + \frac{\Theta_k}{\Theta_k - \norm{T_k}^2}\bigg)
            \cdot
            M_{\opty}\norm{\opt{T}_k - T_k}^2.
        \end{split}
    \end{equation*}
    Thus, in that case, the prediction penalty mainly depends on
    \begin{enumerate}[label=(\alph*)]
        \item the bounds $M_{\optx} \ge \norm{\this\optx}^2$ and $ M_{\opty} \ge \norm{\this\opty}^2$ on the comparison sequences;
        \item the closeness of the prediction operators $T_k$ and $W_k$ to the true temporal couplings $\opt T_k$ and $\opt W_k$; and
        \item the looseness of the upper bounds $\Theta_k$ and $\Lambda_k$ on $\norm{T_k}$ and $\norm{W_k}$.
    \end{enumerate}
\end{remark}

\begin{proof}[Proof of \cref{thm:overallbounds}]
    By the definition of the primal and dual predictors, we have
    \begin{multline}
        \label{eq:pd:predictineq0}
        \norm{\nexxt{\pdpredict}-\nexxt\optu}^2_{\eta_{k+1}\Precond_{k+1}}
        =
        \tauTest_{k+1}\norm{\nexxt x - \nexxt\primalpredict}^2
        +
        \sigmaTest_{k+1}\norm{\nexxt y  - \nexxt\dualpredict}^2
        \\
        -2\eta_{k+1} \iprod{K_{k+1}(W_k\thisx  - \opt{W}_k \this\optx + a_{k+1} - \opt{a}_{k+1} )}{T_k\thisy  - \opt{T}_k \this\opty + b_{k+1} - \opt{b}_{k+1}}
    \end{multline}
    and
    \begin{equation}
        \label{eq:pd:predictineq1}
        \begin{aligned}[b]
            \norm{\thisu-\this\optu}^2_{\eta_k\Precond_k+\Gamma_k}
            & =
            \tauTest_k(1+\gamma_k\tau_k)\norm{\thisx-\this\optx}^2
            +
            \sigmaTest_k(1+\rho_k\sigma_k)\norm{\thisy-\this\opty}^2
            \\
            \MoveEqLeft[-1]
            -2\eta_k \iprod{K_k(\thisx-\this\optx)}{\thisy-\this\opty}.
        \end{aligned}
    \end{equation}
    Expanding the last term at the right-hand side of \eqref{eq:pd:predictineq0} (without the scalar factor $\eta_{k+1}$) yields
    \begin{equation}
        \label{eqn:thirdterm_expansion}
        \begin{aligned}[b]
            -2\iprod{&K_{k+1}(W_k\thisx  - \opt{W}_k \this\optx + a_{k+1} - \opt{a}_{k+1} )}{T_k\thisy  - \opt{T}_k \this\opty + b_{k+1} - \opt{b}_{k+1}}
            \\
            &
            = -2\iprod{K_{k+1}W_k(\thisx  -  \this\optx)}{T_k(\thisy  - \this\opty)}
            \\
            \MoveEqLeft[-1]
            +2\iprod{K_{k+1}W_k( \this\optx -\thisx ) }{(T_k -\opt{T}_k) \this\opty + b_{k+1} - \opt{b}_{k+1}}
            \\
            \MoveEqLeft[-1]
            +2\iprod{ K_{k+1}( \opt{W}_k  - W_k)\this\optx + K_{k+1}(\opt{a}_{k+1} - a_{k+1})}{T_k(\thisy  -  \this\opty)}
            \\
            \MoveEqLeft[-1]
            +2\iprod{  K_{k+1}( \opt{W}_k  - W_k)\this\optx + K_{k+1}(\opt{a}_{k+1} - a_{k+1})}{(T_k -\opt{T}_k) \this\opty + b_{k+1}-\opt{b}_{k+1}}.
        \end{aligned}
    \end{equation}
    Using Young's inequality for some $s>0$,
    \begin{equation}
        \label{eqn:thirdterm_expansion_2nd}
        \begin{aligned}[b]
        2\iprod{K_{k+1}W_k&(\this\optx - \thisx ) }{(T_k -\opt{T}_k) \this\opty + b_{k+1} - \opt{b}_{k+1}}
        \\
        &
        \le
        s \norm{K_{k+1}}^2 \norm{W_k}^2 \norm{\this\optx -\thisx}^2
        + \inv s\norm{(T_k -\opt{T}_k) \this\opty + b_{k+1} - \opt{b}_{k+1}}^2
        \\
        &
        \le
        s \norm{K_{k+1}}^2 \norm{W_k}^2 \norm{\this\optx -\thisx}^2
        + \inv s (1 + \daffinecontrol)M_{\opt{y}}  \norm{T_k -\opt{T}_k }^2
        \\
        \MoveEqLeft[-1]
        + \inv s (1 + \daffinecontrol^{-1}) \norm{b_{k+1} - \opt{b}_{k+1}}^2.
        \end{aligned}
    \end{equation}
    Similarly, by Young's inequality (for some $r, \zeta >0$) we obtain
    \begin{equation}
        \label{eqn:thirdterm_expansion_3rd}
        \begin{aligned}[b]
            2\iprod{&K_{k+1}( \opt{W}_k  - W_k)\this\optx + K_{k+1}(\opt{a}_{k+1} - a_{k+1})}{T_k(\thisy  -  \this\opty)}
            \\
            &
            =
            2\iprod{( \opt{W}_k  - W_k)\this\optx + \opt{a}_{k+1} - a_{k+1}}{K_{k+1}^{\ast}T_k(\thisy  -  \this\opty)}
            \\
            &
            \le
            r \norm{( \opt{W}_k - W_k)\this\optx  + \opt{a}_{k+1} - a_{k+1} }^2
            + \inv r \norm{K_{k+1}}^2 \norm{T_k}^2 \norm{\thisy - \this\opty}^2
            \\
            &
            \le
            M_{\optx} r (1 + \paffinecontrol)  \norm{\opt{W}_k - W_k}^2
            +  r (1 + \paffinecontrol^{-1}) \norm{ \opt{a}_{k+1} - a_{k+1} }^2
            \\
            \MoveEqLeft[-1]
            + \inv r\norm{K_{k+1}}^2 \norm{T_k}^2\norm{\thisy - \this\opty}^2
        \end{aligned}
    \end{equation}
    and
    \begin{equation}
        \begin{aligned}[t]
            \label{eqn:thirdterm_expansion_4th}
            2\iprod{K_{k+1}&( \opt{W}_k  - W_k)\this\optx + K_{k+1}(\opt{a}_{k+1} - a_{k+1})}{(T_k -\opt{T}_k) \this\opty + b_{k+1}-\opt{b}_{k+1}}
            \\
            &
            \le
            \zeta \norm{K_{k+1}( \opt{W}_k  - W_k)\this\optx + K_{k+1}(\opt{a}_{k+1} - a_{k+1})}^2
            \\
            \MoveEqLeft[-1]
            + \inv\zeta\norm{(T_k -\opt{T}_k) \this\opty + b_{k+1}-\opt{b}_{k+1}}^2
            \\
            &
            \le
            \zeta M_{\optx} (1 + \paffinecontrol) \norm{K_{k+1}}^2 \norm{ \opt{W}_k  - W_k}^2
            \\
            \MoveEqLeft[-1]
            +  \zeta(1 + \paffinecontrol^{-1})\norm{K_{k+1}}^2 \norm{\opt{a}_{k+1} - a_{k+1}}^2
            \\
            \MoveEqLeft[-1]
            + \inv\zeta M_{\opt{y}}(1 + \daffinecontrol)\norm{T_k -\opt{T}_k}^2 + \inv\zeta (1 + \daffinecontrol^{-1})\norm{ b_{k+1}-\opt{b}_{k+1}}^2.
        \end{aligned}
    \end{equation}
    Substituting \eqref{eqn:thirdterm_expansion_2nd}, \eqref{eqn:thirdterm_expansion_3rd} and \eqref{eqn:thirdterm_expansion_4th} into \eqref{eqn:thirdterm_expansion} yields
    \begin{equation*}
        \begin{aligned}[b]
            - 2\iprod{K_{k+1}&(W_k\thisx   - \opt{W}_k \this\optx)}{T_k\thisy  - \opt{T}_k \this\opty}
            \\
            &
            \leq
            -  2\iprod{K_{k+1}W_k(\thisx  - \this\optx) }{T_k(\thisy  - \this\opty)}
            \\
            \MoveEqLeft[-1]
            +  s \norm{K_{k+1}}^2 \norm{W_k}^2 \norm{\this\optx -\thisx}^2
            + \inv s(1 + \daffinecontrol)M_{\opt{y}}  \norm{T_k -\opt{T}_k }^2
            \\
            \MoveEqLeft[-1]
            + \inv s (1 + \daffinecontrol^{-1}) \norm{b_{k+1} - \opt{b}_{k+1}}^2
            + M_{\optx} r (1 + \paffinecontrol)  \norm{\opt{W}_k - W_k}^2
            \\
            \MoveEqLeft[-1]
            + r (1 + \paffinecontrol^{-1}) \norm{ \opt{a}_{k+1} - a_{k+1} }^2
            + \inv r \norm{K_{k+1}}^2 \norm{T_k}^2\norm{\thisy - \this\opty}^2
            \\
            \MoveEqLeft[-1]
            + \zeta M_{\optx} (1 + \paffinecontrol) \norm{K_{k+1}}^2 \norm{ \opt{W}_k  - W_k}^2
            + \zeta (1 + \paffinecontrol^{-1}) \norm{K_{k+1}}^2 \norm{\opt{a}_{k+1} - a_{k+1}}^2
            \\
            \MoveEqLeft[-1]
            + \inv\zeta M_{\opt{y}}(1 + \daffinecontrol)\norm{T_k -\opt{T}_k}^2
            + \inv\zeta (1 + \daffinecontrol^{-1})\norm{ b_{k+1}-\opt{b}_{k+1}}^2.
        \end{aligned}
    \end{equation*}
    Combining this inequality with \eqref{eq:pd:predictineq0}, then invoking the prediction bounds and penalties from \cref{lemma:primaldualbounds}, we get
    \begin{equation}
        \label{eq:pd:predictineq0_upper}
        \begin{aligned}[b]
            \norm{\nexxt{\pdpredict}-\nexxt\optu}^2_{\eta_{k+1}\Precond_{k+1}}
            &
            \leq
            \tauTest_{k+1} \Lambda_k \norm{\thisx  -  \this\optx}^2
            + \sigmaTest_{k+1} \Theta_k \norm{\thisy  - \this\opty}^2
            + \eta_{k+1} h_{k+1}
            \\
            \MoveEqLeft[-1]
            + 2\tauTest_{k+1}\epsilon_{k+1} + 2\sigmaTest_{k+1}\tilde{\epsilon}_{k+1}
        \end{aligned}
    \end{equation}
    for
    \[
        \begin{aligned}
            h_{k+1}
            &
            \defeq
            - 2 \iprod{K_{k+1}W_k(\thisx  - \this\optx) }{T_k(\thisy  - \this\opty)}
            +   s \norm{K_{k+1}}^2 \norm{W_k}^2 \norm{\this\optx -\thisx}^2
            \\
            \MoveEqLeft[-1]
            + \inv s(1 + \daffinecontrol)M_{\opt{y}}  \norm{T_k -\opt{T}_k }^2
            + \inv s (1 + \daffinecontrol^{-1}) \norm{b_{k+1} - \opt{b}_{k+1}}^2
            \\
            \MoveEqLeft[-1]
            +  M_{\optx} r (1 + \paffinecontrol)  \norm{\opt{W}_k - W_k}^2
            +  r (1 + \paffinecontrol^{-1}) \norm{ \opt{a}_{k+1} - a_{k+1} }^2
            \\
            \MoveEqLeft[-1]
            + \inv r  \norm{K_{k+1}}^2 \norm{T_k}^2\norm{\thisy - \this\opty}^2
            +  \zeta \norm{K_{k+1}}^2 M_{\optx} (1 + \paffinecontrol)\norm{ \opt{W}_k  - W_k}^2
            \\
            \MoveEqLeft[-1]
            +  \zeta \norm{K_{k+1}}^2(1 + \paffinecontrol^{-1}) \norm{\opt{a}_{k+1} - a_{k+1}}^2
            + \inv\zeta M_{\opt{y}}(1 + \daffinecontrol)\norm{T_k -\opt{T}_k}^2
            \\
            \MoveEqLeft[-1]
            + \inv\zeta (1 + \daffinecontrol^{-1})\norm{ b_{k+1}-\opt{b}_{k+1}}^2.
        \end{aligned}
    \]
    Furthermore, we note that by another application of Young's inequality (for some $t>0$) and by using the assumption that $\norm{\frac{\eta_k}{\eta_{k+1}}K_k - T_k^{\ast}K_{k+1}W_k}^2 \leq C_k$, it follows that
    \begin{equation}
        \label{eqn:thirdterm_expansion_1st}
        \begin{aligned}[t]
            2 \eta_k \iprod{K_k(\thisx-\this\optx)}{\thisy-\this\opty} &- 2\eta_{k+1}\iprod{K_{k+1}W_k(\thisx  -  \this\optx) }{T_k(\thisy  -\this\opty)}
            \\
            &
            =
            2\eta_{k+1}\adaptiprod{\left(\frac{\eta_k}{\eta_{k+1}}K_k - T_k^{\ast}K_{k+1}W_k\right)(\thisx - \this\optx)}{\thisy - \this\opty}
            \\
            &
            \leq
            \eta_{k+1} C_k t \norm{\thisx - \this\optx}^2 + \inv t\eta_{k+1}\norm{\thisy-\this\opty}^2
        \end{aligned}
    \end{equation}
    Subtracting  \eqref{eq:pd:predictineq1} from  \eqref{eq:pd:predictineq0_upper}, and using  \eqref{eqn:thirdterm_expansion_1st} yields
    \begin{equation}
        \begin{aligned}[t]
        \label{eq:pd:overallfinal}
        \norm{\nexxt{\pdpredict}-\nexxt\optu}^2_{\eta_{k+1}\Precond_{k+1}} & - \norm{\thisu-\this\optu}^2_{\eta_k\Precond_k+\Gamma_k}
        \\
        \MoveEqLeft[6]
        \leq
        \bigl(\tauTest_{k+1} \Lambda_k + \eta_{k+1}C_k t +  \eta_{k+1} \norm{K_{k+1}}^2 \norm{W_k}^2 s  - \tauTest_k(1+\gamma_k\tau_k) \bigr)\norm{\this\optx -\thisx}^2
        \\
        \MoveEqLeft[9]
        + \bigl(\sigmaTest_{k+1} \Theta_k + \eta_{k+1}\inv t +  \eta_{k+1} \norm{K_{k+1}}^2 \norm{T_k}^2\inv r  - \sigmaTest_k(1+\rho_k\sigma_k)\bigr) \norm{\thisy - \this\opty}^2
        \\
        \MoveEqLeft[9]
        + \eta_{k+1}\bigl( r + \norm{K_{k+1}}^2  \zeta\bigr) \left[M_{\optx}(1+\paffinecontrol) \norm{ \opt{W}_k- W_k}^2 + (1+\paffinecontrol^{-1}) \norm{ \opt{a}_{k+1}- a_{k+1}}^2\right]
        \\
        \MoveEqLeft[9]
        + \eta_{k+1}\bigl(\inv s + \inv\zeta\bigr)\bigl[
            M_{\opty}(1+\daffinecontrol)\norm{ \opt{T}_k - T_k}^2 + (1+\daffinecontrol^{-1})\norm{ \opt{b}_{k+1} - b_{k+1}}^2
        \bigr]
        \\
        \MoveEqLeft[9]
        +  2\tauTest_{k+1}\epsilon_{k+1} + 2\sigmaTest_{k+1}\tilde{\epsilon}_{k+1}.
        \end{aligned}
    \end{equation}
    Since $\tauTest_k(1+\gamma_k\tau_k) - \tauTest_{k+1} \Lambda_k >0 $ by assumption, setting $t = \beta s$ for some $\beta > 0$ and choosing
    \[
        s =  \dfrac{\tauTest_k(1+\gamma_k\tau_k) - \tauTest_{k+1} \Lambda_k}{\eta_{k+1}(C_k \beta + \norm{K_{k+1}}^2 \norm{W_k}^2)} > 0
    \]
    gives us
    \[
            \tauTest_{k+1} \Lambda_k + \eta_{k+1}C_k t + \eta_{k+1} \norm{K_{k+1}}^2 \norm{W_k}^2 s - \tauTest_k(1+\gamma_k\tau_k) = 0.
    \]
    By using the definition of $\epsilon_{k+1}$ and $\tilde{\epsilon}_{k+1}$ from \cref{lemma:primaldualbounds}, and by setting
    \[
        r = \dfrac{\eta_{k+1}\norm{K_{k+1}}^2\norm{T_k}^2}{\psi_k(1-\kappa)(1 + \rho_k \sigma_k)} \quad\text{and}\quad
        \zeta = 1,
    \]
    the remaining terms on the right-hand side of \eqref{eq:pd:overallfinal} give the prediction penalties. Thus, \eqref{eq:pd:overallfinal} turns into the desired overall prediction bound.
\end{proof}

\begin{remark}
Alternatively, for any $\beta, \mu, \omega \in (0,1)$, we can set
    \begin{align*}
        t & =\dfrac{\eta_{k+1}}{(1-\omega)(1+\rho_k\sigma_k)\psi_k},
        &
        s & = \dfrac{(1-\mu)\phi_k (1+\gamma_k\tau_k)}{\eta_{k+1} \norm{K_{k+1}}^2 \norm{W_k}^2},
        \\
        r & = \dfrac{\eta_{k+1} \norm{K_k}^2 \norm{T_k}^2}{(1-\beta)\omega\psi_k(1+\rho_k\sigma_k)},
        \quad\text{and}
        &
        \zeta & = 1
    \end{align*}
    in \eqref{eq:pd:overallfinal} while imposing the dual predictor restriction and the primal metric update bounds
        \begin{align*}%
             \beta\kappa \psi_k(1+\rho_k\sigma_k)
            &> \psi_{k+1}\Theta_k
            \quad\text{and}
            \\
            \mu \tauTest_k(1+\gamma_k \tau_k)
            & >
            { \tauTest_{k+1}\Lambda_k + \frac{ \eta_{k+1}^2 C_k}{\psi_k(1-\kappa)(1+\sigma_k\rho_k)}}.
        \end{align*}%
    Then
    $
        e_N(u^{0:N-1},\optu^{0:N}) \le \sum_{k=0}^{N-1} \epsilon_{k+1}^{\dagger}(\this\optu)
    $
    with
    \[
        \begin{aligned}
            \epsilon_{k+1}^\dagger
            &
            =
            \bigg(\frac{\eta_{k+1}^2  \norm{K_{k+1}}^2 \norm{T_k}^2}{2(1-\beta)\kappa\omega \psi_k (1+\rho_k\sigma_k)} + \frac{\eta_{k+1}  \norm{K_{k+1}}^2}{2}  +  \frac{\tauTest_{k+1}\Lambda_k}{\Lambda_k - \norm{W_k}^2} \bigg)
            \\
            \MoveEqLeft[-2]
            \cdot
            \Bigl[M_{\optx}(1+\paffinecontrol) \norm{ \opt{W}_k- W_k}^2 + (1+\paffinecontrol^{-1}) \norm{ \opt{a}_{k+1}- a_{k+1}}^2\Bigr]
             \\
            \MoveEqLeft[-1]
            +  \bigg( \frac{\eta_{k+1}^2  \norm{K_{k+1}}^2 \norm{W_k}^2}{2(1-\mu) \phi_k (1+\gamma_k \tau_k)} + \frac{\eta_{k+1}}{2} + \frac{\sigmaTest_{k+1}\Theta_k}{\Theta_k - \norm{T_k}^2} \bigg)
            \\
            \MoveEqLeft[-2]
            \cdot
             \Bigl[M_{\opty}(1+\daffinecontrol)\norm{ \opt{T}_k - T_k}^2 + (1+\daffinecontrol^{-1})\norm{ \opt{b}_{k+1} - b_{k+1}}^2\Bigr].
        \end{aligned}
    \]
    Observe that, in contrast to \cref{thm:overallbounds}, this version of $\epsilon_{k+1}^\dagger$ does not depend on  $\norm{\thisy - \this\opty}^2$.
\end{remark}

\subsection{Total variation preserving predictors}
\label{ssec:TV_predictors}

We now consider examples of pseudo-linear predictors within the framework of optical flow. In what follows, we let
\[
    G_k = \alpha\|\cdot\|_{2,1} \text{ for some } \alpha>0, \text{ and } K_k = D
\]
some $D \in \linear(X_k; Y_k)$, typically, but not necessarily a differential operator, with $Y_k \subset L^2(\Omega; \R^m)$ for some $m \ge 1$ and a domain $\Omega$ such that the “global 1-norm, pointwise 2-norm” is defined. We will make more specific assumptions on $\Omega$ and $X_k$ for individual predictors.
Recall our notation for the primal-dual predictor $P_k: X_k \times Y_k \to X_{k+1} \times Y_{k+1}$ where we wrote $ (\nexxt\primalpredict,\nexxt\dualpredict)=P_k(\thisx,\thisy)$ for all $k \in \N$. We say that a predictor $P_k$ is (strictly) \emph{total variation preserving}\footnote{The nomenclature changes if $K_k$ is chosen differently from a differential operator.} if the primal and dual predictions satisfy
\begin{multline}
    \label{eq:TV_predictors:preservation}
    \alpha \|D\thisx\|_{2,1} = \iprod{D\thisx}{\thisy} \text{ and } \norm{\thisy}_{2,\infty} \leq  \alpha
    \\
    \implies
    \alpha \norm{D \nexxt\primalpredict}_{2,1}= \iprod{D\nexxt\primalpredict}{\nexxt\dualpredict}
    \quad (\text{and, for strict preservation, } \norm{\nexxt\dualpredict}_{2,\infty} \le \alpha).
\end{multline}
We give several examples of this kind of predictor in the following discussion.
In each of these examples, we establish conditions guaranteeing the assumptions of \cref{thm:overallbounds} on the predictors, that is, bounds $C_k > \norm{\frac{\eta_k}{\eta_{k+1}}K_k - T_k^{\ast}K_{k+1}W_k}^2$, $\Theta_k > \norm{T_k}^2$, and $\Lambda_k > \norm{W_k}^2$, for the linear parts $T_k$ and $W_k$ of the primal and dual predictors, where the overall prediction $P_k(\thisx,\thisy) = (W_k\thisx + a_k, T_k\thisy + b_k)$.

Observe that \eqref{eq:pd:cn-zero-cond}, which guarantees that the comparison set solution discrepancy $ c_N(\optx^{1:N}, y^{1:N})  \le 0$, holds if  $\alpha \|D\opt x^0\|_{2,1} = \iprod{D\opt x^0}{\opt y^0}$ and $\norm{\opt y^0}_{2,\infty} \leq  \alpha$, and the true temporal coupling is strictly total variation preserving:
\begin{multline}
    \label{eq:TV_predictors:preservation:strong:coupling}
    \alpha \|D\this{\optx}\|_{2,1} = \iprod{D\this{\optx}}{\this{\opty}}
    \text{ and }
    \norm{\this{\opty}}_{2,\infty} \leq  \alpha
    \\
    \implies
    \alpha \norm{D \nexxt{\optx}}_{2,1}= \iprod{D\nexxt{\optx}}{\nexxt{\opty}}
    \text{ and }
    \norm{\nexxt{\opty}}_{2,\infty} \le \alpha.
\end{multline}

\subsubsection*{Pointwise preservation in $L^{2}$}

For a domain $\Omega \subset \R^n$, we work with the subspaces $X_k \subset L^2(\Omega)$, $Y_k \subset L^2(\Omega; \R^{m})$, and a subset $\InvDisplacements_k \subset H^1(\Omega; \Omega)$ of bijective displacement fields, from which we take a measured $v^k$ and a true $\optv^k$ that give the primal prediction as $\nexxt\primalpredict = x^k \circ v^k$, and the true primal temporal evolution as $\nexxt\optx=\this\optx \circ \this\optv$.

\renewcommand{\abs}[1]{|#1|}

\begin{lemma}
    \label{lemma:predictors:pointwise-l2}
    Let  $\InvDisplacements_k \subset H^1(\Omega; \Omega)$ be a set of bijective displacement fields satisfying $\inv{(\thisv)} \in H^1(\Omega; \Omega)$ for all $\thisv \in \InvDisplacements_k$, and
     \[
        \Lambda_{ \InvDisplacements_k} \defeq \sup_{\thisv \in \InvDisplacements_k, \xi \in \Omega }\abs{\det{\grad \inv{(\thisv)}(\xi)}} < \infty.
    \]
    Let $\thisx \in X_k$, $\thisy \in Y_k$, $\thisv \in {\InvDisplacements_k}$, and
    \[
        (\nexxt\primalpredict,\nexxt\dualpredict) = P_k(\thisx,\thisy) = (W_k\thisx, T_k\thisy),
    \]
    where the primal  $W_k: X_k \to X_{k+1}$ and  dual predictors $T_k: Y_k \to Y_{k+1}$  are defined  pointwise by
    \begin{subequations}
    \label{lemma:eq:predictions_ex1}
    \begin{align}
        (W_k\thisx)(\xi) & \defeq (\thisx \circ \thisv)(\xi)
        \quad\text{ and } \quad
        (T_k\thisy)(\xi)\defeq
            \this{t}(\xi) \thisy(\thisv(\xi))
    \intertext{with $\this{t}(\xi)$ given by}
        \this{t}(\xi) & \defeq
        \begin{cases}
            \frac{| \grad \thisv(\xi)^* D \thisx (\thisv(\xi))| }{|D\thisx (\thisv(\xi))|}(\grad \thisv(\xi))^{-1},
            & \ |D\thisx (\thisv(\xi))| \neq 0,
            \\
            \Id,
            & \ |D\thisx (\thisv(\xi))| = 0.
        \end{cases}
    \end{align}
    \end{subequations}
    This predictor is total variation preserving, i.e., satisfies the non-strict variant of \eqref{eq:TV_predictors:preservation}.
    If we have $\abs{(\grad \thisv(\xi))^{-1}} \abs{\grad \thisv(\xi)} \le 1$ for all $\xi \in \Omega$ (e.g., $\thisv$ is a simple translation and rotation), this preservation is strict, and the true temporal couplings constructed analogously based on a true displacement field $\this{\bar v}$, satisfy \eqref{eq:TV_predictors:preservation:strong:coupling}. Moreover, $\norm{W_k}^2, \norm{T_k}^2 \leq \Lambda_{ \InvDisplacements_k}$ and
    \[
        \adaptnorm{\frac{\eta_k}{\eta_{k+1}}K_k - T_k^{\ast}K_{k+1}W_k}^2 \leq \tilde{C}_k\norm{D}^2
    \]
    for
    \[
       \tilde{C}_k \defeq \max\left(\frac{\eta_k}{\eta_{k+1}}, \sup_{\xi \in \Omega \setminus\Omega_0} \left|\frac{\eta_k}{\eta_{k+1}}  -  \frac{|\det \grad \inv{(\thisv)}(\xi)|| \grad \thisv(\inv{(\thisv)}(\xi))^* D \thisx (\xi)| }{|D\thisx(\xi)|}    \right|^2\right).
    \]
     with $\Omega_0 \defeq \{ \xi \in \Omega \mid D\thisx(\thisv(\xi)) = 0\}$.
\end{lemma}

\begin{proof}
    Note that due to $\iprod{D\this{\optx}}{\this{\opty}} \le \norm{D\this{\optx}}_{2,1}\norm{\this{\opty}}_{2,\infty}$ the strict variant of the total variation preservation \eqref{eq:TV_predictors:preservation} can now be written pointwise
    \begin{multline}
        \label{eq:TV_predictors:preservation:l2}
        \alpha|D \thisx(\xi)|=\iprod{D \thisx(\xi)}{\thisy(\xi)} \text{ and } |\thisy(\xi)| \leq \alpha \text{ for a.e. }\xi \in \Omega
        \\
        \implies
        \alpha|D\nexxt\primalpredict(\xi)| =  \iprod{D\nexxt\primalpredict(\xi)}{\nexxt\dualpredict(\xi)}
        \text{ and } |\nexxt\dualpredict(\xi)| \leq \alpha
        \text{ for a.e. }\xi \in \Omega.
    \end{multline}

    Let $\xi \in \Omega$ be an arbitrary point where the antecedent \eqref{eq:TV_predictors:preservation:l2} holds.
    When $|D\thisx (\thisv(\xi))| =0$, we deduce that $|D\nexxt\primalpredict(\xi)| =0$ and so $\alpha|D\nexxt\primalpredict(\xi)| =  \iprod{D\nexxt\primalpredict(\xi)}{\nexxt\dualpredict(\xi)}$ holds trivially.
    Suppose then that $|D\thisx (\thisv(\xi))| \neq 0$.
    Since  $0<\alpha|D \thisx(\zeta)|=\iprod{D \thisx(\zeta)}{\thisy(\zeta)}$ with $\abs{\thisy(\zeta)}\le\alpha$ for all $\zeta$, we must have $\thisy(\zeta)=\alpha|D \thisx(\zeta)|^{-1}D \thisx(\zeta)$.
    Taking $\zeta=\thisv(\xi)$, we obtain $\thisy(\thisv(\xi))=\alpha|D \thisx(\thisv(\xi))|^{-1}D \thisx(\thisv(\xi))$. By definition of $\this\primalpredict$ and $\this\dualpredict$ we obtain, as required
    \begin{equation*}
        \begin{aligned}[t]
        \iprod{D\nexxt\primalpredict(\xi)}{\nexxt\dualpredict(\xi)}
                &
                =   \adaptiprod{\grad \thisv(\xi)^* D \thisx (\thisv(\xi))}{\this{t}(\xi)y^k(v^k(\xi))}
                \\
                &
                = \adaptiprod{D \thisx (\thisv(\xi))}{\frac{| \grad \thisv(\xi)^* D \thisx (\thisv(\xi))| }{|D\thisx (\thisv(\xi))|}  \thisy(\thisv(\xi))}
                \\
                &
                = | \grad \thisv(\xi)^* D \thisx (\thisv(\xi))| \adaptiprod{\frac{D \thisx (\thisv(\xi))}{|D\thisx (\thisv(\xi))|}}{\frac{\alpha D \thisx(\thisv(\xi)) }{|D\thisx (\thisv(\xi))|}}
                \\
                &
                = \alpha|D\nexxt\primalpredict(\xi)|.
        \end{aligned}
    \end{equation*}
    Thus the fist part of the consequent of \eqref{eq:TV_predictors:preservation:l2} holds, proving non-strict total variation preservation.
    If $\abs{(\grad \thisv(\xi))^{-1}} \abs{\grad \thisv(\xi)} \le 1$, we have $\abs{\this t(\xi)} \le 1$, so $|\nexxt{\dualpredict(\xi)}| \le \alpha$ for all $\xi$, so the second part of the consequent of \eqref{eq:TV_predictors:preservation:l2} holds, proving strict total variation preservation.
    The same argumentation applied to analogous true temporal couplings, proves \eqref{eq:TV_predictors:preservation:strong:coupling}.

    Moreover, with $\Omega_0 = \{ \xi \in \Omega \mid D\thisx(\thisv(\xi)) = 0\}$,
    \begin{equation*}
        \begin{aligned}
        \norm{T_k}^2
            &
            = \sup_{\|\thisy\| = 1}\left(
                \int_{\Omega \setminus \Omega_0}\left| \frac{| \grad \thisv(\xi)^* D \thisx (\thisv(\xi))| }{|D\thisx (\thisv(\xi))|}(\grad \thisv(\xi))^{-1} \thisy(\thisv(\xi)) \right|^2 d\xi
                + \int_{\Omega_0}\bigl| \thisy(\thisv(\xi)) \bigr|^2 d\xi
            \right)
            \\
            &
            \leq \sup_{\|\thisy\| = 1}  \int_{\Omega} \bigl|\thisy(\thisv(\xi))\bigr|^2 d\xi
            \\
            &
            = \sup_{\|\thisy\| = 1}  \int_{\Omega} \bigl|\thisy(\xi)\bigr|^2 \bigl|\det \grad \inv{(\thisv)}(\xi)\bigr| d\xi
            \leq \Lambda_{ \InvDisplacements_k}.
        \end{aligned}
    \end{equation*}
    The upper bound for $\norm{W_k}$ is obtained by a similar computation. For any $\nexxt{z} \in Y_{k+1}$ and $\xi \in \Omega$, we have
    \[
        \begin{aligned}
        \int_{\Omega}\iprod{\nexxt z(\xi)}{t^k(\xi)\thisy(\thisv(\xi))}d\xi
        &
        =
        \int_{\Omega}\iprod{t^k(\xi)^{\ast} \nexxt z(\xi)}{\thisy(\thisv(\xi))}d\xi
        \\
        &
        =
        \int_{\Omega}\int\iprod{t^k((\thisv)^{-1}(\xi))^{\ast} \nexxt z((\thisv)^{-1}(\xi))}{\thisy(\xi)} |\det \grad \inv{(\thisv)}(\xi)|d\xi
        \\
        &
        =
        \int_\Omega\iprod{|\det \grad \inv{(\thisv)}(\xi)| t^k((\thisv)^{-1}(\xi))^{\ast} \nexxt z((\thisv)^{-1}(\xi))}{\thisy(\xi)}d\xi
        \end{aligned}
    \]
    for
    \[
        \this{t}(\xi)^{\ast} =
        \begin{cases}
            \frac{| \grad \thisv(\xi)^* D \thisx (\thisv(\xi))| }{|D\thisx (\thisv(\xi))|}((\grad \thisv(\xi))^{-1})^{\ast},
            & \ |D\thisx (\thisv(\xi))| \neq 0,
            \\
            \Id,
            & \ |D\thisx (\thisv(\xi))| = 0.
        \end{cases}
    \]
    Hence
    \[
        \iprod{\nexxt z}{T_k \thisy}
        = \int_{\Omega} \iprod{\nexxt z(\xi)}{t^k(\xi)\thisy(\thisv(\xi))} \d\xi
        = \int_{\Omega} \iprod{(T_k^*\nexxt z)(\xi)}{\thisy} \d\xi
    \]
    for
    \[
        (T_k^{\ast} \nexxt z)(\xi) =
        |\det \grad \inv{(\thisv)}(\xi)| \, \this{t}(\inv{(\thisv)}(\xi))^{\ast} \nexxt z(\inv{(\thisv)}(\xi)).
    \]
    Using this expression for $T_k^{\ast}$, for any $\xi \in \Omega\setminus\Omega_0$ we obtain
    \[
        \begin{aligned}
        [T_k^{\ast} D W_k \thisx] (\xi)
        &
        =
        T_k^{\ast}[\xi \mapsto \grad \thisv(\xi)^{\ast} D\thisx(\thisv(\xi))](\xi)
        \\
        &
        =
        \frac{|\det \grad \inv{(\thisv)}(\xi)|| \grad \thisv(\inv{(\thisv)}(\xi))^* D \thisx (\xi)| }{|D\thisx(\xi)|}  D\thisx(\xi),
        \end{aligned}
    \]
    while for any $\xi \in \Omega_0$ we have $[T_k^{\ast} D W_k \thisx] (\xi)=0$.
    Thus
    \begin{equation*}
        \begin{aligned}
        \biggl\|\frac{\eta_k}{\eta_{k+1}}D - T_k^{\ast}DW_k\biggr\|^2
            &
            =
            \sup_{\|\thisx\| = 1} \biggl( \int_{\Omega\setminus\Omega_0}\left|\frac{\eta_k}{\eta_{k+1}}D\thisx(\xi) - [T_k^{\ast} D W_k \thisx] (\xi)\right|^2 d\xi
            +
            \int_{\Omega_0}\left|\frac{\eta_k}{\eta_{k+1}}D\thisx(\xi)\right|^2 d\xi \biggr)
            \\
            &
            = \sup_{\|\thisx\| = 1} \biggl( \int_{\Omega\setminus \Omega_0} \left| \frac{\eta_k}{\eta_{k+1}}-   \frac{|\det \grad \inv{(\thisv)}(\xi)|| \grad \thisv(\inv{(\thisv)}(\xi))^* D \thisx (\xi)| }{|D\thisx(\xi)|}  \right|^2  |D\thisx(\xi) |^2 d\xi
            \\
            \MoveEqLeft[-1]
            +
            \int_{\Omega_0}\left|\frac{\eta_k}{\eta_{k+1}}\right|^2 \left|D\thisx(\xi)\right|^2 d\xi  \biggr)
            \\
            &
            \leq \tilde{C}_k  \sup_{\|\thisx\| = 1} \int_{\Omega} |D\thisx(\xi)|^2 d\xi
            = \tilde{C}_k \norm{D}^2. \qedhere
        \end{aligned}
    \end{equation*}
\end{proof}

\begin{remark}
Note that if  $\eta_k = \eta_{k+1}$ and $\thisv = \Id$ (i.e. there is no prediction),  then we can choose $\tilde{C}_k = 1$.
\end{remark}

\begin{remark}[Strict Greedy predictor]
    \label{rem:predictors:strictgreedy}
    During the review process, we realised that a slightly modified version of the dual predictor of \cref{lemma:predictors:pointwise-l2} strictly preserves total variation.
    Indeed, take
    \[
        \nexxt{\breve y}(\xi) \defeq [T_k y^k](\xi) \defeq \frac{\iprod{D\this x(\this v(\xi))}{\this y(\this v(\xi))}}{|D\this x(\this v(\xi))|} \frac{D\nexxt{\breve x}(\xi)}{|D\nexxt{\breve x}(\xi)|},
    \]
    where we can replace $D\nexxt{\breve x}(\xi)/|D\nexxt{\breve x}(\xi)|$ by any unit vector if $D\nexxt{\breve x}(\xi)=0$.
    Then $|\nexxt{\breve y}(\xi)| \le |\this y(\this v(\xi))| \le \alpha$ and
    $\iprod{D\nexxt{\breve x}(\xi)}{\nexxt{\breve y}(\xi)} = \alpha|D\nexxt{\breve x}(\xi)|$ if  $\iprod{D\this x(\this v(\xi))}{\this y(\this v(\xi))}=\alpha|D\this x(\this v(\xi))|$.

    We call this the Strict Greedy predictor, as $\nexxt{\breve y}(\xi)$ has the same direction as $D\nexxt{\breve x}(\xi)$ with a scaling baed on the total variation attainment discrepancy of the previous step.
    Although this is a theoretically highly satisfying predictor, we do not include a derivation of the prediction error bounds, as the practical performance is comparable to our other proposed predictors.
\end{remark}

\subsubsection*{Pointwise preservation in $L^{2}$ by rotation}

An alternative way to preserve total variation following a primal predictor is by implementing a suitable rotation on the dual variable. We use this idea to construct a dual predictor as in the following lemma, which uses the same spaces $X_k \subset L^2(\Omega)$ and $Y_k \subset L^2(\Omega; \R^{m})$ as \cref{lemma:predictors:pointwise-l2}.

\begin{lemma}
    \label{lemma:predictor-rotation}
    Let $\InvDisplacements_k \subset H^1(\Omega; \Omega)$ be a set of bijective displacement fields satisfying $\inv{(\thisv)} \in H^1(\Omega; \Omega)$ for all $\thisv \in \InvDisplacements_k$, and
    \[
        \Lambda_{ \InvDisplacements_k} \defeq \sup_{\thisv \in \InvDisplacements_k, \xi \in \Omega }\abs{\det{\grad \inv{(\thisv)}(\xi)}} < \infty.
    \]
    Let $\thisx \in X_k$, $\thisy \in Y_k$, $\thisv \in \InvDisplacements_k$,  and  $(\nexxt\primalpredict,\nexxt\dualpredict) = P_k(\thisx,\thisy) = (W_k\thisx, T_k\thisy + b_k)$ where the primal predictor $W_k: X_k \to X_{k+1}$ and the affine dual predictor consisting of $T_k: Y_k \to Y_{k+1}$ and $b_k \in Y_{k+1}$, are defined pointwise for $\xi \in \Omega$ by
    \begin{equation}
        \label{lemma:eq:TVpreserve_rotation}
        (W_k\thisx)(\xi) = (\thisx \circ \thisv)(\xi),
        \quad
        (T_k\thisy)(\xi) =
        \begin{cases}
            R_{\theta_{\xi}}(\thisy(\xi)), & D \thisx(\xi) \ne 0, \\
            0, & D \thisx(\xi) = 0 \\
        \end{cases}
    \end{equation}
    and
    \[
        b_k(\xi)
        \defeq
        \begin{cases}
            \alpha\frac{D\nexxt\primalpredict(\xi)}{\norm{D\nexxt\primalpredict(\xi)}_2}, & D \thisx(\xi) = 0, D\nexxt\primalpredict(\xi) \ne 0, \\
            0, & \text{otherwise,}
        \end{cases}
    \]
    where $R_{\theta_{\xi}}$ is a rotation operator defined by an oriented angle $\theta_{\xi}$ from $D\thisx(\xi)$ to $D\nexxt\primalpredict(\xi)$ such that
    \[
        D\nexxt\primalpredict(\xi) = c_{\xi} R_{\theta_{\xi}}(D\thisx(\xi)) \quad\text{for some}\quad  c_{\xi} \geq 0.
    \]
    This predictor is strictly total variation preserving, i.e., satisfies the strict variant of \eqref{eq:TV_predictors:preservation}. Moreover, true temporal couplings constructed analogously based on a true displacement field $\this{\bar v}$, satisfy \eqref{eq:TV_predictors:preservation:strong:coupling}.
    Moreover, $\norm{W_k}^2 \leq \Lambda_{\InvDisplacements_k}$, $\norm{T_k}^2=1$, and
    \[
        \adaptnorm{\frac{\eta_k}{\eta_{k+1}}K_k - T_k^{\ast}K_{k+1}W_k}^2 \leq \tilde{C}_k\norm{D}^2 \quad \text{for} \quad \tilde{C}_k \defeq  \sup_{\xi  \in \Omega\setminus\Omega_0} \left| \frac{\eta_k}{\eta_{k+1}} - c_{\xi} ,\right|^2
    \]
    where $\Omega_0 \defeq \{\xi \in \Omega \mid D\thisx(\xi)=0\}$.
\end{lemma}

\begin{proof}
    Let $\xi \in \Omega$ satisfy the antecedent of the $L^2$ characterisation \eqref{eq:TV_predictors:preservation:l2} of the strict variant of \eqref{eq:TV_predictors:preservation}.
    If $D \thisx(\xi) = 0$, the consequent of \eqref{eq:TV_predictors:preservation:l2} is clearly true for $\xi$ by construction. So suppose $D \thisx(\xi) \ne 0$.
    Since $\nexxt\dualpredict(\xi) =  T_k\thisy(\xi) = R_{\theta_{\xi}}(\thisy(\xi))$ and $D\nexxt\primalpredict(\xi) = c_{\xi} R_{\theta_{\xi}}(D\thisx(\xi))$  for a $c_{\xi} \geq 0$ such that $ |D\nexxt\primalpredict(\xi)| = c_{\xi}|D\thisx(\xi)|$, then, as required by \eqref{eq:TV_predictors:preservation:l2},
    \begin{equation*}
        \begin{aligned}
        \iprod{D\nexxt\primalpredict(\xi)}{\nexxt\dualpredict(\xi)}
            &
            = \iprod{c_{\xi} R_{\theta_{\xi}}(D\thisx(\xi))}{R_{\theta_{\xi}}(\thisy(\xi))}
            \\
            &
            = c_{\xi}\iprod{D\thisx(\xi)}{\thisy(\xi)}
            = c_{\xi}\alpha|D \thisx(\xi)|
            = \alpha|D\nexxt\primalpredict(\xi)|.
        \end{aligned}
    \end{equation*}
    Clearly $|\nexxt{\dualpredict(\xi)}| \le \alpha$, so the consequent of \eqref{eq:TV_predictors:preservation:l2} holds for $\xi$, as required. The same argumentation applied to analogously constructed true temporal couplings, proves \eqref{eq:TV_predictors:preservation:strong:coupling}.

    Computing for upper bounds of $\norm{W_k}^2$ and $\norm{T_k}^2$ is straightforward. Moreover, by the definition of predictors, for any $\xi \in \Omega$ with $D\thisx(\xi) \ne 0$, we have
    \[
        \begin{aligned}
        T_k^{\ast} D W_k \thisx (\xi)
        &
        =
        T_k^{\ast} D \nexxt\primalpredict(\xi)
        =
        T_k^{\ast}(c_{\xi} R_{\theta_{\xi}}(D\thisx(\xi))
        \\
        &
        =
        R_{-\theta_{\xi}}(c_{\xi} R_{\theta_{\xi}}(D\thisx(\xi))
        =
        c_{\xi} D\thisx(\xi).
        \end{aligned}
    \]
    On the other hand, if $D\thisx(\xi) = 0$, by the construction of $T_k^{\ast}$, clearly $T_k^{\ast} D W_k \thisx (\xi)=0$.
    Thus
    \begin{equation*}
        \begin{aligned}
        \adaptnorm{\frac{\eta_k}{\eta_{k+1}}K_k - T_k^{\ast}K_{k+1}W_k}^2
        &
        =
        \sup_{\|\thisx\| = 1} \int_{\Omega}\left|\frac{\eta_k}{\eta_{k+1}}D\thisx(\xi) - T_k^{\ast} D W_k \thisx (\xi) \right|^2 d\xi
        \\
        &
        =
        \sup_{\|\thisx\| = 1}\int_{\Omega \setminus \Omega_0}\left|\frac{\eta_k}{\eta_{k+1}}D\thisx(\xi) - c_{\xi} D\thisx(\xi)\right|^2 d\xi
        \leq
        \tilde{C}_k \|D\|^2. \qedhere
        \end{aligned}
    \end{equation*}
\end{proof}

\begin{remark}
If $\eta_k = \eta_{k+1}$ and $c_{\xi} = 1$ (i.e. $\norm{D\nexxt\primalpredict(\xi)} = \norm{D\thisx(\xi)}$) for all $\xi$, then we can choose $\tilde{C}_k = 0$.
\end{remark}

\subsubsection*{Global preservation for general operators}

We now study a primal-dual predictor that preserves total variation (defined with a general operator $D$) vectorwise.
To do so, we need a left-invertible modification $\breve D$ of the (discretised, differential) operator $D$.
We let $X_k$ be a Hilbert space, $Y_k \subset L^2(\Omega; \R^m)$, and $K_k = D \in \linear(X_k;Y_k)$ for all $k$ for an arbitrary domain $\Omega$, such that the $2,1$-norm is defined.

\begin{lemma}
    \label{lemma:predictors:global}
    Let $\thisx \in X_k, \thisy \in Y_k$,  $W_k \in \linear(X_k;X_{k+1})$, $D,\breve D \in \linear(X_k;Y_k)$.
    Suppose $(\nexxt\primalpredict,\nexxt\dualpredict) = P_k(\thisx,\thisy)$ where, for $Q_k \defeq  \norm{\breve D\nexxt\primalpredict}_{2,1} \norm{ D\thisx}_{2,1}^{-1} \Id \in \linear(Y_k; Y_k)$, the primal-dual predictor
    \begin{equation}
        \label{lemma:eq:TVpreserve_operator}
        P_k(\thisx,\thisy)= (W_k \thisx ,  \breve{D} \thisz ) \text{ for a } \thisz \in X_k \text{ satisfying } W_k^{\ast}\breve{D}^{\ast}\breve{D}\thisz = D^{\ast}Q_k\thisy
    \end{equation}

    This predictor is total variation preserving, i.e., satisfies \eqref{eq:TV_predictors:preservation}.
    If $\,W_k$ is invertible and $\range \breve{D}^*= \range \inv{(W_k^*)}D^*$, then we can write \eqref{lemma:eq:TVpreserve_operator} in explicit form as
    \[
        P_k(\thisx,\thisy)= \bigl(W_k \thisx ,  T_k \thisy \bigr) \text{ where } T_k =\breve{D} \inv{(\breve{D}^*\breve{D})} \inv{(W_k^*)} D^*Q_k .
    \]
    Furthermore,
    \begin{align*}
        \norm{T_k}^2
        &
        \le
        \frac{\norm{\breve{D} \inv{(\breve{D}^*\breve{D})} \inv{(W_k^*)} D^*} \norm{\breve D\nexxt\primalpredict}_{2,1}}{\norm{ D\thisx}_{2,1}}
        \quad \text{and}
        \\
        \adaptnorm{\frac{\eta_k}{\eta_{k+1}}K_k - T_k^{\ast}K_{k+1}W_k}^2
        &
        \le
        \tilde{C}_k \defeq \left\|\frac{\eta_k}{\eta_{k+1}}D -  Q_kD \inv{W_k} \inv{(\breve D^*\breve D)} \breve{D}^{\ast} D W_k \right\|^2.
    \end{align*}
\end{lemma}
\begin{proof}
    Using the definition of predictors, we have
    \begin{equation*}
        \begin{aligned}
        \iprod{ \breve{D} \nexxt\primalpredict } {\nexxt\dualpredict}
        &
        = \iprod{\breve{D} W_k \thisx } { \breve{D} \thisz }
        =  \iprod{ \thisx } {W_k^{\ast}\breve{D}^{\ast}\breve{D}\thisz }
        \\
        &
        =  \iprod{ \thisx } {D^{\ast}Q_k\thisy}
        =  \iprod{D  \thisx }{Q_k\thisy}
        =  \frac{\norm{\breve D\nexxt\primalpredict}_{2,1}}{ \norm{ D\thisx}_{2,1}}\iprod{D  \thisx }{\thisy}
        = \alpha \norm{\breve D\nexxt\primalpredict}_{2,1},
    \end{aligned}
    \end{equation*}
    where in the last equality we used $\alpha\norm{D\thisx}_{2,1}= \iprod{D\thisx}{\thisy}$. Now, let $W_k$ be invertible. Then using $\range \breve{D}^* = \range \inv{(W_k^*)}D^*$ and \eqref{lemma:eq:TVpreserve_operator} we have
    \[
    W_k^{\ast}\breve{D}^{\ast}\breve{D}\thisz = D^{\ast}Q_k\thisy \,
    \implies \,
    T_k =  \breve{D} \inv{(\breve{D}^*\breve{D})} \inv{(W_k^*)} D^*Q_k.
    \]
    With this, the bounds for $\norm{T_k}^2$ and $\norm{\frac{\eta_k}{\eta_{k+1}}K_k - T_k^{\ast}K_{k+1}W_k}^2$ follow immediately.
\end{proof}

\begin{example}
    \label{ex:predictors:d}
    For total variation, it is natural to take $D$ with (discrete) Neumann boundary conditions. This operator is not left-invertible ($\range D^* \ne X_k$): constant functions are in the kernel.
    To form a left-invertible $\breve D$, for example, with forward differences discretisation, it suffices to take it with (discrete) Dirichlet boundary conditions.
    If constant functions are an invariant subspace of $\inv{(W_k^*)}$, then also $D$ with forward differences discretisation and (discrete) Neumann boundary conditions satisfies the invertibility condition.
\end{example}

\begin{remark}
    If $D = \breve{D}$, then
    \[
        \adaptnorm{\frac{\eta_k}{\eta_{k+1}}K_k - T_k^{\ast}K_{k+1}W_k}^2
        =
        \adaptnorm{\frac{\eta_k}{\eta_{k+1}}D -  Q_kD}^2 =  \left|
            \frac{\eta_k}{\eta_{k+1}} -  \frac{\norm{\breve D\nexxt\primalpredict}_{2,1}}{ \norm{ D\thisx}_{2,1}}
        \right|^2\norm{D}^2.
    \]
    Thus, if  we further set $\eta_k = \eta_{k+1}$ and $W_k = \Id$ (i.e., there are no predictions), then we can take $\tilde{C}_k =0$.
\end{remark}

\subsection{Inner product preserving predictors}
\label{ssec:IP_predictors}

Instead of preserving the total variation after each prediction, we can impose that only the angles are preserved, that is,
\begin{equation}
    \label{eq:innerpreserve}
    \iprod{D \nexxt{\primalpredict}}{\nexxt{\dualpredict}} = \iprod{D \thisx}{\thisy}
\end{equation}
for all $k$.  We refer to such predictors as \emph{inner product preserving}. This is motivated by the fact that, in the static case, the inner product $\iprod{D \nexxt{x}}{\nexxt{y}}$ converges to the total variation. Thus, if we are given a primal predictor, we mitigate the error caused by the predictions by computing for a dual prediction such that equality in \eqref{eq:innerpreserve} is preserved.

We start by noting that the predictors from \cref{ssec:TV_predictors} preserve inner products by a simple rescaling of the dual predictors:
\begin{itemize}
    \item The predictors of \cref{lemma:predictors:pointwise-l2} with $\this{t}(\xi) = |\det \grad\thisv(\thisv(\xi))| \inv{(\grad \thisv({\xi}))}$.
    \item The rotating predictors of \cref{lemma:predictor-rotation} with $(T_k\thisy)(\xi) =  c_{\xi}^{-1}R_{\theta_{\xi}}(\thisy(\xi))$.
    \item The abstract predictors of \cref{lemma:predictors:global} with $Q_k = \Id$.
\end{itemize}
We omit the proofs and bounds, as we consider inner product preservation a weaker result than total variation preservation.

\subsubsection*{(Greedy) Component-wise preservation in finite dimensions}

An alternative approach to guarantee the preservation of the inner product is a component-wise update of the dual variable. This strategy directly enforces the point-wise equality $ \iprod{D \nexxt{\primalpredict}(\xi)}{\nexxt{\dualpredict}(\xi)} = \iprod{D \thisx(\xi)}{\thisy(\xi)}$ down to each individual component of $\xi$. Although straightforward, it does not consider the potential impact on other components, hence the prediction is greedy. However, as will be shown later, this method can still exhibit good numerical performance in practice.

\begin{lemma}
    \label{lemma:predictor-adhoc}
    Let $\thisx \in \R^n$, $\thisy \in \R^{m}$ and  $D \in \R^{m\times n}$. For $W_k \in \R^{n \times n}$, if $(\nexxt\primalpredict,\nexxt\dualpredict) = P_k(\thisx,\thisy) = (W_k\thisx, T_k\thisy)$ where the primal prediction $\nexxt\primalpredict = W_k\thisx$ and, for each component $i$, the dual prediction
    \begin{equation}
        \label{lemma:eq:innerpreserve_discrete}
            (T_k\thisy)_i =
            \begin{cases}
            \frac{(D\thisx)_i}{(D\nexxt\primalpredict)_i} \thisy_i,
            & |(D\nexxt\primalpredict)_i| > \epsilon
            \\
            \thisy_i,
            & |(D\nexxt\primalpredict)_i |\leq \epsilon
        \end{cases},
    \end{equation}
    for some tolerance $\epsilon > 0$,
    then $ (D\nexxt\primalpredict )_i {\nexxt\dualpredict}_i =  (D \thisx)_i \thisy_i$  whenever $|(D\nexxt\primalpredict)_i| > \epsilon$.  Moreover, if $(D\nexxt\primalpredict)_i \neq 0$ for all $i$, then
    \begin{gather*}
        \norm{T_k}^2 \leq  \max_{i} \frac{|(D\thisx)_i|^2}{|(D\nexxt\primalpredict)_i|^{2}}
    \shortintertext{and}
        \adaptnorm{\frac{\eta_k}{\eta_{k+1}}K_k - T_k^{\ast}K_{k+1}W_k}^2 \leq \tilde{C}_k\norm{D}^2
        \quad {for} \quad
        \tilde{C}_k \defeq  \left|\frac{\eta_k}{\eta_{k+1}} - 1 \right|^2.
    \end{gather*}
\end{lemma}

\begin{proof}
    It is straightforward to show that  $(D\nexxt\primalpredict )_i {\nexxt\dualpredict}_i =  (D \thisx)_i \thisy_i$ whenever $(D\nexxt\primalpredict)_i \neq 0$. Now, if $(D\nexxt\primalpredict)_i \neq 0$ for all $i$, then
    \begin{equation*}
        \begin{aligned}
        \adaptnorm{\frac{\eta_k}{\eta_{k+1}}K_k - T_k^{\ast}K_{k+1}W_k}^2
        &
        =
        \sup_{\|\thisx\| = 1} \adaptnorm{\frac{\eta_k}{\eta_{k+1}} D \thisx - D\thisx}^2
        \leq \tilde{C}_k\|D\|^2.\qedhere
        \end{aligned}
    \end{equation*}
The computation for the upper bound of $\norm{T_k}^2$ is straightforward.
\end{proof}
\begin{remark}
By simply setting $\eta_k = \eta_{k+1}$,  we can choose $\tilde{C}_k = 0$.
\end{remark}

\subsection{Dual scaling predictors}
\label{ssec:dualscaling}

We now consider another type of predictor that uses pointwise scaling in the dual variable.
For these predictors, we do not have any explicit preservation results, although, as we shall see, they perform numerically remarkably well.

\begin{lemma}
    \label{lemma:predictor-dualscaling}
    For a domain  $\Omega \subset \R^n$, consider  the subspaces $X_k \subset L^2(\Omega)$ and $Y_k \subset L^2(\Omega; \R^n)$ equipped with the $L^2$-norm.  Let $\thisx \in X_k$, $\thisy \in Y_k$, $K_k = D \in \linear(X_k; Y_k)$  a differential operator,  and  $(\nexxt\primalpredict,\nexxt\dualpredict) = P_k(\thisx,\thisy) = (W_k\thisx, T_k\thisy)$ for any primal predictor $W_k \in \linear(X_k; X_{k+1})$, and dual predictor $T_k \in \linear(Y_k;Y_{k+1})$, defined pointwise by
    \begin{gather}
        \label{lemma:eq:dualscaling}
        (T_k\thisy)(\xi) = c_k(\xi) \thisy(\xi).
    \shortintertext{where $c_k\in L^{\infty}(\Omega)$. Then}
        \nonumber
        \norm{T_k}^2 = \opt{C}_k
        \quad \text{and} \quad
        \adaptnorm{\frac{\eta_k}{\eta_{k+1}}K_k - T_k^{\ast}K_{k+1}W_k}^2  \leq 2\left( \tilde{C}_k + \bar{C}_k \|\Id-W_k\|^2\right)\|D\|^2
    \shortintertext{for}
        \nonumber
        \tilde{C}_k \defeq  \sup_{\xi \in \Omega} \left| \frac{\eta_k}{\eta_{k+1}} - c_k(\xi) \right|^2
        \quad \text{and} \quad
        \opt{C}_k \defeq \displaystyle \sup_{\xi \in \Omega} |c_k(\xi)|^2.
    \end{gather}
\end{lemma}
\begin{proof}
It is straightforward to verify that $\norm{T_k}^2 = \opt{C}_k$. Now,
   \begin{equation*}
    \begin{aligned}
    \adaptnorm{\frac{\eta_k}{\eta_{k+1}}K_k - T_k^{\ast}K_{k+1}W_k}^2
    &
    =
    \sup_{\|\thisx\| = 1} \int_{\Omega}\left|\frac{\eta_k}{\eta_{k+1}}D\thisx(\xi) - c_k(\xi) D W_k \thisx (\xi)\right|^2 d\xi
    \\
    &
    \leq
    2\sup_{\|\thisx\| = 1}  \int_{\Omega} \left(\left|\frac{\eta_k}{\eta_{k+1}} - c_k(\xi)\right|^2|D\thisx(\xi)|^2 + |c_k(\xi)|^2 | D(\Id -  W_k) \thisx (\xi) |^2 \right) d\xi
    \\
    &
    \leq
    2\bigl( \tilde{C}_k + \bar{C}_k \|\Id-W_k\|^2\bigr)\|D\|^2. \qedhere
    \end{aligned}
    \end{equation*}
\end{proof}

\begin{remark}
    \label{rem:dualscaling}
    If $\eta_{k} = \eta_{k+1}$,  and $|c_k(\xi)| \to 1$ for all $\xi$ as $k \to \infty$ whenever $W_k \to \Id$, then $\norm{\frac{\eta_k}{\eta_{k+1}}K_k - T_k^{\ast}K_{k+1}W_k} \to 0$.
\end{remark}

In \cref{sec:numerics}, we test in application a dual scaling predictor tailored in such a way that its scaling factors invalidate the dual at points where there are significant changes in the primal.

\section{Numerical experiments}
\label{sec:numerics}

We conducted experiments with our method in two applications: the image stabilisation (compare with \cite{tuomov-predict}) and the dynamic Positron Emission Tomography (PET) reconstruction.

In our image stabilisation experiment, we sequentially process highly noisy randomly displaced sub-images of a bigger image. For simplicity we assume to have access to noisy measurements of the displacements.
Leveraging these displacements in the primal predictors, our objective is to achieve real-time denoising of the displaced sub-images with total variation regularisation. In dynamic PET imaging, our focus shifts to an unknown density (image) undergoing rotational motion during the measurement process. Our task involves reconstructing this density using data obtained from its partial Radon transform, again subject to Poisson noise. We also assume access to noisy measurements of the rotation angle. If measurements of the displacement or the rotation is not otherwise available, they can be incorporated into our model following the optical flow displacement estimation approach in \cite{tuomov-predict}, for PET by adapting patient body motion tracking such as that in \cite{iwao2022brain}.

\subsection{Predictors}

For image stabilisation, the primal prediction is simply $\nexxt\primalpredict \defeq \thisx \circ \thisv$, where $\thisv$ represents the noisy measurement of an unknown displacement $\this\optv$ of the subimage for frame $k$. For the dynamic PET experiment, the primal prediction is $\nexxt\primalpredict(\xi) \defeq R_{\theta_\xi}^k(\thisx(\xi))$, where $R_{\theta_\xi}^k$ models the random rotational motion.
Based on the theoretical framework outlined in \cref{sec:predictors}, we evaluate several choices for the dual predictors:

\begin{itemize}
    \item \textit{Rotation}: A total variation preserving predictor that employs rotation on the dual variable as described in \cref{lemma:predictor-rotation}.
    \item \textit{Greedy}: An inner product preserving predictor by components as described in \cref{lemma:predictor-adhoc}.
    \item \textit{Strict Greedy}: A strict total variation preserving predictor as described in \cref{rem:predictors:strictgreedy}.
    \item \textit{Dual Scaling}: A predictor enforces component-wise scaling as described in \cref{lemma:predictor-dualscaling} with
    \[
       c_k(\xi) \defeq 1 - \chi_k \nu_k( \nextx_\delta(\xi) )  \quad \text{where} \quad \nextx_\delta(\xi) \defeq \frac{|\nexxt\primalpredict(\xi) - \thisx(\xi)|}{\max\lbrace 10^{-12}, \max_\xi |\nexxt\primalpredict(\xi) - \thisx(\xi)| \rbrace},
    \]
    for a scalar $\chi_k$ and an “activation” function $\nu_k$ satisfying $\nu_k(0) = 0$ and $\nu_k(1)=1$. For experiments where the test image or phantom has mostly flat regions (e.g., Shepp-Logan phantom below), we use $\chi_k = 1.0$ and $\nu_k (\cdot) = (1 + e^{-1000(\cdot-0.05)})^{-1}$. Otherwise (e.g., for lighthouse test image and brain phantom), we fix $\chi_k = 0.75$ and $\nu_k (\cdot) = 1 - \abs{\cdot-1}^{1/5}$.

    The idea is that the dual predictor attempts to move $\norm{\this{\breve y}(\xi_k)}$ towards either $\alpha$ or $0$ when the primal predictor has changed the corresponding primal variable. However, when the changes in the primal variable become small, the activation function also causes the changes in the dual variable to become small, allowing it to stabilise, and the prediction errors to become small via $\abs{c_k(\xi)} \to 1$ as in \cref{rem:dualscaling}.
\end{itemize}
We compare these proposed predictors to
\begin{itemize}
    \item \textit{No Prediction}, i.e., both primal and dual predictors are identity maps;
    \item \textit{Primal Only} predictor with identity dual prediction;
    \item \textit{Zero Dual} predictor, i.e., $\this{\breve y}=0$; and the
    \item \textit{Proximal} (old) dual predictor from \cite{tuomov-predict}.
\end{itemize}
For the latter, we use  $\tilde{G}_k^{\ast} = G_k^{\ast} + \frac{\tilde{\rho_k}}{2}\|\cdot\|^2_{L^2(\Omega)}$ in the proximal prediction step, and a ``phantom'' $\rho = 100$, as discussed in \cite{tuomov-predict}, to allow for larger dual step lengths.
The Zero Dual predictor is motivated by the staircasing effect of total variation: it promotes flat regions, where the dual variable can be close to zero. As we will see, it performs remarkably well when there is movement in the images, but fails to stabilise when the movement is stopped.

\def\denoisegreedy{lighthouse200x300_pdps_known_greedy_25_10000_100}
\def\denoisestrictgreedy{lighthouse200x300_pdps_known_strictgreedy_25_10000_100}
\def\denoiseprimalonly{lighthouse200x300_pdps_known_primalonly_25_10000_100}
\def\denoiserotation{lighthouse200x300_pdps_known_rotation_25_10000_100}
\def\denoiseproximal{lighthouse200x300_pdps_known_proximal_25_10000_100}
\def\denoisenoprediction{lighthouse200x300_pdps_known_noprediction_25_10000_100}
\def\denoisedualscaling{lighthouse200x300_pdps_known_dualscaling_25_10000_100}
\def\denoisezerodual{lighthouse200x300_pdps_known_zerodual_25_10000_100}

\subsection{Image stabilisation}

\begin{figure}[t]%
    \centering%
    \begin{subfigure}{0.33\textwidth}%
        \includegraphics[width=\textwidth]{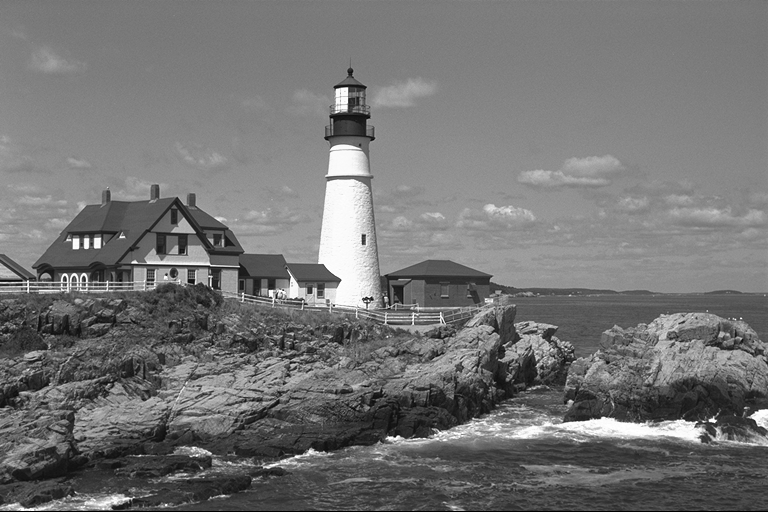}%
        \caption{Original}
        \label{fig:flow:lighthouse:tv:orig}
    \end{subfigure}\,%
    \begin{subfigure}{0.33\textwidth}%
        \includegraphics[width=\textwidth]{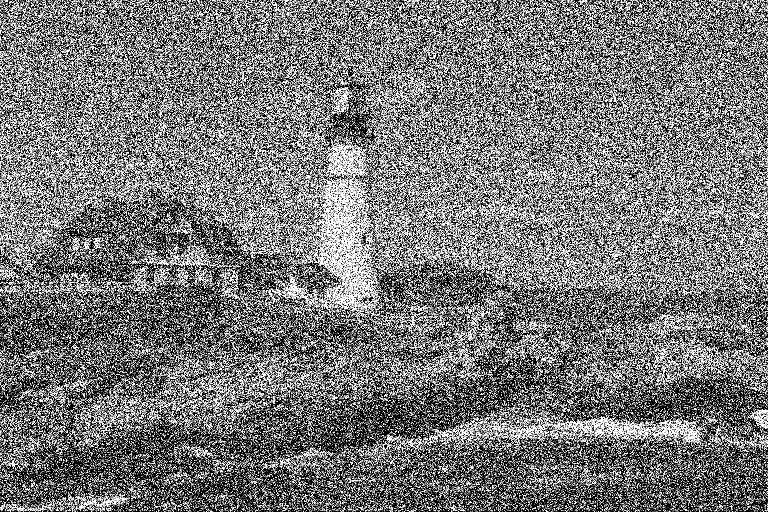}%
        \caption{Data: noise level $0.5$}
    \end{subfigure}\,%
    \begin{subfigure}{0.33\textwidth}%
        \includegraphics[width=\textwidth]{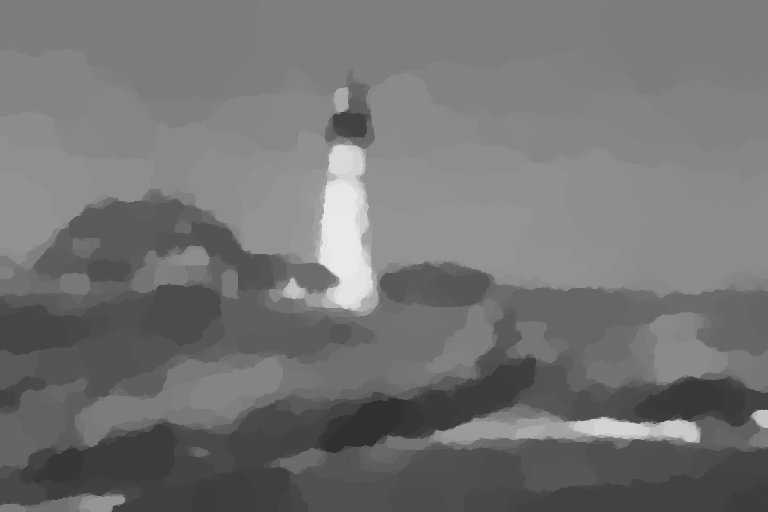}%
        \caption{Reconstruction: $\alpha=1.0$}
        \label{fig:flow:lighthouse:tv:known}
    \end{subfigure}%
    \caption{Test image, added noise, and stationary reconstruction for comparison.}
    \label{fig:flow:lighthouse}
\end{figure}

For image stabilisation, we assume to be given in each frame a noisy measurement $z_k \in X_k$ of a true image $\opt{z}_k \in X_k$ and a noisy measurement $\thisv \in V_k$ of a true displacement field $\this\optv \in V_k$. We require the measured displacement fields $\thisv$ to be bijective. The  finite-dimensional subspaces $X_k \subset L^2(\Omega)$, $Y_k \subset (\Omega; \R^2)$, and $V_k \subset L^2(\Omega; \Omega) \cap C^2(\Omega; \Omega)$ on a domain $\Omega \subset \R^2$ we equip with the $L^2$-norm. We set the objective functions $J_k$ in \eqref{eq:pd:problem} by taking
\[
   F_k(x) \defeq \frac{1}{2}\norm{x - z_k}^2, \quad E_k(x) \defeq 0,\quad  \text{and} \quad(G_k \circ K_k)(x) \defeq \alpha\norm{D_kx}_{2,1}
\]
where $D_k$ is a discretised differential operator, and $\alpha$ is the regularisation parameter for total variation.

\subsubsection*{Numerical setup and results}

\begin{table}
        \centering
        \begin{tabular}{lcccccc}
            \hline
            \multirow{2}{*}{Predictor} & \multicolumn{2}{c}{Average PSNR} & PSNR & \multicolumn{2}{c}{Average SSIM} & SSIM \\
             & iter 1 & iter 500 & 95\% CI & iter 1 & iter 500 & 95\% CI \\ \hline
            Dual Scaling & \textbf{22.6959} & \textbf{27.9238} & \textbf{27.0909--28.7567 }& \textbf{0.6697} & \textbf{0.8101} & \textbf{0.7973--0.8229} \\
            Greedy & 21.7029 & 26.5375 & 25.8964--27.1786 & 0.6509 & 0.7877 & 0.7740--0.8014 \\
            No Prediction & 19.9162 & 24.2983 & 23.3691--25.2275 & 0.6201 & 0.7629 & 0.7436--0.7822 \\
            Primal Only & 21.7029 & 26.5374 & 25.8963--27.1785 & 0.6509 & 0.7877 & 0.7740--0.8014 \\
            Proximal & 21.5815 & 26.3912 & 25.7342--27.0482 & 0.6537 & 0.7955 & 0.7808--0.8102 \\
            Rotation & 21.8185 & 26.6875 & 26.0523--27.3227 & 0.6588 & 0.7963 & 0.7834--0.8092 \\
            Strict Greedy & 21.6471 & 26.4771 & 25.8202--27.1340 & 0.6572 & 0.7989 & 0.7844--0.8134 \\
            Zero Dual & 21.9269 & 26.8247 & 26.2399--27.4095 & 0.5940 & 0.7012 & 0.6766--0.7258 \\
        \end{tabular}
        \caption{Average PSNR and SSIM for computational image stabilisation, from the indicated iteration until 10000 iterations. The confidence intervals (CI) are computed starting from the 500\textsuperscript{th} iteration.}
        \label{tab:denoising}
        \bigskip
        \begin{tabular}{lcccccc}
            \hline
            \multirow{2}{*}{Predictor} & \multicolumn{2}{c}{Average PSNR} & PSNR & \multicolumn{2}{c}{Average SSIM} & SSIM \\
             & iter 1 & iter 500 & 95\% CI & iter 1 & iter 500 & 95\% CI \\ \hline
            Dual Scaling & \textbf{14.7561} & \textbf{19.8815} & \textbf{19.2024--20.5606} & 0.5883 & \textbf{0.8309} & \textbf{0.8102--0.8516} \\
            Greedy & 14.7223 & 19.8249 & 19.1398--20.5100 & 0.5882 & 0.8285 & 0.8074--0.8496 \\
            No Prediction & 12.9669 & 18.0957 & 17.0723--19.1191 & 0.4391 & 0.7106 & 0.6606--0.7606 \\
            Primal Only & 14.7340 & 19.8448 & 19.1615--20.5281 & 0.5875 & 0.8286 & 0.8075--0.8497 \\
            Proximal & 14.7176 & 19.8137 & 19.1259--20.5015 & \textbf{0.5894} & 0.8278 & 0.8065--0.8491 \\
            Rotation & 14.7403 & 19.8539 & 19.1730--20.5348 & 0.5875 & 0.8291 & 0.8081--0.8501 \\
            Strict Greedy & 14.7175 & 19.8143 & 19.126--20.5026 & 0.5883 & 0.8262 & 0.8053--0.8471 \\
            Zero Dual & 14.5527 & 19.4534 & 18.9525--19.9543 & 0.5681 & 0.7888 & 0.7814--0.7962
        \end{tabular}
        \caption{Average PSNR and SSIM for dynamic PET reconstruction with Shepp-Logan phantom. from the indicated iteration until 4000 iterations. The confidence intervals (CI) are computed starting from the 500\textsuperscript{th} iteration.}
        \label{tab:PET}
        \bigskip
        \begin{tabular}{lcccccc}
            \hline
            \multirow{2}{*}{Predictor} & \multicolumn{2}{c}{Average PSNR} & PSNR & \multicolumn{2}{c}{Average SSIM} & SSIM \\
             & iter 1 & iter 500 & 95\% CI & iter 1 & iter 500 & 95\% CI \\ \hline
            Dual Scaling & \textbf{15.9222} & \textbf{21.1804} & \textbf{20.8011--21.5597} & \textbf{0.4720} & \textbf{0.6542} & \textbf{0.6353--0.6731} \\
            Greedy & 15.8040 & 20.9289 & 20.6396--21.2182 & 0.4677 & 0.6447 & 0.6292--0.6602 \\
            No Prediction & 14.2962 & 19.7518 & 19.1786--20.3250 & 0.4096 & 0.6021 & 0.5770--0.6272 \\
            Primal Only & 15.7985 & 20.9236 & 20.6335--21.2137 & 0.4672 & 0.6443 & 0.6287--0.6599 \\
            Proximal & 15.7758 & 20.8758 & 20.5797--21.1719 & 0.4667 & 0.6431 & 0.6274--0.6588 \\
            Rotation & 15.8037 & 20.9331 & 20.6444--21.2218 & 0.4673 & 0.6446 & 0.6291--0.6601 \\
            Strict Greedy & 15.7772 & 20.8803 & 20.5828--21.1778 & 0.4666 & 0.6432 & 0.6273--0.6591 \\
            Zero Dual & 15.8139 & 20.9511 & 20.6682--21.2340 & 0.4639 & 0.6372 & 0.6250--0.6494
        \end{tabular}
        \caption{Average PSNR and SSIM for dynamic PET reconstruction with brain phantom, from the indicated iteration until 4000 iterations. The confidence intervals (CI) are computed starting from the 500\textsuperscript{th} iteration.}
        \label{tab:PETb}
\end{table}

To evaluate our approach, we use as test image the lighthouse image (from the free Kodak image suite \cite{franzenkodak}) displayed in \cref{fig:flow:lighthouse} along with a noisy version and a single-frame total variation reconstruction for comparison. The original size is 768$\times$512 pixels. For our experiments, we pick a 300$\times$200 sub-image moving according to Brownian motion of standard deviation 2.
This motion is stopped on two subintervals (frames 2500--5000 and 8700--10000).
Thus the displacement fields $\this{\opt v}(\xi)=\xi + \this{d}$ with $\this{d} \in \R^2$ are constant in space. We added 50\% Gaussian noise (standard deviation 0.5 with original intensities in $[0, 1]$) to the sub-image. To construct the measured displacements available to the algorithm, we add 2.5\% Gaussian noise to the true displacements. It is worth mentioning that even on those intervals where the motion is stopped, the displacements available to the algorithm are still not necessarily zero due to introduced noise.

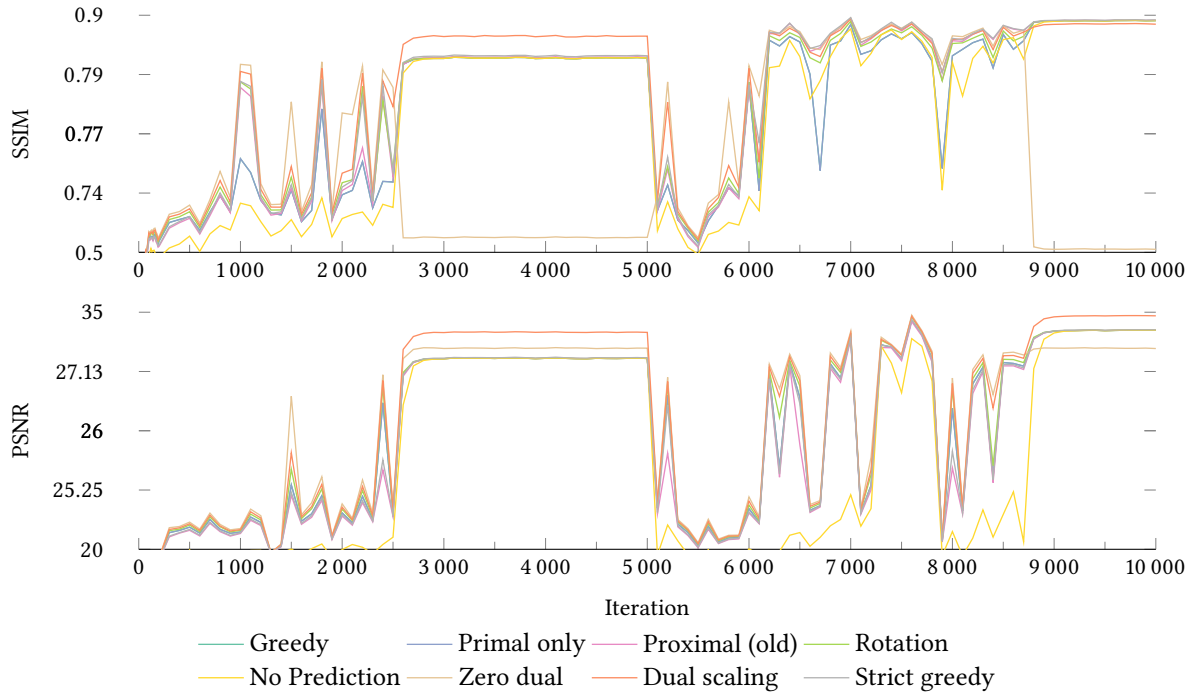
\begin{figure}[t!]
    \centering
    \tikzsetnextfilename{denoise_graphs}
    \tikzexternalenable
    \begin{tikzpicture}
        \def\yvalue{ssim}%
        \pgfplotsset{ylabel={SSIM},
                     every axis/.append style={yshift={-0.25\linewidth}}}%
        \def\ssimname{ssim}
\ifx\yvalue\ssimname
    \def\ymin{0.50}
    \def\yzoom{0.77}
    \def\ymax{0.9}
    \def\yprec{2}
\else
    \def\ymin{20}
    \def\yzoom{26}
    \def\ymax{35}
    \def\yprec{2}
\fi
\begin{axis}[%
    width={0.95\linewidth},
    height={0.3\linewidth},
    xmin=1,xmax=10000,
    scaled x ticks=false,
    x tick label style={/pgf/number format/fixed, /pgf/number format/set thousands separator={\,}},
    xminorticks=true,
    minor x tick num=1,
    axis x line*=bottom,
    legend style={legend pos=south east,inner sep=0pt,outer sep=0pt,legend cell align=left,align=left,draw=none,fill=none,font=\tiny, legend columns=1},
    zoomed = {\ymin}{\yzoom}{\ymax},
    tick label style={font=\footnotesize},
    label style={font=\footnotesize},
    legend style={font=\small},
    legend columns=4,
    legend to name={leg:denoise:\yvalue},
    ]

    \addplot [color=Set2-A, line width=0.5pt]
        table[x=iter,y=\yvalue]{img/\denoisegreedy.txt};
        \addlegendentry{Greedy}

    \addplot [color=Set2-C, line width=0.5pt]
        table[x=iter,y=\yvalue]{img/\denoiseprimalonly.txt};
        \addlegendentry{Primal only}

    \addplot [color=Set2-D, line width=0.5pt]
        table[x=iter,y=\yvalue]{img/\denoiseproximal.txt};
        \addlegendentry{Proximal (old)}

    \addplot [color=Set2-E, line width=0.5pt]
        table[x=iter,y=\yvalue]{img/\denoiserotation.txt};
        \addlegendentry{Rotation}

    \addplot [color=Set2-F, line width=0.5pt]
        table[x=iter,y=\yvalue]{img/\denoisenoprediction.txt};
        \addlegendentry{No Prediction}

    \addplot [color=Set2-G, line width=0.5pt]
        table[x=iter,y=\yvalue]{img/\denoisezerodual.txt};
        \addlegendentry{Zero dual}

    \addplot [color=Set2-B, line width=0.5pt]
        table[x=iter,y=\yvalue]{img/\denoisedualscaling.txt};
        \addlegendentry{Dual scaling}
        
    \addplot [color=Set2-H, line width=0.5pt]
        table[x=iter,y=\yvalue]{img/\denoisestrictgreedy.txt};
        \addlegendentry{Strict greedy}

\end{axis}%
        \def\yvalue{psnr}%
        \pgfplotsset{ylabel={PSNR},
                     every axis/.append style={yshift={-0.25\linewidth}}}%
        \pgfplotsset{xlabel=Iteration}
        \def\ssimname{ssim}
\ifx\yvalue\ssimname
    \def\ymin{0.50}
    \def\yzoom{0.77}
    \def\ymax{0.9}
    \def\yprec{2}
\else
    \def\ymin{20}
    \def\yzoom{26}
    \def\ymax{35}
    \def\yprec{2}
\fi
\begin{axis}[%
    width={0.95\linewidth},
    height={0.3\linewidth},
    xmin=1,xmax=10000,
    scaled x ticks=false,
    x tick label style={/pgf/number format/fixed, /pgf/number format/set thousands separator={\,}},
    xminorticks=true,
    minor x tick num=1,
    axis x line*=bottom,
    legend style={legend pos=south east,inner sep=0pt,outer sep=0pt,legend cell align=left,align=left,draw=none,fill=none,font=\tiny, legend columns=1},
    zoomed = {\ymin}{\yzoom}{\ymax},
    tick label style={font=\footnotesize},
    label style={font=\footnotesize},
    legend style={font=\small},
    legend columns=4,
    legend to name={leg:denoise:\yvalue},
    ]

    \addplot [color=Set2-A, line width=0.5pt]
        table[x=iter,y=\yvalue]{img/\denoisegreedy.txt};
        \addlegendentry{Greedy}

    \addplot [color=Set2-C, line width=0.5pt]
        table[x=iter,y=\yvalue]{img/\denoiseprimalonly.txt};
        \addlegendentry{Primal only}

    \addplot [color=Set2-D, line width=0.5pt]
        table[x=iter,y=\yvalue]{img/\denoiseproximal.txt};
        \addlegendentry{Proximal (old)}

    \addplot [color=Set2-E, line width=0.5pt]
        table[x=iter,y=\yvalue]{img/\denoiserotation.txt};
        \addlegendentry{Rotation}

    \addplot [color=Set2-F, line width=0.5pt]
        table[x=iter,y=\yvalue]{img/\denoisenoprediction.txt};
        \addlegendentry{No Prediction}

    \addplot [color=Set2-G, line width=0.5pt]
        table[x=iter,y=\yvalue]{img/\denoisezerodual.txt};
        \addlegendentry{Zero dual}

    \addplot [color=Set2-B, line width=0.5pt]
        table[x=iter,y=\yvalue]{img/\denoisedualscaling.txt};
        \addlegendentry{Dual scaling}
        
    \addplot [color=Set2-H, line width=0.5pt]
        table[x=iter,y=\yvalue]{img/\denoisestrictgreedy.txt};
        \addlegendentry{Strict greedy}

\end{axis}%
    \end{tikzpicture}
    \tikzexternaldisable%
    \pgfplotslegendfromname{leg:denoise:ssim}
    \caption{Iteration-wise SSIM and PSNR for the image stabilisation experiment.}%
    \label{fig:denoising:ssimpsnr}%
\end{figure}

\begin{figure}[tp]
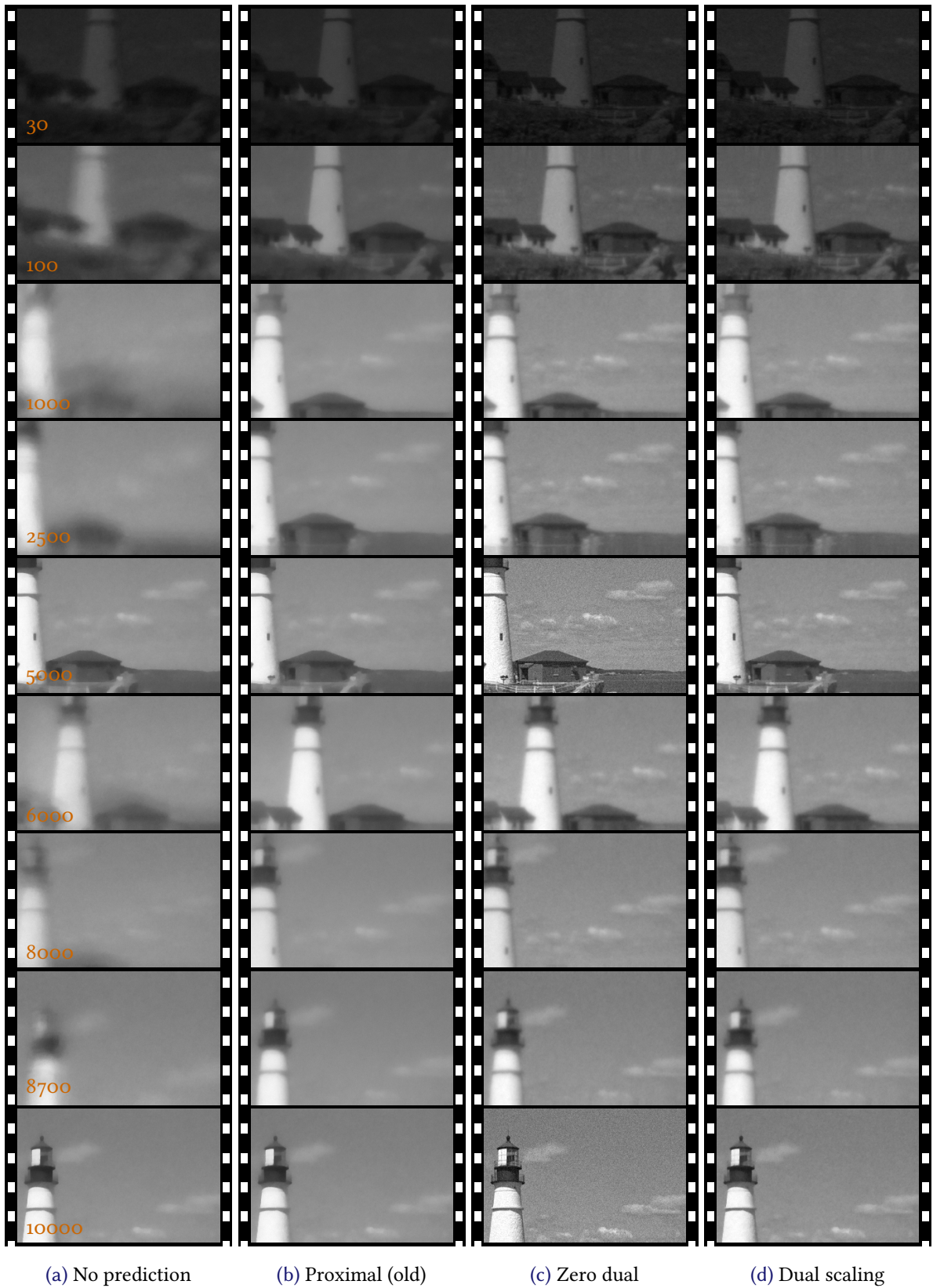
%
    \centering
    \dostripd{\denoisenoprediction}{\denoiseproximal}{\denoisezerodual}{\denoisedualscaling}
    \caption{Image stabilisation results for several predictors when $\alpha = 0.25$.
             The numbers on the left indicate the iteration/frame.}
    \label{fig:denoising:reco:l}
\end{figure}

We take the regularisation parameter $\alpha=0.25$. We set the following parameters for \cref{alg:pd:alg}:
\begin{itemize}[label={--},nosep]
    \item Step length parameter $\tau=0.01$, as well as $\Lambda=\Theta=1$.
    \item Primal strong convexity factor $\gamma=1$ and, generally, dual factor $\rho=0$.
    \item Maximal $\sigma$ and estimate $\norm{K_k} \le \sqrt{8}$ for forward-differences discretisation of $K_k=D$ with cell width $h=1$.
\end{itemize}

We always take zero as the initial iterate (primal and dual).

We implemented our algorithms in Julia, and performed our experiments on a mid-2022 MacBook Air with 16GB RAM and eight CPU cores. The implementation is available on Zenodo \cite{predict-code}.

\Cref{fig:denoising:reco:l} shows the reconstructed lighthouse subimages. The initial frames are darker and noisier due to the algorithm requiring a certain number of steps to stabilise. The reconstructions without any predictor (a) exhibit the poorest quality. The lighthouse and background features appear blurry and lack definition. The reconstruction generated by the Proximal (old) predictor of \cite{tuomov-proxtest} (b) appears visually distinct from the others. While it retains some level of detail, the overall image appears smoother and flatter. The reconstructions obtained using the Zero Dual (c) and the Dual Scaling predictors (d) demonstrate improved image quality in intervals where the object is not stable. These methods recover sharper edges and finer details, making the lighthouse structure and the clouds more prominent. Howeve, the Zero Dual still gives noisy reconstructions at intervals where the movement is stopped, compared to that of the Dual Scaling predictors.

To quantify these observations, \cref{fig:denoising:ssimpsnr} shows the plots of Structural Similarity Index (SSIM) and Peak Signal-to-Noise Ratio (PSNR) for each reconstructed frame. The Proximal (old) dual predictor exhibits lower SSIM and PSNR values compared to those of our proposed  Dual Scaling, Greedy, Strict Greedy, and Rotation dual predictors. Moreover, the reconstructions produced by the Dual Scaling predictor attained the highest SSIM and PSNR scores, even during intervals where motion is stopped, surpassing those of the No Prediction method.  During stabilisation, even though the true displacement fields are zeroed out, the displacement fields available to the algorithms are still subject to small noise, and the Dual Scaling predictor adapted well even with these small errors.
\Cref{tab:denoising} additionally shows average, SSIM and PSNR, as well as 95\% confidence intervals (CI) for them.
The Dual Scaling predictor consistently delivers high image quality scores.
The Zero Dual, however, delivers poor SSIM scores, especially on the intervals where the movement is stopped.

In summary, both visual inspection and quantitative metrics demonstrate the effectiveness of our proposed simplified \cref{alg:pd:alg} and several proposed predictors, compared to the original algorithm from \cite{tuomov-predict}, and to employing no prediction at all. This suggests that carefully designed dual predictors can play a crucial role in improving visual quality of the reconstructions.

\subsection{Dynamic positron emission tomography}

\begin{figure}[t!]
    \centering
    \begin{subfigure}{0.3\textwidth}
        \centering
        \includegraphics[width=\textwidth]{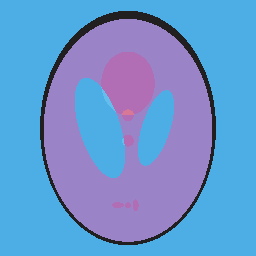}
        \caption{Test image}
    \end{subfigure}\hfill
    \raisebox{15\height}{$\to$}\hfill
    \begin{subfigure}{0.12\textwidth}
        \centering
        \includegraphics[width=\textwidth]{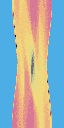}
        \caption{\centering True sinogram}
    \end{subfigure}\hfill
    \raisebox{15\height}{$\to$}\hfill
    \begin{subfigure}{0.12\textwidth}
        \centering
        \includegraphics[width=\textwidth]{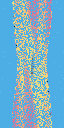}
       \caption{\centering Noisy sinogram}
    \end{subfigure}\hfill
    \raisebox{15\height}{$\to$}\hfill
    \begin{subfigure}{0.3\textwidth}
        \centering
        \includegraphics[width=\textwidth]{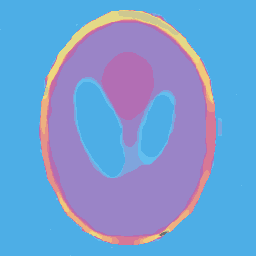}
       \caption{Reconstruction $\alpha = 1$}
    \end{subfigure}
    \caption{Shepp-Logan phantom, true sinogram, noisy subsampled sinogram, and static reconstruction. The colours represent values in $[0,1]$ as \protect\includegraphics[height=1.5ex]{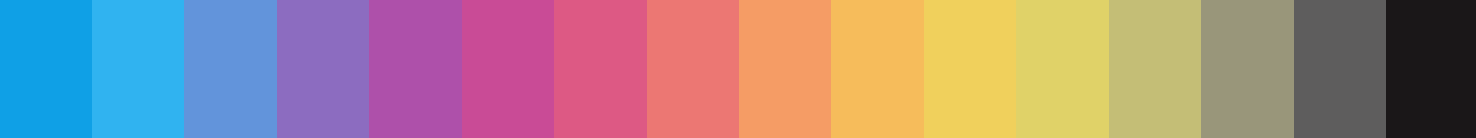}.}
   \label{fig:flow:shepplogan}
\end{figure}

\begin{figure}[t!]
    \centering
    \begin{subfigure}{0.3\textwidth}
        \centering
        \includegraphics[width=\textwidth]{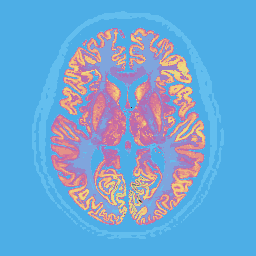}
        \caption{Test image}
    \end{subfigure}\hfill
    \raisebox{15\height}{$\to$}\hfill
    \begin{subfigure}{0.12\textwidth}
        \centering
        \includegraphics[width=\textwidth]{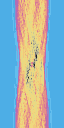}
        \caption{\centering True sinogram}
    \end{subfigure}\hfill
    \raisebox{15\height}{$\to$}\hfill
    \begin{subfigure}{0.12\textwidth}
        \centering
        \includegraphics[width=\textwidth]{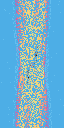}
       \caption{\centering Noisy sinogram}
    \end{subfigure}\hfill
    \raisebox{15\height}{$\to$}\hfill
    \begin{subfigure}{0.3\textwidth}
        \centering
        \includegraphics[width=\textwidth]{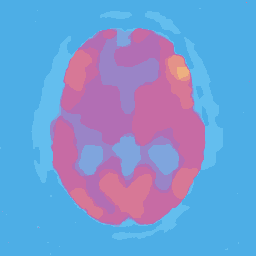}
       \caption{Reconstruction $\alpha = 1$}
    \end{subfigure}
    \caption{Brain phantom \cite{belzunce2020ultra}, true sinogram, noisy subsampled sinogram, and static reconstruction. The colours represent values in $[0,1]$ as \protect\includegraphics[height=1.5ex]{img/colorbar}.}
   \label{fig:flow:brainphantom}
\end{figure}

In dynamic PET, we assume to be given in each frame a noisy PET dataset $z_k$ in the data space obtained from a test image $\this\optx \in X_k$. We let assume to be given a noisy $\theta$ to form $R_{\theta}^k: X_k \to X_{k+1}$ that models the noisy rotational motion (by an angle $\theta$ about a perturbed center) of pixels between frames. The finite-dimensional subspaces $X_k \subset L^2(\Omega)$, $Y_k \subset (\Omega; \R^2)$, and $V \subset L^2(\Omega; \Omega) \cap C^2(\Omega; \Omega)$ on a domain $\Omega \subset \R^2$ equipped with the $L^2$-norm. We set the objective functions $J_k$ in \eqref{eq:pd:problem} by taking
\[
    F_k(x) \defeq \delta_{\R^+}, \quad E_k(x) \defeq \sum_{i=1}^n \left( [A_k\thisx]_i - [z_k]_i \log([A_k\thisx+c_k]_i)\right)  , \quad \text{and} \quad(G_k \circ K_k)(x) \defeq \alpha\norm{D_kx}_{2,1},
\]
where $A_k$ is the forward model based on a partial Radon transform and $c_k$ is a known vector with non-negative entries that models the expected number of background events.

\subsubsection*{Numerical setup and results}

\def\petSgreedy{shepplogan256x256_pdps_known_greedy_25_10000_9000}
\def\petSstrictgreedy{shepplogan256x256_pdps_known_strictgreedy_25_10000_9000}
\def\petSprimalonly{shepplogan256x256_pdps_known_primalonly_25_10000_9000}
\def\petSrotation{shepplogan256x256_pdps_known_rotation_25_10000_9000}
\def\petSproximal{shepplogan256x256_pdps_known_proximal_25_10000_9000}
\def\petSnoprediction{shepplogan256x256_pdps_known_noprediction_25_10000_9000}
\def\petSdualscaling{shepplogan256x256_pdps_known_dualscaling_25_10000_9000}
\def\petSzerodual{shepplogan256x256_pdps_known_zerodual_25_10000_9000}

\def\petBgreedy{brainphantom256x256_pdps_known_greedy_25_10000_9000}
\def\petBstrictgreedy{brainphantom256x256_pdps_known_strictgreedy_25_10000_9000}
\def\petBprimalonly{brainphantom256x256_pdps_known_primalonly_25_10000_9000}
\def\petBrotation{brainphantom256x256_pdps_known_rotation_25_10000_9000}
\def\petBproximal{brainphantom256x256_pdps_known_proximal_25_10000_9000}
\def\petBnoprediction{brainphantom256x256_pdps_known_noprediction_25_10000_9000}
\def\petBdualscaling{brainphantom256x256_pdps_known_dualscaling_25_10000_9000}
\def\petBzerodual{brainphantom256x256_pdps_known_zerodual_25_10000_9000}

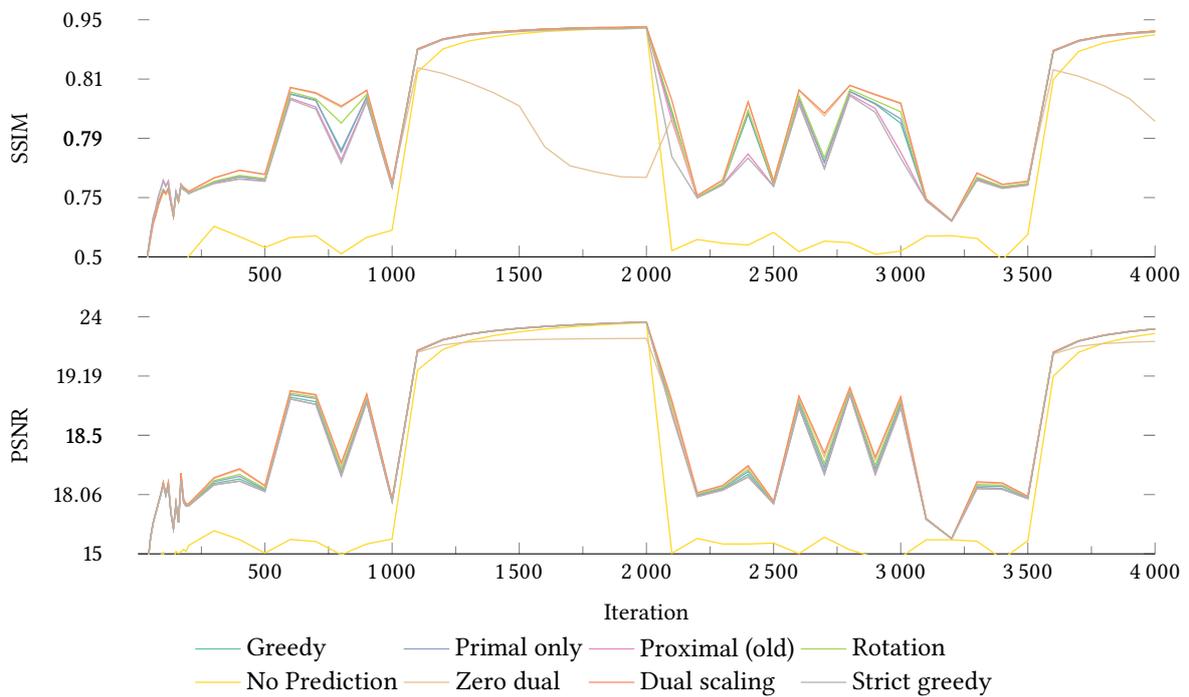
\begin{figure}[t!]
    \centering
    \tikzsetnextfilename{pet_graphs}
    \tikzexternalenable
    \begin{tikzpicture}
        \def\yvalue{ssim}%
        \pgfplotsset{ylabel={SSIM},
                     every axis/.append style={yshift={-0.25\linewidth}}}%
        \def\ssimname{ssim}
\ifx\yvalue\ssimname
    \def\ymin{0.5}
    \def\yzoom{0.79}
    \def\ymax{0.95}
    \def\yprec{2}
\else
    \def\ymin{15}
    \def\yzoom{18.5}
    \def\ymax{24}
    \def\yprec{2}
\fi
\begin{axis}[%
    width={0.95\linewidth},
    height={0.3\linewidth},
    xmin=1,xmax=4000,
    scaled x ticks=false,
    x tick label style={/pgf/number format/fixed, /pgf/number format/set thousands separator={\,}},
    xminorticks=true,
    minor x tick num=1,
    axis x line*=bottom,
    legend style={legend pos=north west,inner sep=0pt,outer sep=0pt,legend cell align=left,align=left,draw=none,fill=none,font=\tiny,legend columns=1}, 
    zoomed = {\ymin}{\yzoom}{\ymax},
    tick label style={font=\footnotesize},
    label style={font=\footnotesize},
    legend style={font=\small},
    legend columns=4,
    legend to name = {leg:pet:\yvalue},
    ]
    \addplot [color=Set2-A, line width=0.5pt]
        table[x=iter,y=\yvalue]{img/\petSgreedy.txt};
        \addlegendentry{Greedy}

    \addplot [color=Set2-C, line width=0.5pt]
        table[x=iter,y=\yvalue]{img/\petSprimalonly.txt};
        \addlegendentry{Primal only}

    \addplot [color=Set2-D, line width=0.5pt]
        table[x=iter,y=\yvalue]{img/\petSproximal.txt};
        \addlegendentry{Proximal (old)}

    \addplot [color=Set2-E, line width=0.5pt]
        table[x=iter,y=\yvalue]{img/\petSrotation.txt};
        \addlegendentry{Rotation}

    \addplot [color=Set2-F, line width=0.5pt]
        table[x=iter,y=\yvalue]{img/\petSnoprediction.txt};
        \addlegendentry{No Prediction}

    \addplot [color=Set2-G, line width=0.5pt]
        table[x=iter,y=\yvalue]{img/\petSzerodual.txt};
        \addlegendentry{Zero dual}

    \addplot [color=Set2-B, line width=0.5pt]
        table[x=iter,y=\yvalue]{img/\petSdualscaling.txt};
        \addlegendentry{Dual scaling}

    \addplot [color=Set2-H, line width=0.5pt]
        table[x=iter,y=\yvalue]{img/\petSstrictgreedy.txt};
        \addlegendentry{Strict greedy}

\end{axis}%
        \def\yvalue{psnr}%
        \pgfplotsset{ylabel={PSNR},
                     every axis/.append style={yshift={-0.25\linewidth}}}%
        \pgfplotsset{xlabel=Iteration}
        \def\ssimname{ssim}
\ifx\yvalue\ssimname
    \def\ymin{0.5}
    \def\yzoom{0.79}
    \def\ymax{0.95}
    \def\yprec{2}
\else
    \def\ymin{15}
    \def\yzoom{18.5}
    \def\ymax{24}
    \def\yprec{2}
\fi
\begin{axis}[%
    width={0.95\linewidth},
    height={0.3\linewidth},
    xmin=1,xmax=4000,
    scaled x ticks=false,
    x tick label style={/pgf/number format/fixed, /pgf/number format/set thousands separator={\,}},
    xminorticks=true,
    minor x tick num=1,
    axis x line*=bottom,
    legend style={legend pos=north west,inner sep=0pt,outer sep=0pt,legend cell align=left,align=left,draw=none,fill=none,font=\tiny,legend columns=1}, 
    zoomed = {\ymin}{\yzoom}{\ymax},
    tick label style={font=\footnotesize},
    label style={font=\footnotesize},
    legend style={font=\small},
    legend columns=4,
    legend to name = {leg:pet:\yvalue},
    ]
    \addplot [color=Set2-A, line width=0.5pt]
        table[x=iter,y=\yvalue]{img/\petSgreedy.txt};
        \addlegendentry{Greedy}

    \addplot [color=Set2-C, line width=0.5pt]
        table[x=iter,y=\yvalue]{img/\petSprimalonly.txt};
        \addlegendentry{Primal only}

    \addplot [color=Set2-D, line width=0.5pt]
        table[x=iter,y=\yvalue]{img/\petSproximal.txt};
        \addlegendentry{Proximal (old)}

    \addplot [color=Set2-E, line width=0.5pt]
        table[x=iter,y=\yvalue]{img/\petSrotation.txt};
        \addlegendentry{Rotation}

    \addplot [color=Set2-F, line width=0.5pt]
        table[x=iter,y=\yvalue]{img/\petSnoprediction.txt};
        \addlegendentry{No Prediction}

    \addplot [color=Set2-G, line width=0.5pt]
        table[x=iter,y=\yvalue]{img/\petSzerodual.txt};
        \addlegendentry{Zero dual}

    \addplot [color=Set2-B, line width=0.5pt]
        table[x=iter,y=\yvalue]{img/\petSdualscaling.txt};
        \addlegendentry{Dual scaling}

    \addplot [color=Set2-H, line width=0.5pt]
        table[x=iter,y=\yvalue]{img/\petSstrictgreedy.txt};
        \addlegendentry{Strict greedy}

\end{axis}%
    \end{tikzpicture}
    \tikzexternaldisable%
    \pgfplotslegendfromname{leg:pet:ssim}
    \caption{Iteration-wise SSIM and PSNR for PET with Shepp-Logan phantom.}%
    \label{fig:pet:ssimpsnr}%
\end{figure}

\begin{figure}[t!]
    \centering
    \tikzsetnextfilename{petb_graphs}
    \tikzexternalenable
    \begin{tikzpicture}
        \def\yvalue{ssim}%
        \pgfplotsset{ylabel={SSIM},
                     every axis/.append style={yshift={-0.25\linewidth}}}%
        \def\ssimname{ssim}
\ifx\yvalue\ssimname
    \def\ymin{0.5}
    \def\yzoom{0.6}
    \def\ymax{0.8}
    \def\yprec{2}
\else
    \def\ymin{18}
    \def\yzoom{20.5}
    \def\ymax{23}
    \def\yprec{2}
\fi
\begin{axis}[%
    width={0.95\linewidth},
    height={0.3\linewidth},
    xmin=1,xmax=4000,
    scaled x ticks=false,
    x tick label style={/pgf/number format/fixed, /pgf/number format/set thousands separator={\,}},
    xminorticks=true,
    minor x tick num=1,
    axis x line*=bottom,
    legend style={legend pos=north west,inner sep=0pt,outer sep=0pt,legend cell align=left,align=left,draw=none,fill=none,font=\tiny,legend columns=1}, 
    zoomed = {\ymin}{\yzoom}{\ymax},
    tick label style={font=\footnotesize},
    label style={font=\footnotesize},
    legend style={font=\small},
    legend columns=4,
    legend to name = {leg:petb:\yvalue},
    ]
    \addplot [color=Set2-A, line width=0.5pt]
        table[x=iter,y=\yvalue]{img/\petBgreedy.txt};
        \addlegendentry{Greedy}

    \addplot [color=Set2-C, line width=0.5pt]
        table[x=iter,y=\yvalue]{img/\petBprimalonly.txt};
        \addlegendentry{Primal only}

    \addplot [color=Set2-D, line width=0.5pt]
        table[x=iter,y=\yvalue]{img/\petBproximal.txt};
        \addlegendentry{Proximal (old)}

    \addplot [color=Set2-E, line width=0.5pt]
        table[x=iter,y=\yvalue]{img/\petBrotation.txt};
        \addlegendentry{Rotation}

    \addplot [color=Set2-F, line width=0.5pt]
        table[x=iter,y=\yvalue]{img/\petBnoprediction.txt};
        \addlegendentry{No Prediction}

    \addplot [color=Set2-G, line width=0.5pt]
        table[x=iter,y=\yvalue]{img/\petBzerodual.txt};
        \addlegendentry{Zero dual}

    \addplot [color=Set2-B, line width=0.5pt]
        table[x=iter,y=\yvalue]{img/\petBdualscaling.txt};
        \addlegendentry{Dual scaling}
        
    \addplot [color=Set2-H, line width=0.5pt]
        table[x=iter,y=\yvalue]{img/\petBstrictgreedy.txt};
        \addlegendentry{Strict greedy}

\end{axis}%
        \def\yvalue{psnr}%
        \pgfplotsset{ylabel={PSNR},
                     every axis/.append style={yshift={-0.25\linewidth}}}%
        \pgfplotsset{xlabel=Iteration}
        \def\ssimname{ssim}
\ifx\yvalue\ssimname
    \def\ymin{0.5}
    \def\yzoom{0.6}
    \def\ymax{0.8}
    \def\yprec{2}
\else
    \def\ymin{18}
    \def\yzoom{20.5}
    \def\ymax{23}
    \def\yprec{2}
\fi
\begin{axis}[%
    width={0.95\linewidth},
    height={0.3\linewidth},
    xmin=1,xmax=4000,
    scaled x ticks=false,
    x tick label style={/pgf/number format/fixed, /pgf/number format/set thousands separator={\,}},
    xminorticks=true,
    minor x tick num=1,
    axis x line*=bottom,
    legend style={legend pos=north west,inner sep=0pt,outer sep=0pt,legend cell align=left,align=left,draw=none,fill=none,font=\tiny,legend columns=1}, 
    zoomed = {\ymin}{\yzoom}{\ymax},
    tick label style={font=\footnotesize},
    label style={font=\footnotesize},
    legend style={font=\small},
    legend columns=4,
    legend to name = {leg:petb:\yvalue},
    ]
    \addplot [color=Set2-A, line width=0.5pt]
        table[x=iter,y=\yvalue]{img/\petBgreedy.txt};
        \addlegendentry{Greedy}

    \addplot [color=Set2-C, line width=0.5pt]
        table[x=iter,y=\yvalue]{img/\petBprimalonly.txt};
        \addlegendentry{Primal only}

    \addplot [color=Set2-D, line width=0.5pt]
        table[x=iter,y=\yvalue]{img/\petBproximal.txt};
        \addlegendentry{Proximal (old)}

    \addplot [color=Set2-E, line width=0.5pt]
        table[x=iter,y=\yvalue]{img/\petBrotation.txt};
        \addlegendentry{Rotation}

    \addplot [color=Set2-F, line width=0.5pt]
        table[x=iter,y=\yvalue]{img/\petBnoprediction.txt};
        \addlegendentry{No Prediction}

    \addplot [color=Set2-G, line width=0.5pt]
        table[x=iter,y=\yvalue]{img/\petBzerodual.txt};
        \addlegendentry{Zero dual}

    \addplot [color=Set2-B, line width=0.5pt]
        table[x=iter,y=\yvalue]{img/\petBdualscaling.txt};
        \addlegendentry{Dual scaling}
        
    \addplot [color=Set2-H, line width=0.5pt]
        table[x=iter,y=\yvalue]{img/\petBstrictgreedy.txt};
        \addlegendentry{Strict greedy}

\end{axis}%
    \end{tikzpicture}
    \tikzexternaldisable%
    \pgfplotslegendfromname{leg:petb:ssim}
    \caption{Iteration-wise SSIM and PSNR for PET with brain phantom.}%
    \label{fig:petb:ssimpsnr}%
\end{figure}
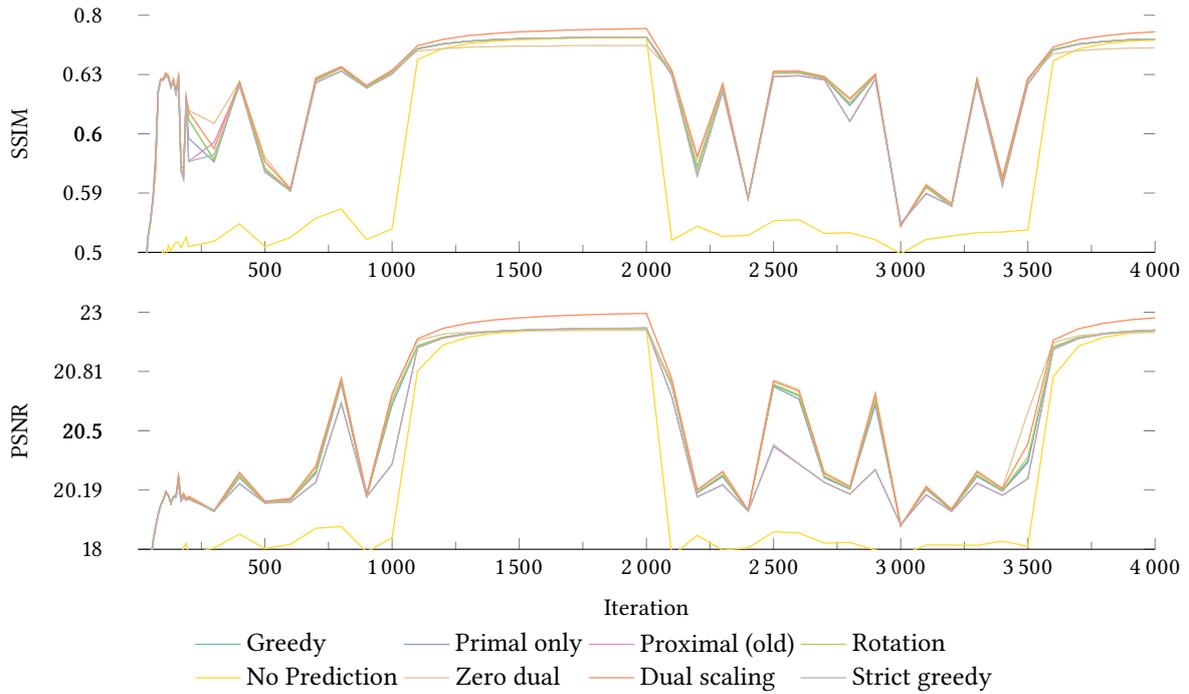

In our experiments, we use the Shepp-Logan phantom and the high-resolution brain phantom of \cite{belzunce2020ultra} to generate true and noisy PET datasets, examples of which are displayed in \cref{fig:flow:shepplogan,fig:flow:shepplogan}. We set both phantoms to have a resolution of 256$\times$256 pixels, while the PET datasets are organised into a sinogram of resolution 128$\times$64. To simulate the random nature of positron emissions, hence PET measurements, we further randomly subsample the sinogram down to 50\%.
To simulate motion, we rotate the phantom around a randomly chosen axis that shifts from the phantom's centre by a standard deviation of 1. The rotation angles themselves are chosen from a separate Gaussian distribution with a standard deviation of 0.15 radians.
We introduce Poisson noise with mean parameter of 0.5 to the PET dataset. To construct the simulated displacement measurements, that are available to the algorithms, we introduced Gaussian noise with a standard deviation of 0.035 radians to the rotation angles and Gaussian noise with a standard deviation of 0.25 to the centre of rotation. The movement is stopped on two subintervals (frames 1000--2000 and 3500--4000). We note again that even in these intervals, the displacements made available to the algorithm are still not necessarily zero due to introduced noise.

Taking the regularisation parameter $\alpha=0.25$, for \cref{alg:pd:alg} we take:
\begin{itemize}[label={--},nosep]
    \item Step length parameter $\tau=0.003$, as well as $\Lambda=\Theta=\kappa=1$.
    \item Primal strong convexity factor $\gamma=1$ and generally the dual factor $\rho=0$.
    \item Maximal $\sigma$ with the estimate $\norm{K_k} \le \sqrt{8}$ when $K_k=D$ is the forward differences operator with cell width $h=1$ \cite{chambolle2004algorithm}.
    \item We also fix $L = 300$, experimentally determined to be an upper bound over all iterations $k$ for
    \[
        L_{k+1} = \max \left\{L_k, 0.9 \frac{\norm{\grad E_{k}(\thisx) - \grad E_{k}( \this\primalpredict)}}{\norm{\thisx- \this\primalpredict}}\right\}.
    \]
\end{itemize}

\Cref{fig:pet:reco:l} shows the PET reconstructions of the Shepp-Logan phantom. The initial frames are omitted as they exhibit noisy appearance due to the algorithm requiring a certain number of steps to stabilise. The performance without any predictor (a), is clearly the poorest, producing blurry artefacts around the moving parts. The differences between the other reconstructions are less noticeable. However, a closer inspection reveals that the oval shapes appear more distinct in the reconstructions obtained using the Rotation (c) and Dual Scaling (d) predictors. They appear to be more resolved and separated compared to the reconstructions that used the Proximal (old) predictor (b). Additionally, the intensity levels within the two elliptical shapes more closely resemble those observed in the static reconstruction in \cref{fig:flow:shepplogan}.

\Cref{fig:petb:reco:l} displays the PET reconstructions of the brain phantom. The initial frames are omitted for similar reasons. Without any predictor, the results are again worst for intervals where the phantom is rotating. Although the differences between the other reconstructions are less pronounced, a closer look reveals the subtle differences in intensities especially in the yellow region. When there is no motion, the reconstruction is best with the Dual Scaling predictor, and noisiest with the Zero Dual.

\Cref{fig:pet:ssimpsnr,fig:petb:ssimpsnr} show the evolution of the SSIM and PSNR image quality metrics for the Shepp-Logan and brain phantoms, respectively. In both cases, the Proximal (old) predictor achieves lower SSIM and PSNR values compared to our proposed Greedy, Rotation, and Dual Scaling predictors. The image quality scores of brain phantom reconstructions with Strict Greedy predictor are better than those of the Proximal (old) predictor. Again, the reconstructions produced by the Dual Scaling predictor attained the highest SSIM and PSNR scores on intervals where rotational motion is stopped. Despite the true angles of rotation being zeroed out, the Dual Scaling predictor effectively adjusts for small errors in the available angles of rotation.
\Cref{tab:PET,tab:PETb} additionally show average, SSIM and PSNR, as well as 95\% confidence intervals (CI) for them.
The Dual Scaling predictor has consistently the best performance according to these metrics. The Zero Dual, however, fares poorly on the SSIM.

\subsection*{Conclusions}

The theory and the overall algorithm that we presented are simpler than the earlier approach of \cite{tuomov-predict}.
Moreover, the new predictors---justified by the theory---provide better numerical results. So far, we have, however, only treated relatively simple linear inverse problems with simple temporal characteristics.
The next step will be: nonlinear inverse problems with realtime PDE solution based on the techniques of \cite{jensen2022nonsmooth}.

\begin{figure}[tp]%
    \centering
    \dostrip{\petSnoprediction}{\petSproximal}{\petSzerodual}{\petSdualscaling}
    \caption{%
        Shepp–Logan reconstructions for several predictors.
        The colours represent values in $[0, 1]$ as \protect\includegraphics[height=1.5ex]{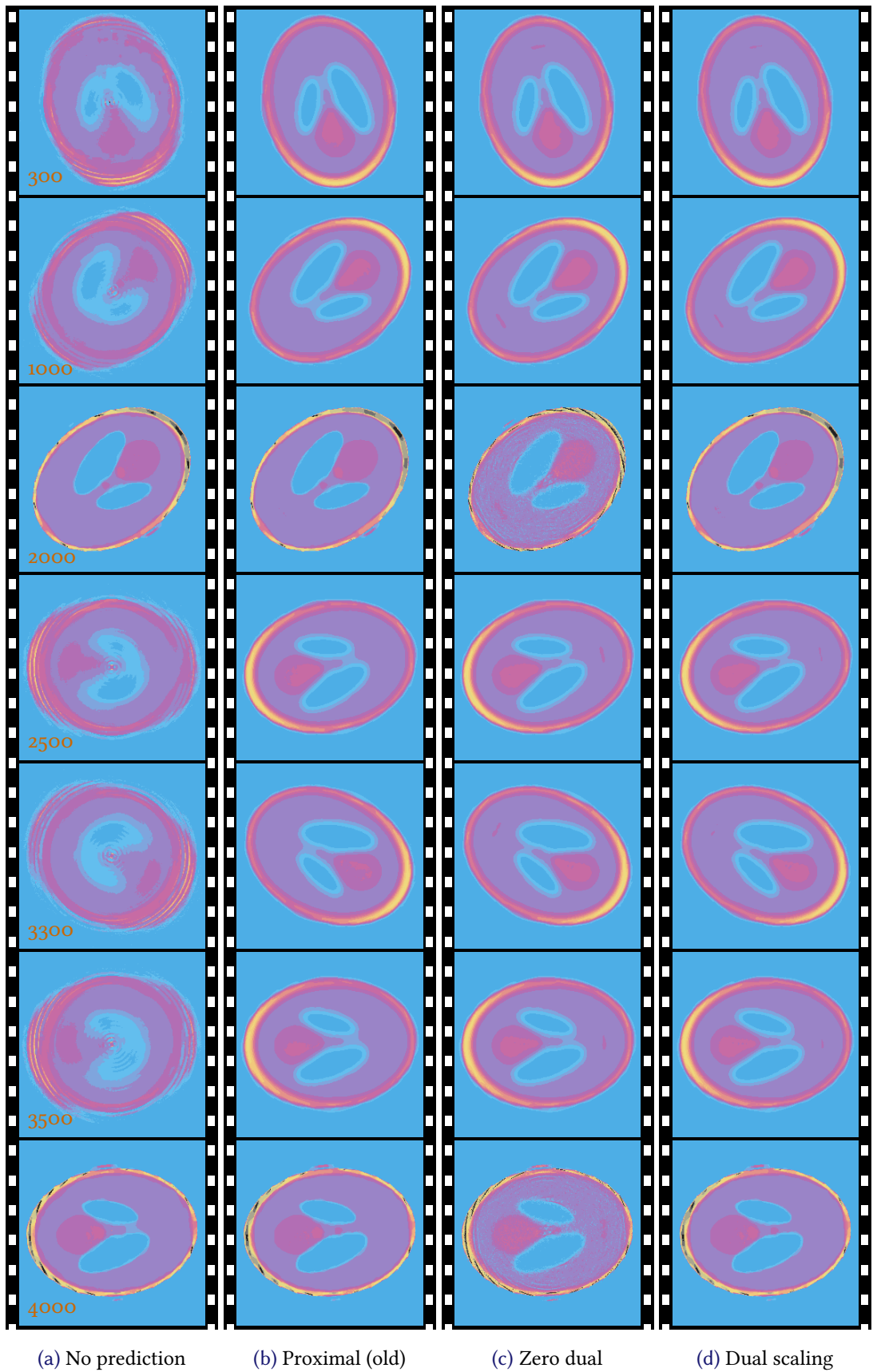}.
        The numbers on the left indicate the iteration.
    }
    \label{fig:pet:reco:l}
\end{figure}

\begin{figure}[tp]%
    \centering
    \dostrip{\petBnoprediction}{\petBproximal}{\petBzerodual}{\petBdualscaling}
    \caption{%
        Brain phantom reconstructions for several predictors.
        The colours represent values in $[0, 1]$ as \protect\includegraphics[height=1.5ex]{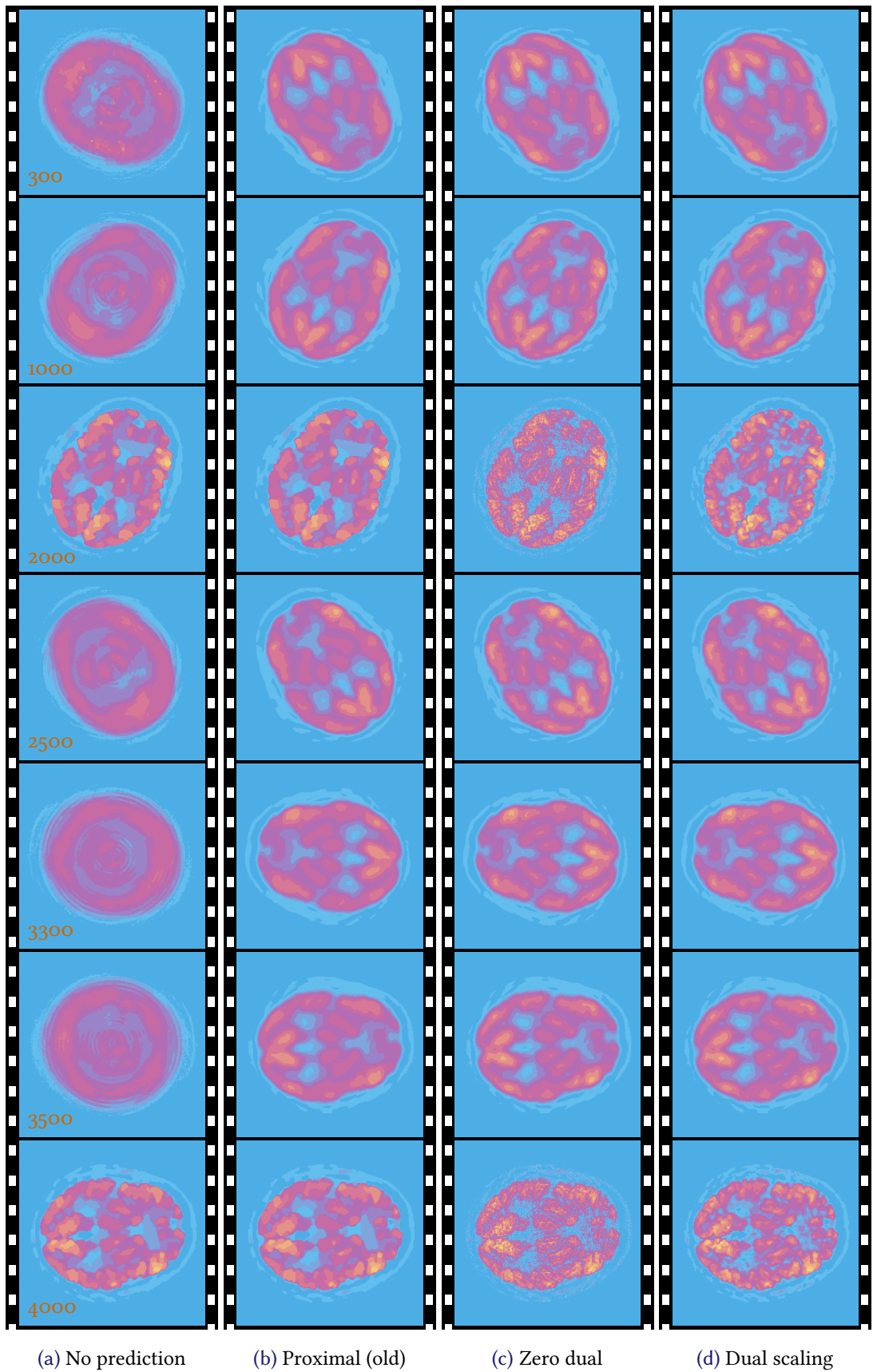}.
        The numbers on the left indicate the iteration.
    }
     \label{fig:petb:reco:l}
\end{figure}

\afterpage{\clearpage}

\appendix


\begin{thebibliography}{10}

\bibitem{belmega2018online}
E.\,V{.\nobreak\kern 0.33333em}Belmega, P{.\nobreak\kern
  0.33333em}Mertikopoulos, R{.\nobreak\kern 0.33333em}Negrel, and
  L{.\nobreak\kern 0.33333em}Sanguinetti, Online convex optimization and
  no-regret learning: Algorithms, guarantees and applications, 2018,
  \href{https://arxiv.org/abs/804.04529}{\nolinkurl{arXiv:804.04529}}.

\bibitem{belzunce2020ultra}
M.\,A{.\nobreak\kern 0.33333em}Belzunce and A.\,J{.\nobreak\kern
  0.33333em}Reader, Technical note: ultra high-resolution radiotracer-specific
  digital pet brain phantoms based on the BigBrain atlas, \emph{Medical
  Physics} 47 (2020),  3356--3362,
  \href{https://dx.doi.org/10.1002/mp.14218}{\nolinkurl{doi:10.1002/mp.14218}}.

\bibitem{tuomov-phaserec}
M{.\nobreak\kern 0.33333em}Benning, L{.\nobreak\kern 0.33333em}Gladden,
  D{.\nobreak\kern 0.33333em}Holland, C.\,B{.\nobreak\kern
  0.33333em}Sch{\"o}nlieb, and T{.\nobreak\kern 0.33333em}Valkonen, Phase
  reconstruction from velocity-encoded {MRI} measurements -- {A} survey of
  sparsity-promoting variational approaches, \emph{Journal of Magnetic
  Resonance} 238 (2014),  26--43,
  \href{https://dx.doi.org/10.1016/j.jmr.2013.10.003}{\nolinkurl{doi:10.1016/j.jmr.2013.10.003}}.

\bibitem{bernstein2019online}
A{.\nobreak\kern 0.33333em}Bernstein, E{.\nobreak\kern 0.33333em}Dall'Anese,
  and A{.\nobreak\kern 0.33333em}Simonetto, Online primal-dual methods with
  measurement feedback for time-varying convex optimization, \emph{IEEE
  Transactions on Signal Processing} 67 (2019),  1978--1991.

\bibitem{bousse2016maximum}
A{.\nobreak\kern 0.33333em}Bousse, O{.\nobreak\kern 0.33333em}Bertolli,
  D{.\nobreak\kern 0.33333em}Atkinson, S{.\nobreak\kern 0.33333em}Arridge,
  S{.\nobreak\kern 0.33333em}Ourselin, B.\,F{.\nobreak\kern 0.33333em}Hutton,
  and K{.\nobreak\kern 0.33333em}Thielemans, Maximum-Likelihood Joint Image
  Reconstruction/Motion Estimation in Attenuation-Corrected Respiratory Gated
  {PET/CT} Using a Single Attenuation Map, \emph{IEEE Transactions on Medical
  Imaging} 35 (2016),  217--228,
  \href{https://dx.doi.org/10.1109/TMI.2015.2464156}{\nolinkurl{doi:10.1109/tmi.2015.2464156}}.

\bibitem{burger2017variational}
M{.\nobreak\kern 0.33333em}Burger, H{.\nobreak\kern 0.33333em}Dirks,
  L{.\nobreak\kern 0.33333em}Frerking, A{.\nobreak\kern 0.33333em}Hauptmann,
  T{.\nobreak\kern 0.33333em}Helin, and S{.\nobreak\kern 0.33333em}Siltanen, A
  variational reconstruction method for undersampled dynamic x-ray tomography
  based on physical motion models, \emph{Inverse Problems} 33 (2017),  124008.

\bibitem{chambolle2004algorithm}
A{.\nobreak\kern 0.33333em}Chambolle, An algorithm for total variation
  minimization and applications, \emph{Journal of Mathematical Imaging and
  Vision} 20 (2004),  89--97,
  \href{https://dx.doi.org/10.1023/B:JMIV.0000011325.36760.1e}{\nolinkurl{doi:10.1023/b:jmiv.0000011325.36760.1e}}.

\bibitem{chambolle2010first}
A{.\nobreak\kern 0.33333em}Chambolle and T{.\nobreak\kern 0.33333em}Pock, A
  first-order primal-dual algorithm for convex problems with applications to
  imaging, \emph{Journal of Mathematical Imaging and Vision} 40 (2011),
  120--145,
  \href{https://dx.doi.org/10.1007/s10851-010-0251-1}{\nolinkurl{doi:10.1007/s10851-010-0251-1}}.

\bibitem{chang2021online}
T.\,J{.\nobreak\kern 0.33333em}Chang and S{.\nobreak\kern
  0.33333em}Shahrampour, On online optimization: Dynamic regret analysis of
  strongly convex and smooth problems, in \emph{Proceedings of the AAAI
  Conference on Artificial Intelligence}, volume~35, 2021,  6966--6973.

\bibitem{clasonvalkonen2020nonsmooth}
C{.\nobreak\kern 0.33333em}Clason and T{.\nobreak\kern 0.33333em}Valkonen,
  Introduction to Nonsmooth Analysis and Optimization, 2020,
  \href{https://arxiv.org/abs/2001.00216}{\nolinkurl{arXiv:2001.00216}}.
\newblock Work in progress.

\bibitem{franzenkodak}
R{.\nobreak\kern 0.33333em}Franzen, Kodak lossless true color image suite,
  PhotoCD PCD0992. Lossless, true color images released by the Eastman Kodak
  Company, 1999, \url{http://r0k.us/graphics/kodak/}.

\bibitem{hall13dynamical}
E{.\nobreak\kern 0.33333em}Hall and R{.\nobreak\kern 0.33333em}Willett,
  Dynamical models and tracking regret in online convex programming, in
  \emph{Proceedings of the 30th International Conference on Machine Learning},
  S{.\nobreak\kern 0.33333em}Dasgupta and D{.\nobreak\kern 0.33333em}McAllester
  (eds.), volume~28 of Proceedings of Machine Learning Research, PMLR, Atlanta,
  Georgia, USA, 2013,  579--587,
  \url{http://proceedings.mlr.press/v28/hall13.html}.

\bibitem{hazan2016introduction}
E{.\nobreak\kern 0.33333em}Hazan et~al., Introduction to online convex
  optimization, \emph{Foundations and Trends in Optimization} 2 (2016),
  157--325.

\bibitem{he2012convergence}
B{.\nobreak\kern 0.33333em}He and X{.\nobreak\kern 0.33333em}Yuan, Convergence
  analysis of primal-dual algorithms for a saddle-point problem: from
  contraction perspective, \emph{SIAM Journal on Imaging Sciences} 5 (2012),
  119--149,
  \href{https://dx.doi.org/10.1137/100814494}{\nolinkurl{doi:10.1137/100814494}}.

\bibitem{holland2010reducing}
D.\,J{.\nobreak\kern 0.33333em}Holland, D.\,M{.\nobreak\kern
  0.33333em}Malioutov, A{.\nobreak\kern 0.33333em}Blake, A.\,J{.\nobreak\kern
  0.33333em}Sederman, and L.\,F{.\nobreak\kern 0.33333em}Gladden, Reducing data
  acquisition times in phase-encoded velocity imaging u sing compressed
  sensing, \emph{Journal of Magnetic Resonance} 203 (2010),  236--46.

\bibitem{hunt2014weighing}
A{.\nobreak\kern 0.33333em}Hunt, Weighing without touching: applying electrical
  capacitance tomography to mass flowrate measurement in multiphase flows,
  \emph{Measurement and Control} 47 (2014),  19--25,
  \href{https://dx.doi.org/10.1177/0020294013517445}{\nolinkurl{doi:10.1177/0020294013517445}}.

\bibitem{iwao2022brain}
Y{.\nobreak\kern 0.33333em}Iwao, G{.\nobreak\kern 0.33333em}Akamatsu,
  H{.\nobreak\kern 0.33333em}Tashima, M{.\nobreak\kern 0.33333em}Takahashi, and
  T{.\nobreak\kern 0.33333em}Yamaya, Brain PET motion correction using 3D
  face-shape model: the first clinical study, \emph{Annals of Nuclear Medicine}
  36 (2022),  904--912.

\bibitem{jensen2022nonsmooth}
B{.\nobreak\kern 0.33333em}Jensen and T{.\nobreak\kern 0.33333em}Valkonen, A
  nonsmooth primal-dual method with interwoven {PDE} constraint solver,
  \emph{Computational Optimization and Appplications}  (2024),
  \href{https://dx.doi.org/10.1007/s10589-024-00587-3}{\nolinkurl{doi:10.1007/s10589-024-00587-3}},
  \href{https://arxiv.org/abs/2211.04807}{\nolinkurl{arXiv:2211.04807}}.

\bibitem{kar2010gossip}
S{.\nobreak\kern 0.33333em}Kar and J.\,M{.\nobreak\kern 0.33333em}Moura, Gossip
  and distributed Kalman filtering: Weak consensus under weak detectability,
  \emph{IEEE Transactions on Signal Processing} 59 (2010),  1766--1784.

\bibitem{lipponen2011nonstationary}
A{.\nobreak\kern 0.33333em}Lipponen, A{.\nobreak\kern 0.33333em}Seppänen, and
  J.\,P{.\nobreak\kern 0.33333em}Kaipio, Nonstationary approximation error
  approach to imaging of three-dimensional pipe flow: experimental evaluation,
  \emph{Measurement Science and Technology} 22 (2011),  104013,
  \href{https://dx.doi.org/10.1088/0957-0233/22/10/104013}{\nolinkurl{doi:10.1088/0957-0233/22/10/104013}}.

\bibitem{natterer2001mathematics}
F{.\nobreak\kern 0.33333em}Natterer, \emph{The Mathematics of Computerized
  Tomography}, Society for Industrial and Applied Mathematics, 2001,
  \href{https://dx.doi.org/10.1137/1.9780898719284}{\nolinkurl{doi:10.1137/1.9780898719284}}.

\bibitem{Nonhoff2020945}
M{.\nobreak\kern 0.33333em}Nonhoff and M.\,A{.\nobreak\kern 0.33333em}Müller,
  Online Gradient Descent for Linear Dynamical Systems,
  \emph{IFAC-PapersOnLine} 53 (2020),  945--952,
  \href{https://dx.doi.org/https://doi.org/10.1016/j.ifacol.2020.12.1258}{\nolinkurl{doi:https://doi.org/10.1016/j.ifacol.2020.12.1258}}.
\newblock 21st IFAC World Congress.

\bibitem{olfati2007consensus}
R{.\nobreak\kern 0.33333em}Olfati-Saber, J.\,A{.\nobreak\kern 0.33333em}Fax,
  and R.\,M{.\nobreak\kern 0.33333em}Murray, Consensus and cooperation in
  networked multi-agent systems, \emph{Proceedings of the IEEE} 95 (2007),
  215--233.

\bibitem{orabona2020modern}
F{.\nobreak\kern 0.33333em}Orabona, A Modern Introduction to Online Learning,
  2020, \href{https://arxiv.org/abs/1912.13213}{\nolinkurl{arXiv:1912.13213}}.

\bibitem{simonetto2018dual}
A{.\nobreak\kern 0.33333em}Simonetto, Dual prediction--correction methods for
  linearly constrained time-varying convex programs, \emph{IEEE Transactions on
  Automatic Control} 64 (2018),  3355--3361.

\bibitem{simonetto2020time}
A{.\nobreak\kern 0.33333em}Simonetto, E{.\nobreak\kern 0.33333em}Dall'Anese,
  S{.\nobreak\kern 0.33333em}Paternain, G{.\nobreak\kern 0.33333em}Leus, and
  G.\,B{.\nobreak\kern 0.33333em}Giannakis, Time-varying convex optimization:
  Time-structured algorithms and applications, \emph{Proceedings of the IEEE}
  108 (2020),  2032--2048.

\bibitem{simonetto2017prediction}
A{.\nobreak\kern 0.33333em}Simonetto and E{.\nobreak\kern
  0.33333em}Dall’Anese, Prediction-correction algorithms for time-varying
  constrained optimization, \emph{IEEE Transactions on Signal Processing} 65
  (2017),  5481--5494.

\bibitem{simonetto2016class}
A{.\nobreak\kern 0.33333em}Simonetto, A{.\nobreak\kern 0.33333em}Mokhtari,
  A{.\nobreak\kern 0.33333em}Koppel, G{.\nobreak\kern 0.33333em}Leus, and
  A{.\nobreak\kern 0.33333em}Ribeiro, A class of prediction-correction methods
  for time-varying convex optimization, \emph{IEEE Transactions on Signal
  Processing} 64 (2016),  4576--4591.

\bibitem{tang2022running}
Y{.\nobreak\kern 0.33333em}Tang, E{.\nobreak\kern 0.33333em}Dall'Anese,
  A{.\nobreak\kern 0.33333em}Bernstein, and S{.\nobreak\kern 0.33333em}Low,
  Running primal-dual gradient method for time-varying nonconvex problems,
  \emph{SIAM Journal on Control And Optimization} 60 (2022),  1970--1990.

\bibitem{tico2009digital}
M{.\nobreak\kern 0.33333em}Tico, Digital image stabilization, \emph{Recent
  Advances in Signal Processing}  (2009).

\bibitem{tuomov-proxtest}
T{.\nobreak\kern 0.33333em}Valkonen, Testing and non-linear preconditioning of
  the proximal point method, \emph{Applied Mathematics and Optimization} 82
  (2020),  591--636,
  \href{https://dx.doi.org/10.1007/s00245-018-9541-6}{\nolinkurl{doi:10.1007/s00245-018-9541-6}}.

\bibitem{tuomov-firstorder}
T{.\nobreak\kern 0.33333em}Valkonen, First-order primal-dual methods for
  nonsmooth nonconvex optimisation, in \emph{Handbook of Mathematical Models
  and Algorithms in Computer Vision and Imaging}, K{.\nobreak\kern
  0.33333em}Chen, C.\,B{.\nobreak\kern 0.33333em}Schönlieb,
  X.\,C{.\nobreak\kern 0.33333em}Tai, and L{.\nobreak\kern 0.33333em}Younes
  (eds.), Springer, Cham, 2021,
  \href{https://dx.doi.org/10.1007/978-3-030-03009-4_93-1}{\nolinkurl{doi:10.1007/978-3-030-03009-4_93-1}},
  \href{https://arxiv.org/abs/1910.00115}{\nolinkurl{arXiv:1910.00115}}.

\bibitem{tuomov-predict}
T{.\nobreak\kern 0.33333em}Valkonen, Predictive online optimisation with
  applications to optical flow, \emph{Journal of Mathematical Imaging and
  Vision} 63 (2021),  329--355,
  \href{https://dx.doi.org/10.1007/s10851-020-01000-4}{\nolinkurl{doi:10.1007/s10851-020-01000-4}},
  \href{https://arxiv.org/abs/2002.03053}{\nolinkurl{arXiv:2002.03053}}.

\bibitem{predict-code}
T{.\nobreak\kern 0.33333em}Valkonen, N{.\nobreak\kern 0.33333em}Dizon, and
  J{.\nobreak\kern 0.33333em}Jauhiainen, Predictive online optimisation codes
  for dynamic inverse imaging problems, Software on Zenodo, 2024,
  \href{https://dx.doi.org/10.5281/zenodo.12667014}{\nolinkurl{doi:10.5281/zenodo.12667014}}.

\bibitem{zhang2023regrets}
L{.\nobreak\kern 0.33333em}Zhang, H{.\nobreak\kern 0.33333em}Liu, and
  X{.\nobreak\kern 0.33333em}Xiao, Regrets of proximal method of multipliers
  for online non-convex optimization with long term constraints, \emph{Journal
  of Global Optimization} 85 (2023),  61--80.

\bibitem{zhang2021online}
Y{.\nobreak\kern 0.33333em}Zhang, E{.\nobreak\kern 0.33333em}Dall’Anese, and
  M{.\nobreak\kern 0.33333em}Hong, Online proximal-ADMM for time-varying
  constrained convex optimization, \emph{IEEE Transactions on Signal and
  Information Processing over Networks} 7 (2021),  144--155.

\bibitem{zhang2019distributed}
Y{.\nobreak\kern 0.33333em}Zhang, R.\,J{.\nobreak\kern 0.33333em}Ravier,
  V{.\nobreak\kern 0.33333em}Tarokh, and M.\,M{.\nobreak\kern
  0.33333em}Zavlanos, Distributed Online Convex Optimization with Improved
  Dynamic Regret, 2019,
  \href{https://arxiv.org/abs/1911.05127}{\nolinkurl{arXiv:1911.05127}}.

\bibitem{zhou2016image}
J{.\nobreak\kern 0.33333em}Zhou, P{.\nobreak\kern 0.33333em}Hubel,
  M{.\nobreak\kern 0.33333em}Tico, A.\,N{.\nobreak\kern 0.33333em}Schulze, and
  R{.\nobreak\kern 0.33333em}Toft, Image registration methods for still image
  stabilization, 2016.
\newblock US Patent 9,384,552.

\end{thebibliography}
\end{document}